%% file: main.tex
\RequirePackage[l2tabu, orthodox]{nag}
\documentclass{article}

\input{preamble}

\title{\textbf{Semiparametric semi-supervised learning for general targets under distribution shift and decaying overlap}}
\author[1,2]{Lorenzo Testa}
\author[1]{Qi Xu}
\author[1]{Jing Lei}
\author[1,3]{Kathryn Roeder}

\affil[1]{Department of Statistics \& Data Science, Carnegie Mellon University}
\affil[2]{L'EMbeDS, Sant'Anna School of Advanced Studies}
\affil[3]{Department of Computational Biology, Carnegie Mellon University}
\date{\vspace{-0.1em}\texttt{\{ltesta,qixu,jinglei,roeder\}@andrew.cmu.edu} \\ \vspace{1em}\today}

\begin{document}

\maketitle

\begin{abstract}
    In modern scientific applications, large volumes of covariate data are readily available, while outcome labels are costly, sparse, and often subject to distribution shift. This asymmetry has spurred interest in semi-supervised (SS) learning, but most existing approaches rely on strong assumptions -- such as missing completely at random (MCAR) labeling or strict positivity -- that put substantial limitations on their practical usefulness. In this work, we introduce a general semiparametric framework for estimation, inference, and efficiency benchmarking in SS settings where labels are missing at random (MAR) and the overlap may vanish as sample size increases. Our framework, that we label \ds{}, accommodates a wide range of smooth statistical targets -- including means, linear regression coefficients, quantiles, and causal effects -- and remains valid under high-dimensional nuisance estimation and distributional shift between labeled and unlabeled samples. We extend the theoretical guarantees of augmented inverse probability weighting estimators to preserve double robustness, asymptotic normality, and semiparametric efficiency under this challenging \ds{} regime. A key insight is that classical root-$n$ convergence fails under vanishing overlap; we instead provide corrected asymptotic rates that capture the impact of the decay in overlap. We validate our theory through simulations and demonstrate practical utility in real-world applications on the internet of things and public health where labeled data are scarce.
\end{abstract}

\section{Introduction}

In contemporary scientific applications, as highlighted by the success of \textit{prediction-powered inference} \citep{angelopoulos2023prediction}, semi-supervised (SS) learning has emerged as a powerful and increasingly essential paradigm, particularly in contexts where labeled data are expensive, difficult, or ethically challenging to obtain. These situations are ubiquitous in scientific domains such as biomedical research (e.g., large biobanks with partial annotations), public policy analysis (e.g., outcomes available only for selected sub-populations) and digital platforms (e.g., engagement or intervention data available only for active users). In these settings, researchers typically have access to a large volume of unlabeled covariate information, whereas labeled observations -- i.e., data points with observed outcomes of interest -- remain sparse.

A fundamental and often overlooked consequence of this asymmetry is the inherent \textit{violation} of the \textit{positivity} or \textit{overlap} condition. As the amount of unlabeled data grows much faster than the number of labeled samples, it becomes increasingly likely that certain covariate regions are labeled with vanishing probability. Formally, the labeling propensity, defined as the conditional probability of observing the outcome given the covariate, may approach zero on non-negligible portions of the covariate space, violating the standard assumption that the labeling propensity is bounded away from zero almost surely. This phenomenon, known as \textit{overlap decay}, is not merely a technicality, but a structural feature of semi-supervised learning. It affects estimation and inference, and renders many traditional estimators unstable. Crucially, this distinguishes the semi-supervised setting from standard missing data problems, where the labeling mechanism is always assumed to satisfy overlap conditions \citep{tsiatis2006semiparametric}. In SS learning, by contrast, decaying overlap is the norm, and asymptotic theory must be adapted accordingly.

Moreover, much of the existing work in semi-supervised learning \citep{angelopoulos2023prediction, xu2025unified} has been developed under the assumption that labels are \textit{missing completely at random} (MCAR), where the probability of observing a label is independent of both covariates and outcomes. Although analytically convenient, this assumption is rarely plausible in practice, especially in observational studies. In real-world applications where randomized design cannot be implemented, the labeling process is almost always influenced by observable characteristics, such as patient demographics, policy priorities, or platform usage patterns, making the MCAR assumption often inappropriate. A more realistic framework is the \textit{missing at random} (MAR) setting, where the labeling probability depends on the covariates, but is conditionally independent of the unobserved outcome. The MAR assumption, also known as \textit{ignorability} condition or \textit{selection bias}, allows systematic \textit{distribution shift} patterns in the labeling mechanism, while still permitting identifiability of the outcome distribution from observed data.

Under this \textit{decaying distribution shift semi-supervised} (\ds{}) framework, naive applications of existing semi-supervised or missing data methods can yield biased or inefficient estimates, or even fail to be well defined. In this paper, we propose a general semiparametric framework for efficient estimation and inference in semi-supervised settings under distribution shift and decaying overlap. Our framework covers a wide range of statistical inference targets -- including means, causal effects, quantiles, and smooth functionals of the data-generating process -- and offers principled strategies for estimation even when the probability of observing a label vanishes in large regions of the covariate space. Crucially, our theory allows for distributional shift between the labeled and unlabeled samples, meaning that the covariate distributions can be different, so long as appropriate reweighting can be performed with accurate nuisance parameter estimation.

We build on semiparametric efficiency theory to derive influence functions for \textit{general} targets in the \ds{} setting, and characterize augmented inverse probability weighting (\texttt{AIPW}) estimators that are doubly robust, as they remain consistent and asymptotically normal even if one of the two nuisance components (labeling or projection) is misspecified. Moreover, we show that asymptotic normality and semiparametric efficiency remain attainable -- though at a slower rate -- provided that the decay in labeling probability is appropriately controlled. To our knowledge, this is the first investigation of the theoretical guarantees of augmented inverse probability weighting estimators for general parameters of interest under the \ds{} setting. A detailed comparison of our results with existing literature is provided in Section \ref{subsec:1.1-related-work}.


\subsection{Related work}\label{subsec:1.1-related-work}
Our work, lying at the intersection of several branches of scholarship, combines the virtues of multiple literatures into a unified inferential framework for semi-supervised learning under distribution shift and decaying overlap.

\textbf{``Pure'' semi-supervised learning.} SS learning has traditionally focused on settings in which the labeled data are assumed to be missing completely at random (MCAR) -- see \citet{zhu2005semi, chapelle2009semi} for a review. Most existing work in this area focuses on estimation methods tailored to specific problems, such as mean estimation \citep{zhang2019semi, zhang2022high}, quantile estimation \citep{chakrabortty2022semi}, and linear regression \citep{azriel2022semi, tony2020semisupervised}. Notable exceptions are \citet{song2024general}, which studies general M-estimation problems, and \citet{xu2025unified}, which encompasses all parameters satisfying some form of functional smoothness. However, the MCAR assumption is increasingly recognized as unrealistic in many scientific applications, limiting the practical utility of such techniques in the presence of distribution shift.

\textbf{Prediction-powered inference.} A recent line of work addressing semi-supervised learning under MCAR is the \textit{prediction-powered inference} (\texttt{PPI}) framework. \texttt{PPI} integrates machine learning predictions into formal statistical procedures, using both labeled and unlabeled data to improve estimation efficiency while maintaining valid inference, resembling the approach proposed by \citet{chen2000unified}. Since the seminal work of \citet{angelopoulos2023prediction}, several extensions have expanded its scope. \citet{angelopoulos2023ppi++} and \citet{miao2023assumption} propose weighting strategies that ensure the \texttt{PPI} estimator performs at least as well as estimators relying solely on labeled data. \citet{zrnic2024cross} further develops the framework by introducing a cross-validation-based approach to learn the predictive model directly from the labeled sample. \citet{kluger2025prediction} generalizes \texttt{PPI} beyond M-estimation tasks and to MAR settings, assuming propensity scores are known by design. \citet{ji2025predictions} links \texttt{PPI} to surrogate outcome models in a MCAR setting.
See \citet{carlson2025unifying} for a summary of additional \texttt{PPI} extensions.

\textbf{\ds{} setting.} Classical results in semiparametric theory offer general principles for constructing doubly robust estimators in settings where outcomes are missing at random \citep{tsiatis2006semiparametric}. These estimators combine outcome projection models and propensity score models -- estimated either from the labeled sample alone or by incorporating information from both labeled and unlabeled data -- to yield consistent estimators of population-level parameters. However, most existing work assumes strong overlap. On the other hand, the literature on limited overlap and positivity violations has largely emerged from the causal inference community, where imbalance in treatment assignment and high-dimensional covariates pose analogous challenges \citep{crump2009dealing, khan2010irregular, rothe2017robust, yang2018asymptotic, d2021overlap}. Proposed solutions include trimming low-overlap regions, covariate balancing, and targeted reweighting schemes that stabilize estimation by controlling variance in sparse areas. Nevertheless, these methods typically assume that the propensity score can approach zero only in some specific regions of the support of the covariates $X$ and does not depend on the sample size. As such, they are ill-equipped to handle settings where the labeling probability vanishes uniformly across the covariate space as the total sample size grows. A notable exception is the recent work of \citet{zhang2023double}, who develop a doubly robust estimator for the outcome mean in a semi-supervised MAR setting with decaying overlap.

\textbf{Semiparametric statistics and missing data.} From a technical standpoint, our work lies in the furrow of semiparametric statistics, which focuses on estimation of finite dimensional functionals of the data generating process in the presence of high-dimensional nuisance functions. The missing data literature provides tools such as inverse probability weighting, augmented IPW, and targeted maximum likelihood estimation, many of which naturally extend to SS learning with MAR labeling. Foundational results on influence functions, tangent spaces, and semiparametric efficiency bounds -- developed by \citet{bickel1993efficient, tsiatis2006semiparametric, van2000asymptotic} -- serve as the theoretical backbone of modern SS learning inference. However, most of these frameworks assume strict positivity. As a result, applying semiparametric tools in semi-supervised learning with shrinking labeling fractions requires substantial modifications to asymptotic analysis. In particular, the decaying overlap forces data to be non i.i.d., requiring a significant extension of semiparametric efficiency theory to non i.i.d.~scenarios. We complement the few available works exploring this direction \citep{mcneney2000application, mcneney1998asymptotic} by adapting foundational results \citep{cam1960locally, le1972limits} to the \ds{} setting.

\subsection{Our contributions}

We propose a general framework for estimation and inference of arbitrary target parameters in the semi-supervised setting with missing at random labeling and decaying overlap. Our work builds on semiparametric theory, but departs significantly from classical setups by relaxing the positivity assumption and allowing the labeling probability to vanish uniformly as the sample size increases. Our main contributions are as follows:
\begin{itemize}
\item \textbf{General-purpose inference under decaying overlap.} We develop a flexible and model-agnostic \ds{} framework to accomodate estimation and inference of general target parameters under MAR labeling and decaying overlap. Our approach accommodates arbitrary functionals -- such as means, causal effects, or linear regression coefficients -- and handles distribution shifts without relying on parametric or Donsker-type assumptions about the projection functions or the missingness mechanism.
\item \textbf{Nonstandard asymptotics, double robustness, semiparametric efficiency.} We extend the theoretical guarantees of augmented inverse probability weighting estimators to the \ds{} framework, showing that they preserve double robustness and asymptotic normality. Importantly, the central limit theorem we prove exhibits a non-standard dependence on the sample size that captures the effect of overlap decay. This opens up the question of semiparametric efficiency in the \ds{} regime, in contrast to classical settings where the propensity score is fixed, i.e.~does not decay as the sample size increases. We characterize the semiparametric efficiency bound in this challenging non i.i.d.~setting, proving a convolution theorem in the spirit of Le Cam.
\item \textbf{Prediction-powered inference under MAR.} As a side product of our theoretical and methodological development, we provide a natural extension of \texttt{PPI} to missing at random settings, greatly increasing its applicability in real-world scenarios. 
\end{itemize}
Table~\ref{tab:contributions} summarizes our contribution in relation to the existing literature, highlighting how our framework integrates MAR labeling, accommodates overlap decay, supports general statistical targets, and is backed up by semiparametric efficiency theory.

\begin{table}
    \centering
    \begin{tabular}{lcccc}
    \toprule
    \textbf{Scholarship} & \textbf{MAR} & \textbf{Decaying overlap} & \textbf{General target} & \textbf{Efficiency theory} \\
    \midrule
    \citet{tsiatis2006semiparametric} & \cmark & \xmark & \cmark & \cmark \\
    \citet{zhang2023double} & \cmark & \cmark & \xmark & \xmark \\
    \texttt{PPI} \citep{angelopoulos2023prediction} & \xmark & \xmark & \cmark & \xmark \\
    \citet{xu2025unified} & \xmark & \xmark & \cmark & \cmark \\
    \ds{} (our proposal) & \cmark & \cmark & \cmark & \cmark \\
    \bottomrule
    \end{tabular}
    \caption{Framing our contribution}
    \label{tab:contributions}
\end{table}

\subsection{Organization}
This paper is organized as follows. Section~\ref{sec:setup} introduces the semi-supervised learning framework under MAR labeling with decaying overlap (\ds{}), and formally defines the statistical estimation problem. In Section~\ref{sec:ext_models}, we present a general strategy for building augmented inverse probability weighting estimators and characterize their asymptotic behavior under \ds{}. Following \citet{tsiatis2006semiparametric, zhang2023double}, we begin by analyzing the case where the propensity score is known by design, and then extend our results to the more realistic setting where the propensity score must be estimated from the data. Here, we also extend the \textit{prediction-powered inference} framework proposed by \citet{angelopoulos2023prediction} to the \ds{} setting
. We also characterize the notion of semiparametric efficiency in the \ds{} setting, extending semiparametric efficiency theory to accommodate observations generated by a triangular array. This culminates in a convolution theorem characterizing the semiparametric efficiency lower bound. Section~\ref{sec:sim} applies our general theory to specific statistical functionals and assesses the empirical performance of \texttt{AIPW} in the \ds{} setting through simulation studies. In Section~\ref{sec:app}, we illustrate the practical value of our methodology using data from the \texttt{BLE-RSSI} study \citep{mohammadi2017semisupervised}, where we estimate the spatial position of a device based on signal strength measurements captured by remote sensors, and the \texttt{NHEFS} study \citep{hernan2010causal, ertefaie2023nonparametric}, where we estimate the change in weight and smoking intensity in the population using only labeled data from alcohol abstainers. Concluding remarks and potential directions for future research are discussed in Section~\ref{sec:end}. Additional theoretical results, simulation details, and a further application to the \texttt{METABRIC} study \citep{curtis2012genomic,pereira2016somatic}, where we revisit a question posed by \citet{hsu2023learning} and estimate the association between previously identified biomarkers \citep{cheng2021integrating} and survival outcomes in patients with breast cancer, are provided in the Supplementary Material. All code for reproducing our analysis is available at \url{https://github.com/testalorenzo/decMAR_inf}.

\section{Problem setup}
\label{sec:setup}

We assume that we observe a \textit{small} collection of $n$ independent and identically distributed \textit{labeled} samples $\left\{(X_i, Y_i)\right\}_{i=1}^n$
, where $X_i\in\RR^p$ is a $p$-dimensional vector of covariates and $Y_i\in\RR^k$ is a $k$-dimensional outcome. We let $(X,Y)$ denote an independent copy of $(X_i,Y_i)$
. We also assume to have access to a further \textit{large} collection of $N$ \textit{unlabeled} data samples $\left\{X_i \right\}_{i=1}^N$. As before, we let $X$ denote an independent copy of $X_i$. Following traditional nomenclature from semiparametric statistics literature, we recast our problem in a missing data framework. We denote the \textit{observed data} as $\left\{ \data_i =(X_i,R_{i},R_{i}Y_i) \right\}_{i=1}^{n+N}$, where $R_{i}$ is a binary indicator that indicates whether observation $\data_i$ is labeled 
($R_{i}=1$), or unlabeled ($R_{i}=0$). We let $\data = (X,R,RY)$ denote an independent copy of $\data_i = (X_i,R_{i},R_{i}Y_i)$. We denote the data-generating distribution as $\jPtarget\in\mathcal{P}$, where $\mathcal{P}$ is the set of distributions induced by a nonparametric model, and thus we write $\data\sim\jPtarget$. For theoretical convenience, we also define the \textit{full data} $\data^F = (X,Y)$, that is, the data that we would observe if there were no missingness mechanisms in place. 

The potential discrepancy between the labeled and unlabeled datasets due to \textit{distribution shift} naturally leads us to adopt a \textit{missing at random} (MAR) labeling mechanism. Moreover, a distinctive aspect of our setting is that the proportion of labeled data $n$ may become asymptotically negligible compared to the volume of unlabeled data $N$. Specifically, we consider the extreme regime in which $\lim_{n,N\to\infty} n / (n+N) = 0$. This scenario effectively violates the standard assumption of \textit{positivity}, typically required in classical missing data analyses, where the ratio $n / (n+N)$ is always bounded away from zero. By defining the \textit{propensity score} as $\pitarget(x) = \PP{R=1\mid X=x}$, we can formally state these assumptions below.

\begin{assumption}[\ds{} setting]
\label{ass:decMAR}
    Let the following assumptions hold:
    \begin{enumerate}[label=\textbf{\alph*.}]
    \item \textbf{Missing at random.} $R \indep Y \mid X$. 
    \item \textbf{Decaying overlap.} Given $\an^{-1} = \EE{{\pitarget}^{-1}(X)}<\infty$ for every $n$ and $N$, it holds that $\an>0$, with potentially $\an\to0$ as $n,N\to\infty$.
\end{enumerate}
    \end{assumption}

\begin{remark}
    Assumption \textbf{a} corresponds to the well-known \textit{ignorability} condition, commonly invoked in both the missing data and causal inference literature \citep{tsiatis2006semiparametric, kennedy2024semiparametric}. Unlike classical semi-supervised settings that assume labels are missing completely at random (MCAR), our setup allows for selection bias, or distribution shift, in the labeled data. As a result, population-level parameters cannot be consistently estimated using the labeled sample alone. The goal of our work is not merely to enhance supervised estimators using additional unlabeled data (which are no longer unbiased under MAR), but rather to construct a theoretically grounded estimator of general functionals of the full data-generating process. Assumption \textbf{b}, inspired by the asymptotic framework of \citet{zhang2023double}, formalizes the regime under which our theoretical guarantees are derived. In contrast to much of the existing literature, which assumes strong overlap (i.e., a uniformly bounded away from zero labeling propensity), we allow for overlap to decay with the sample size, requiring new tools to characterize convergence rates in this more general setting.
\end{remark}

\begin{remark}[Triangular arrays] 
\label{rem:triangular}
The \textit{propensity score} $\pitarget(x) = \PP{R=1\mid X=x}$ plays a central role in our theoretical development under decaying overlap. In particular, Assumption~\textbf{b} imposes structural constraints on the behavior of $\pitarget$, by allowing it to depend on the sample sizes $n$ and $N$. A canonical example of the regime of decaying overlap captured by Assumption~\ref{ass:decMAR} is the MCAR case, in which the propensity score reduces to $\pitarget(x) = n/(n+N)$ for all $x\in\RR^p$, implying that $\an = n / (n+N)$ in this special case. Hence, both $R_{i}$ and $\pitarget(X_i)$ form \textit{triangular arrays} over $n,N$ and $i$, but we suppress the dependence of $R_i$ on $n,N$ for notational simplicity. Notice that we are also implicitly dropping the dependence on $n,N$ even from $\jPtarget$, which, under Assumption~\ref{ass:decMAR}, instead depends on the sample sizes as $\jPtarget = \mPtarget \cPRtarget \cPtarget$ and $\cPRtarget$ changes with the sample size. This notational subtlety will be innocuous until Section~\ref{sec:efficiency}, and we defer comments to its implications there. Finally, Supplementary Section~\ref{supp_sec:dec_models} provides some examples of statistical models where the decaying construction arises naturally.
\end{remark}

Throughout this paper, we focus on the identification and estimation of a target quantity $\thetatarget\in\RR^q$ that is defined as the Euclidean functional solving $\thetatarget = \theta(\jPtarget)$, with $\theta\,:\,\mathcal{P} \to \RR^q$.
We restrict our attention to target parameters that could be estimated using \textit{regular} and \textit{asymptotically linear} (RAL) \textit{full-data} estimators, that is, targets that admit the expansion
\begin{equation}
\label{eq:ral}
    \sqrt{n+N} \left(\thetaf - \thetatarget \right) = (n+N)^{-1/2} \sum_{i=1}^{n+N} \finfluence{\data_i^F;\thetatarget} + o_\mathbbmss{P}(1)\,,
\end{equation}
for some regular and asymptotically linear \textit{full-data} estimator $\thetaf$, and some function $\finfluence{\data_i^F;\thetatarget}$, referred to as \textit{full-data influence function}, evaluating the contribution of the \textit{full-data} observation $\data_i^F$ to the overall estimator $\thetaf$ \citep{hampel1974influence}. The influence function $\finfluence{\data_i^F;\thetatarget}$ depends on the true parameter $\thetatarget$ and is such that $\EE{\finfluence{\data^F;\thetatarget}} = 0$ and $\VV{\finfluence{\data^F;\thetatarget}}$ is positive definite.
This framework is very general and encompasses M-estimation, Z-estimation, U-statistics, and many estimands in causal inference. Examples of such targets are the outcome mean, the outcome quantiles, least squares coefficients, average treatment effects, and general minimizers of convex loss functions depending on both $X$ and $Y$ \citep{tsiatis2006semiparametric, van2000asymptotic}.

There is a deep connection between regular and asymptotically linear estimators and influence functions -- see Theorem 3.1 in \citet{tsiatis2006semiparametric}. In fact, given a full-data influence function $\finfluence{\data^F;\thetatarget}$, one can recover the associated estimator $\hat{\theta}^F$ by solving the estimating equation
\begin{equation}
\label{eq:full_est_eq}
    \sum_{i=1}^{n+N} \finfluence{\data^F_i;\thetatarget} = 0\,.
\end{equation}
Standard semiparametric theory guarantees that, under mild conditions, the classical estimator $\thetaf$ characterized in Equation~\ref{eq:ral} and Equation~\ref{eq:full_est_eq} is consistent for $\thetatarget$ and asymptotically normal. This follows from well-known properties of influence functions and an application of the Central Limit Theorem. The asymptotic variance of the estimator $\thetaf$ is given by the variance of its influence function divided by the sample size, i.e.~$\VV{\finfluence{\data^F;\thetatarget}} / (n+N)$. Notice that if we do not make assumptions on the distribution of the data-generating process, we are left with a problem that is inherently nonparametric. In this case, semiparametric theory shows that each functional is associated with one and only one full-data influence function, and thus it is also the efficient one (i.e. the one delivering the smallest asymptotic variance). 

While convenient from a theoretical standpoint, the full-data estimator $\thetaf$ is of little practical relevance, as it relies on outcome labels that are missing in the unlabeled portion of the data. A natural first approximation is to construct an estimator using only the labeled sample. However, under the missing at random (MAR) assumption, such an estimator is generally biased due to distribution shifts in the labeling mechanism. Our objective is therefore to understand \textit{how} to construct an unbiased and asymptotically normal estimator for the target parameter $\thetatarget$, using the observed data $\data=(X,R,RY)$ in the presence of decaying overlap. We then develop a general and flexible strategy for constructing estimators that remain valid under this challenging regime.

\subsection{Notation}

Before proceeding, we introduce some more notation. We denote \textit{weak convergence} of a random object $Z_n$ to a limit $Z$ as $Z_n\dto Z$. Similarly, we denote \textit{convergence in $\mathbbmss{L}^p$ norm} as $Z_n\Lto{p}Z$. Given a random object $Z_n$, we denote its variance as $\VV{Z_n}$. If $Z_n$ is $q$-dimensional, then $\VV{Z_n}$ must be understood as a $q \times q$ positive definite variance-covariance matrix. We also define the $\mathbbmss{L}_p$ norm of a $q$-dimensional vector $Z$ as $\norm{Z}{p} = \left(\sum_{j=1}^q Z_j^p\right)^{1/p}$. Given a $q\times q$ matrix $A$, $\mineig{A}$ and $\maxeig{A}$ indicate its smallest and largest eigenvalues, respectively. The $q\times q$ identity matrix is denoted as $I_q$. $0_q$ denotes a vector of size $q$ whose entries are all zeros. For any two sequences $a_n$ and $b_n$, we write $a_n \asymp b_n$ if there exist constants $c,C,n_0$ such that $cb_n < a_n < Cb_n$ for all $n>n_0$.

\section{Learning in the \ds{} setting}
\label{sec:ext_models}

\subsection{Known propensity score model}
In this Section, we develop results for estimation and inference in the \ds{} setting under the assumption that the propensity score $\pitarget$ is known. This setup is particularly relevant in experimental designs where the labeling mechanism is controlled and fully specified, but resources for treatment administration are scarce.

When outcome labels are missing at random, the full-data estimator defined in Equation~\ref{eq:ral} is not directly computable, as it relies on unobserved outcomes in the unlabeled sample. To properly account for the MAR mechanism, we derive the \textit{observed-data influence function}, which is the projection of the full-data influence function onto the observed data model \citep{tsiatis2006semiparametric}. Under the \ds{} setting in Assumption~\ref{ass:decMAR}, the observed-data influence function for the target $\thetatarget$ is given by:
\begin{equation}
    \label{eq:if}
    \begin{split}
        \influence{\data;\thetatarget} &= \frac{R}{\pitarget(X)} \finfluence{\data;\thetatarget} - \frac{R-\pitarget(X)}{\pitarget(X)} \EE{\finfluence{\data;\thetatarget}\mid X} \\
        &= \EE{\finfluence{\data;\thetatarget}\mid X} + \frac{R}{\pitarget(X)} \left( \finfluence{\data;\thetatarget} - \EE{\finfluence{\data;\thetatarget}\mid X} \right)\,,
    \end{split}
\end{equation}
where $\finfluence{\data^F;\thetatarget}$ is any valid full-data influence function, and $\pitarget : \RR^p \to (0,1)$ denotes the propensity score, defined as $\pitarget(x) = \PP{R = 1 \mid X = x}$. The derivation of this result is detailed in Supplementary Section~\ref{supp_sec:proofs}.

In Equation~\ref{eq:if}, the only unknown component is the \textit{nuisance function} $\mutarget(X) = \EE{\finfluence{\data;\thetatarget}\mid X}$, which we refer to as the \textit{projection model}. In fact, for the moment, we assume that the propensity score is known, and this amounts to assume that $\hat\pi=\pitarget$. To make the dependence on nuisance functions explicit, we denote the influence function evaluated at $\hat{\mu}$ and $\hat\pi$ as $\influence{\data;\thetatarget;\hat{\mu};\hat\pi}$. In particular, the observed-data influence function where $\hat{\mu} = \mutarget$ and $\hat{\pi} = \pitarget$ is denoted as $\influence{\data;\thetatarget;\mutarget;\pitarget}$. 

Since $\mutarget$ must be estimated from the data, we employ \textit{cross-fitting} to avoid restrictive Donsker conditions and to retain full-sample efficiency \citep{bickel1988estimating, chernozhukov2018double, robins2008higher, schick1986asymptotically,zheng2010asymptotic}. Cross-fitting works as follows. We first randomly split the observations $\{\data_1,\dots,\data_{n+N}\}$ into $J$ disjoint folds (without loss of generality, we assume that the number of observations $n+N$ is divisible by $J$). For each $j=1,\ldots, J$ we form $\hat{\mathbbmss{P}}^{[-j]}$ with all but the $j$-th fold, and we denote the set of indices of the observations belonging to the $j$-th fold as $\mathcal{K}^{[j]}\subseteq \{1,\dots,n+N\}$. Then, we learn $\hat{\mu}^{[-j]}$ on $\hat{\mathbbmss{P}}^{[-j]}$. For most supervised methods, learning $\hat{\mu}^{[-j]}$ on $\hat{\mathbbmss{P}}^{[-j]}$ will involve using only labeled observations within $\hat{\mathbbmss{P}}^{[-j]}$. Then, we compute the estimator $\hat{\theta}_{\hat\mu}$ by solving the estimating equation
\begin{equation}
\label{eq:est_eq_piknown}
    \sum_{j=1}^J \sum_{i\in\mathcal{K}^{[j]}} \influence{\data_i;\thetatarget;\hat{\mu}^{[-j]};\pitarget} = 0\,.
\end{equation}
Estimating the true projection function $\mutarget$ can be challenging in practice, especially when the number of labeled samples is limited. However, due to \textit{double robustness}, the estimator derived from Equation~\ref{eq:est_eq_piknown} remains consistent and asymptotically normal even when $\hat\mu$ is misspecified, as long as it converges to some limit $\Bar{\mu}\neq\mutarget$ (in a sense we will make precise shortly). The cost of such misspecification is an increase in asymptotic variance relative to the case where $\hat\mu = \mutarget$. 

We can now state some regularity conditions needed to show the theoretical properties of our estimators.
\begin{assumption}[Inference]
\label{ass:reg}
Let the number of cross-fitting folds be fixed at $J$, and assume that: 
\begin{enumerate}[label=\textbf{\alph*.}]
    \item There exist $\Bar{\mu}$ and $\pibar$ such that, for each $j\in\{1,\dots,J\}$, one has \begin{equation}
        \influence{\data;\thetatarget;\hat{\mu}^{[-j]};\hat{\pi}^{[-j]}}\Lto{2} \influence{\data;\thetatarget;\Bar{\mu};\pibar}\,.
    \end{equation}
    \item For each $j\in\{1,\dots,J\}$, one has \begin{equation}
        (n+N)^{1/2}\an^{1/2} \sum_{j=1}^J \norm{\eta^{[j]}}{2} = o_\mathbbmss{P}(1)\,,
    \end{equation} 
    where the \textit{remainder term} $\eta^{[j]}$ is defined as $\eta^{[j]} = \theta\left(\hat{\mathbbmss{P}}^{[-j]}\right) - \theta(\jPtarget) + \EE{\influence{\data;\thetatarget;\hat{\mu}^{[-j]};\hat{\pi}^{[-j]}}}$.
    \item $\mineig{\VV{\influence{\data;\thetatarget;\Bar{\mu};\pibar}}} \asymp \an^{-1}$ and $\maxeig{\VV{\influence{\data;\thetatarget;\Bar{\mu};\pibar}}} \asymp \an^{-1}$.
    \item For every $\varepsilon>0$, it holds that
    \begin{equation}
    \EE{\an\norm{\influence{\data;\thetatarget;\Bar{\mu};\pibar}}{2}^2 \onea{\norm{\influence{\data;\thetatarget;\Bar{\mu};\pibar}}{2} > \varepsilon \sqrt{\frac{n+N}{\an}}}} \to 0\quad\text{as}\,n,N\to\infty\,.
    \end{equation}
\end{enumerate}
\end{assumption}
\begin{remark}
    Fixing the number of cross-fitting folds prevents undesirable asymptotic behavior and is standard in modern semiparametric literature. Assumptions~\textbf{a} and \textbf{b} are also conventional; see, for example, \citet{kennedy2024semiparametric}. Assumption~\textbf{a} is used to control the empirical process fluctuations, while Assumption~\textbf{b} ensures that the remainder term in the \textit{von Mises expansion} is negligible. Assumptions \textbf{c} and \textbf{d} restrict the class of admissible targets within our framework, while still allowing for a broad and general set of results. Specifically, Assumption~\textbf{c} imposes a condition on the scaling of the asymptotic variance of the influence function, whereas Assumption~\textbf{d} introduces a technical requirement on the rate at which the observed-data influence function decays with the sample sizes $n$ and $N$. This latter condition is motivated by the classical Lindeberg condition, which is commonly assumed in proving Central Limit Theorems (CLTs), and is in fact close to being necessary for asymptotic normality. Our formulation generalizes the analogous condition used by \citet{zhang2023double} to accommodate a wider class of targets. In particular, we recover Assumption 3.2 of \citet{zhang2023double} by noticing that, for a univariate target (and thus a unidimensional influence function), the norm in Assumption~\textbf{d} reduces to the absolute value of the influence function itself.
\end{remark}

\begin{remark}[Lindeberg condition]
    At first glance, Assumption~\textbf{d} may appear benign: as $n, N \to \infty$, the term $\an \to 0$, which might suggest that the Lindeberg-type condition becomes increasingly easy to satisfy. However, this is not the case. A key subtlety lies in the fact that the observed-data influence function $\influence{\data;\thetatarget}$ is itself a function of $n$ and $N$, particularly through its dependence on the labeling mechanism and the propensity scores. This introduces a non-trivial coupling between the behavior of the influence function and the rate $\an$. As $\an \to 0$, the weights applied to individual observations in the influence function increase, especially in regions of the covariate space where the propensity score is small. Consequently, while the factor $\an$ in the condition vanishes, the tails of the influence function may become heavier due to extreme reweighting. This interplay can make the Lindeberg condition more delicate -- not less. In this sense, the condition acts as a technical safeguard that balances asymptotic variance control with tail robustness under decaying overlap.
\end{remark}

We are now ready to provide our first result, characterizing the asymptotic behavior of our estimators with known propensity score under the \ds{} setting.

\begin{theorem}[Consistency and asymptotic normality, known $\pitarget$]
\label{th:ansi_knownpi}
Assume $(n+N)\an\to\infty$ and that $\pitarget$ is known. Under Assumption~\ref{ass:decMAR} (\ds{} setting) and Assumptions~\ref{ass:reg} (inference) \textbf{a}, \textbf{c}, and \textbf{d} , one has
\begin{equation}
    (n+N)^{1/2} \an^{1/2} \norm{\hat\theta_{\hat\mu} - \thetatarget}{2} = O_\mathbbmss{P}(1)\,,
\end{equation}
and
\begin{equation}
    (n+N)^{1/2} \VV{\influence{\data;\thetatarget;\hat{\mu};\pitarget}}^{-1/2} \left(\hat\theta_{\hat\mu} - \thetatarget \right) \dto \Normal{0}{I_q}\,.
\end{equation}
\end{theorem}

\begin{remark}[Effective sample size]
\label{rem:effective_ss}
    The quantity $(n+N)\an$ can be interpreted as the \emph{effective sample size} that governs the convergence rate of the estimator $\hat\theta_{\hat\mu}$ to the target parameter $\thetatarget$. The factor $\an$ reflects the severity of overlap decay, by capturing both how concentrated the propensity score $\pitarget$ is around zero and the relative size of the labeled dataset $n$. The greater the concentration of $\pitarget$ near zero, the slower the rate of convergence. In the special case where labels are missing completely at random (MCAR), we have $\pitarget(x) = n / (n+N)$ for all $x \in \mathbb{R}^p$, which implies $\an = n / (n+N)$ and thus an effective sample size equal to $n$.  In this regime, the rate of convergence reduces to the standard root-$n$ rate typically achieved when using labeled data alone. Our findings align with asymptotic results established in \citet{angelopoulos2023ppi++} and \citet{xu2025unified}.
\end{remark}

\subsection{Unknown propensity score model}
In general observational studies, both the projection model $\mutarget$ and the propensity score $\pitarget$ are unknown and need to be estimated. We employ again \textit{cross-fitting}. We randomly split the observations $\{\data_1,\dots,\data_{n+N}\}$ into $J$ disjoint folds. For each $j=1,\ldots, J$ we form $\hat{\mathbbmss{P}}^{[-j]}$ with all but the $j$-th fold, and we denote the set of indices of the observations belonging to the $j$-th fold as $\mathcal{K}^{[j]}\subseteq \{1,\dots,n+N\}$. Then, we learn $\hat{\mu}^{[-j]}$ and $\hat{\pi}^{[-j]}$ on $\hat{\mathbbmss{P}}^{[-j]}$, and compute the estimator $\hat{\theta}_{\hat\mu;\hat\pi}$ by solving the estimating equation
\begin{equation}
\label{eq:est_eq_piunknown}
    \sum_{j=1}^J \sum_{i\in\mathcal{K}^{[j]}} \influence{\data_i;\thetatarget;\hat{\mu}^{[-j]};\hat{\pi}^{[-j]}} = 0\,.
\end{equation} 

We can now extend the results of Theorem~\ref{th:ansi_knownpi} to the more general case where all nuisance functions are estimated from the data.

\begin{theorem}[Consistency and asymptotic normality, unknown nuisances]
\label{th:ansi_unknownpi}
Assume $(n+N)\an\to\infty$. Under Assumption~\ref{ass:decMAR} (\ds{} setting) and Assumptions~\ref{ass:reg} (inference) \textbf{a}, \textbf{b}, \textbf{c}, and \textbf{d}, one has 
\begin{equation}
    (n+N)^{1/2} \an^{1/2} \norm{\hat\theta_{\hat\mu;\hat\pi} - \thetatarget}{2} = O_\mathbbmss{P}(1)\,,
\end{equation}
and
\begin{equation}
    (n+N)^{1/2} \VV{\influence{\data;\thetatarget;\hat{\mu};\hat\pi}}^{-1/2} \left(\hat\theta_{\hat\mu;\hat\pi} - \thetatarget \right) \dto \Normal{0}{I_q}\,.
\end{equation}
\end{theorem}

\begin{remark}[Asymptotic normality under misspecification]
When both the projection model and the propensity score are correctly specified, Theorem~\ref{th:ansi_unknownpi} ensures that $\hat\theta_{\hat\mu;\hat\pi}$ is asymptotically normal. The consistency of the estimator relies on the remainder term $\eta$ vanishing asymptotically, as required by Assumption~\ref{ass:reg}, part~\textbf{b}. Algebraically, this remainder is driven by the product of the errors of the nuisance functions, scaling with $\norm{\hat{\pi} - \pitarget}{2}\norm{\hat{\mu} - \mutarget}{2}$. Therefore, Assumption \textbf{b} implicitly requires that at least one of the two nuisance models is correctly specified (\textit{double robustness}). If both models are misspecified, this condition generally fails, and consistency is lost. Furthermore, if only one component is correctly specified, achieving asymptotic normality is more demanding, as the correctly specified nuisance function must converge at the faster rate of $(n+N)^{-1/2} \an^{-1/2}$, which is typically attainable \textit{only} when its population counterpart belongs to a low-dimensional, parametric model class. Therefore, in complex applications, the relaxation of Assumption \textbf{b} may inhibit asymptotic normality and preserve only consistency. 
\end{remark}

\subsection{Prediction-powered inference in the \ds{} setting}
In this Section, we extend the framework developed so far to incorporate additional information from an off-the-shelf predictive model $f$ that maps the covariate space to the outcome space. We do not impose assumptions on the quality of $f$, treating it as a generic black-box machine learning prediction model. This extension connects naturally to the \textit{prediction-powered inference} (\texttt{PPI}) framework introduced by \citet{angelopoulos2023prediction}, but significantly broadens its applicability in real-world scientific settings.

In particular, while the original \texttt{PPI} framework assumes that labels are missing completely at random (MCAR), our extension relaxes this to the more general missing at random (MAR) setting, where the labeling mechanism may depend on the covariates. Additionally, while the original \texttt{PPI} framework is limited to $M$-estimation problems, our approach supports any target functional satisfying Equation~\ref{eq:ral}.
As a result, our formulation enables the use of predictive models within \texttt{PPI} even under covariate-dependent labeling, allowing the benefits of modern machine learning to be realized in more realistic and challenging scenarios. Notably, standard \texttt{PPI} methods are not valid under MAR.

As in previous Sections, we recast the problem in the language of missing data. The observed data are now given by $\data = (X, f(X), R, RY)$, with $n + N$ realizations. In this context, the observed-data influence function for \texttt{PPI} takes the form:
\begin{equation}
\label{eq:ppiif}
   \ppiinfluence{\data;\thetatarget} = \frac{R}{\pitarget(X, f(X))} \finfluence{\data;\thetatarget} - \frac{R-\pitarget(X,f(X))}{\pitarget(X,f(X))} \EE{\finfluence{\data;\thetatarget}\mid X,f(X)}\,.
\end{equation}

All results presented in the previous Sections extend naturally to this setting. The key distinction is that both the projection model and the propensity score can now be conditional on the additional covariate $f(X)$, which is always observed and may improve the quality of both nuisance estimates. In particular, we can establish consistency and asymptotic normality for the resulting \texttt{PPI} estimator $\hat\theta^\texttt{PPI}$, obtained by solving the associated estimating equation involving the influence function in Equation~\ref{eq:ppiif}, under the \ds{} setting.

\begin{corollary}[\texttt{PPI} asymptotic normality under \ds{} setting] 
\label{cor:an_ppi} 
Assume $(n+N)\an\to\infty$. Under Assumption~\ref{ass:decMAR} (\ds{} setting) and Assumption~\ref{ass:reg} (inference), one has
\begin{equation}
    (n+N)^{1/2} \an^{1/2} \norm{\hat\theta^\texttt{PPI} - \thetatarget}{2} = O_\mathbbmss{P}(1)\,,
\end{equation}
and
\begin{equation}
    (n+N)^{1/2} \VV{\ppiinfluence{\data;\thetatarget;\hat{\mu};\hat\pi}}^{-1/2} \left(\hat\theta^\texttt{PPI} - \thetatarget \right) \dto \Normal{0}{I_q}\,.
\end{equation}
\end{corollary}

\begin{remark}
    The utility of incorporating an external predictive model $f$ into the estimation procedure -- as in the \texttt{PPI} framework -- depends critically on the information that $f(X)$ carries relative to the covariates $X$. In many practical applications, $f$ is a deterministic function, such as a black-box machine learning predictor trained externally. In this case, $f(X)$ is entirely determined by $X$, and the conditional expectations in Equation~\ref{eq:if} project the full-data influence function $\finfluence{\data^F;\thetatarget}$ onto the same $\sigma$-field as in the original setting. As a result, the asymptotic variance of the resulting estimator remains unchanged, and the inclusion of $f$ does not yield a theoretical efficiency gain -- though it can still offer great practical advantages by simplifying nuisance model specification.
    More interesting behavior arises when $f$ introduces additional randomness or knowledge beyond $X$. Many modern models (e.g., variational autoencoders, Bayesian methods, and transformers) produce a predictive distribution rather than a single deterministic point estimate; as a result, multiple valid predictions $f(X)$ may be associated with the same covariates $X$. In such cases, the $\sigma$-field onto which the conditional expectations project the full-data influence function may change, and the augmented variable $f(X)$ can reduce the residual variance by capturing information orthogonal to $X$, especially if it correlates with omitted features influencing the outcome.
\end{remark}

Finally, we note that our framework can naturally accommodate the integration of multiple predictive models; that is, we can simultaneously leverage a collection of off-the-shelf models $\{f_1, \dots, f_m\}$ within our estimation procedure, allowing us to flexibly combine information from diverse sources and enhance robustness in practical applications \citep{shan2025sada}.

\subsection{Semiparametric efficiency}
\label{sec:efficiency}
After having established double robustness and asymptotic normality of estimators in the \ds{} setting, the next -- and perhaps most interesting -- question is the study of their optimality. Semiparametric theory offers a rich toolkit to provide an answer. Importantly, our work seems among the first to investigate extensions of semiparametric optimality theory behind the standard i.i.d.~case. In this Section, we adapt Le Cam's theory of convergence of experiments to the \ds{} setting, providing a rigorous benchmark for optimality through the characterization of a convolution theorem for triangular arrays. While more technical than the rest of the paper, this Section sheds light on the optimality of \texttt{AIPW} estimators even when data are generated through a triangular array mechanism. We believe that the following results are of independent interest, and open the venue for more explorations in the realm of semiparametric theory for non-identically distributed or dependent data. 

As briefly introduced in Remark~\ref{rem:triangular}, $\jPtarget$ depends on the sample sizes $n,N$ as $\jPtarget = \mPtarget \cPRtarget \cPtarget$ and $\cPRtarget$ changes with $n,N$. Under Assumption~\ref{ass:decMAR}, the law $ \mPtarget \cPtarget$ is fixed across $n,N$, while the labeling law $\cPRtarget$ varies with $n,N$. To make this dependence more explicit, in this Section we write $\jPtarget_{n,N}$, and we denote its density with respect to some dominating measure $\nu_{n,N}$ as $p^\star_{n,N}$. Similarly, we track the dependence of influence functions on the sample sizes by writing $\ninfluence{\data;\thetatarget}$. The variance of the influence function evaluated at $\mutarget$ and $\pitarget$ is $\Sigma_{n,N} = \VV{\ninfluence{\data;\thetatarget;\mutarget;\pitarget}}$. We can now introduce the set of assumptions required to study efficiency in the \ds{} setting.

\begin{assumption}[Efficiency]
\label{ass:eff}
Assume: 
\begin{enumerate}[label=\textbf{\alph*.}]
    \item $\varstar = \lim_{n,N\to\infty} \an \Sigma_{n,N}$ exists and is positive definite.
    \item For each fixed $h\in\RR^q$ there exists a sequence of observed-data laws $\jPtarget_{n,N,h}$, all dominated by a common measure $\nu_{n,N}$ with densities $p^\star_{n,N,h}$, such that $\jPtarget_{n,N,0} = \jPtarget_{n,N}$ and the local perturbation
    \begin{equation}
        \ell_{n,N,h}(\data) = (n+N)^{-1/2} \an^{-1/2} h^T \Sigma_{n,N}^{-1} \ninfluence{\data;\thetatarget} 
    \end{equation}
    satisfies the following properties:
    \begin{itemize}
        \item Quadratic mean differentiability along $h$:
        \begin{equation}
            (n+N) \int \left(\sqrt{p^\star_{n,N,h}} - \sqrt{p^\star_{n,N}} - \frac{1}{2} \ell_{n,N,h}  \sqrt{p^\star_{n,N}} \right)^2 \,d\nu_{n,N} \to 0\,.
        \end{equation}
        \item Local shift of the target:
        \begin{equation}
            \sqrt{(n+N)\an} \left(\theta\left( \jPtarget_{n,N,h} \right) -\thetatarget \right) \to h\,.
        \end{equation}
    \end{itemize}
    \item An estimator sequence $\hat\theta_{n,N}$ is called regular at rate $\sqrt{(n+N)\an}$ if, for every fixed $h\in\RR^q$,
    \begin{equation}
        \sqrt{(n+N)\an} \left( \hat\theta_{n,N} - \theta\left( \jPtarget_{n,N,h} \right) \right) 
    \end{equation}
    converges in distribution under $\jPtarget_{n,N,h}$ to a limit law that does not depend on $h$.
\end{enumerate}
\end{assumption}

\begin{remark}
    Assumption \textbf{a} is required for the existence of a benchmark for optimality, as embedded in the existence of a properly rescaled asymptotic variance. Assumption \textbf{b} introduces the least-favorable triangular-array submodel. In particular, the first assumption controls the QMD remainder along the local sequence, while the second assumption controls the smoothness of the target along the same sequence. Assumption \textbf{c} is a standard extension of the notion of regularity of estimators to a non-standard rate. See Supplementary Section~\ref{supp_sec:eff_examples} for some sufficient verifiable conditions that satisfy Assumption~\ref{ass:eff}.
\end{remark}

We can now provide the central result of this Section, which extends the notion of semiparametric efficiency to the \ds{} setting through a convolution theorem. An immediate corollary implies that \texttt{AIPW} estimators are semiparametrically efficient under the \ds{} setting.

\begin{theorem}[Convolution]
\label{th:convolution}
    Assume $(n+N)\an\to\infty$. Moreover, let Assumptions~\ref{ass:decMAR}, \ref{ass:reg}, and \ref{ass:eff} hold. Then any estimator $\hat\theta_{n,N}$ regular at rate $\sqrt{(n+N)\an}$ satisfies
    \begin{equation}
       (n+N)^{1/2}\an^{1/2} \left(\hat\theta_{n,N} - \thetatarget \right) \dto \Normal{0}{\Sigma^\star} * M
    \end{equation}
    for some probability measure $M$ on $\RR^q$.
\end{theorem}

\begin{corollary}
\label{cor:aipw}
    Under the same Assumptions as Theorem~\ref{th:convolution}, if $\Bar{\mu} = \mutarget$ and $\pibar=\pitarget$, the estimator $\hat{\theta}_{\hat\mu;\hat\pi}$ obtained by solving Eq.~\ref{eq:est_eq_piunknown} is semiparametric efficient.
\end{corollary}

\begin{remark}[\texttt{AIPW} efficiency in non-i.i.d.~setting]
The results in this section formally establish that Augmented Inverse Probability Weighting (\texttt{AIPW}) estimators are not only doubly robust and asymptotically normal but also semiparametrically efficient within the \ds{} framework. This is a significant extension of classical efficiency theory, which typically assumes a fixed data-generating distribution and i.i.d.~observations. By proving a convolution theorem for triangular arrays, we demonstrate that when the nuisance functions (propensity score and outcome projection) are correctly specified or converge at appropriate rates, the \texttt{AIPW} estimator achieves the semiparametric efficiency lower bound $\Sigma^\star$. This confirms that our proposed framework provides an optimal inferential benchmark even under the challenging regime of decaying overlap and non-i.i.d.~data structures.
\end{remark}


\section{Simulation study}
\label{sec:sim}

In this Section, we instantiate the general theory developed above for two specific target parameters: the \textit{multivariate outcome mean} and the \textit{coefficients of a linear regression model}. For each case, we conduct extensive simulation studies to assess the finite-sample performance of the proposed estimators under varying degrees of overlap and model misspecifications. We compare our estimators to a \textit{naive} estimator computed only on the labeled dataset, 
the \textit{outcome regression} (\texttt{OR}) and the \textit{inverse probability weighting} (\texttt{IPW}) counterparts. The \texttt{OR} influence function is
\begin{equation}
    \orinfluence{\data;\thetatarget} = \EE{\finfluence{\data;\thetatarget}\mid X}\,,
\end{equation}
while the \texttt{IPW} influence function is
\begin{equation}
    \ipwinfluence{\data;\thetatarget} = \frac{R}{\pitarget(X)}\finfluence{\data;\thetatarget}\,.
\end{equation}
We employ cross-fitting to learn nuisance functions involved in the computation of our estimator, \texttt{OR} and \texttt{IPW}, with $J=2$ folds.

We are interested in probing double robustness and measuring actual coverage. Therefore, we assess performance in terms of estimation accuracy, inferential coverage, and inferential power. To measure accuracy, we compute the root mean squared error under the $\mathbbmss{L}_2$-distance between the true $\thetatarget\in\RR^q$ and its estimate $\hat\theta$, i.e.~$\MSE{\hat\theta} = \norm{\hat\theta - \thetatarget}{2}$. To measure coverage, we evaluate the empirical fraction of times $\Delta$ in which each component of $\thetatarget$ falls into the asympotic rectangular region defined by the $1-\alpha$ coordinate-wise confidence intervals of $\hat\theta$. To measure power, we compute the average width of the $1-\alpha$ coordinate-wise confidence intervals for each component of $\thetatarget$. In Supplementary Section~\ref{supp_sec:sim}, we also report simulation results for simultaneous confidence regions.

We mimic and extend the data-generating process in the simulation study conducted by \citet{zhang2023double}. For each $i=1,\dots,n+N$, we first generate $p$-dimensional i.i.d.~Gaussian covariates $X_i\sim\Normal{0}{I_p}$. Then we assign the missingness label according to a Bernoulli law, i.e.~$R_i\mid X_i\sim\text{Ber}(\pitarget(X_i))$, where the propensity score $\pitarget$ is modeled as:
\begin{itemize}
    \item \textit{Decaying MCAR}. The labeled indices are selected uniformly at random, implying $\pitarget(X_i) = n / (n+N)$.
    \item \textit{Decaying logistic}. The selection probabilities are set to $\pitarget(X_i) = \text{expit}\left(\log(n / (n+N)) + X_i^T \gamma \right)$ where $\text{expit}(x) = \exp(x) / (1 + \exp(x))$. The vector of coefficients $\gamma$ is defined as $\gamma=(1,-1,0_{p-2})$. This setting corresponds to a logistic propensity score with \textit{decaying offset} (see Supplementary Figure~\ref{fig:logistic_ps} for a graphical illustration).
\end{itemize}
We then generate outcomes $Y_i\in\RR^k$ according to the following linear model:
\begin{equation}
    Y_i = \beta^TX_i + \varepsilon_i\,,
\end{equation}
where $\beta\in\RR^{p\times k}$, $\beta_\ell\sim\Normal{0}{I_p}$  for each $\ell=1,\dots,k$, and $\varepsilon_i \sim \Normal{0}{0.1\cdot I_k}$.

We model the propensity score using a constant estimator, which essentially corresponds to a MCAR estimator, and an offset logistic regression. Within each fold $j$, the constant estimator computes $\sum_{i\in\hat{\mathbbmss{P}}^{[j]}} R_i/ |\hat{\mathbbmss{P}}^{[j]}|$, while the offset in the logistic regression is set at $\log\left(\sum_{i\in\hat{\mathbbmss{P}}^{[j]}} R_i/ |\hat{\mathbbmss{P}}^{[j]}|\right)$. Similarly, we model the outcome regression using either a constant model (mean) or linear regression. For each simulation scenario, we set $\alpha=0.05$, $p=10$, $n \in \{ 100, 500, 1000\}$, $N \in \{10n, 100n\}$. Each configuration is replicated across $1000$ random seeds.

\subsection{Multivariate outcome mean}
Given full data $\data^F = (X,Y)$, with $X\in\RR^p$ and $Y\in\RR^k$, the $k$-dimensional full-data influence function for the $k$-dimensional target parameter $\thetatarget=\EE{Y}$ is
\begin{equation}
    \finfluence{\data^F;\thetatarget} = Y - \thetatarget\,.
\end{equation}
The observed data are $\data = (X,R,RY)$, and thus the observed-data influence function takes the form
\begin{equation}
    \influence{\data;\thetatarget} = \mutarget(X) + \frac{R}{\pitarget(X)} \left(Y - \mutarget(X)\right) - \thetatarget\,,
\end{equation}
where $\mutarget:\RR^p\to\RR^k$ is equal to $\EE{Y\mid X}$. The previous influence function motivates the well-known \textit{augmented inverse probability weighting} (\texttt{AIPW}) estimator \citep{bang2005doubly}
\begin{equation}
\label{eq:aipw}
    \hat\theta_{\hat\mu;\hat\pi}^\texttt{AIPW} = \frac{1}{n+N}  \sum_{j=1}^J \sum_{i\in\mathcal{K}^{[j]}} \left[ \hat\mu^{[-j]}(X_i) + \frac{R_{i}}{\hat\pi^{[-j]}(X_i)} \left(Y_i - \hat\mu^{[-j]}(X_i)\right) \right]\,.
\end{equation}
Similarly, the naive, \texttt{OR} and \texttt{IPW} \citep{rosenbaum1983central} estimators for the multivariate outcome mean are:
\begin{equation}
    \hat\theta^{\texttt{naive}} = \frac{1}{\sum_{i=1}^{n+N} R_i} \sum_{i=1}^{n+N} R_i Y_i \,,\quad \hat\theta_{\hat\mu}^\texttt{OR} = \frac{1}{n+N} \sum_{j=1}^J \sum_{i\in\mathcal{K}^{[j]}} \hat\mu^{[-j]}(X_i)\,,\quad  \hat\theta_{\hat\pi}^\texttt{IPW} = \frac{1}{n+N}  \sum_{j=1}^J \sum_{i\in\mathcal{K}^{[j]}} \frac{R_{i}Y_i}{\hat\pi^{[-j]}(X_i)}\,.
\end{equation}

Here, we set $k=2$, so that the problem becomes bivariate outcome mean estimation. Figure~\ref{fig:sim_mean} reports results for the decaying logistic scenario with $n = 100$ labeled samples and $N = 1000$ unlabeled samples, the outcome model $\hat\mu$ estimated using linear regression, and propensity score model $\hat\pi$ estimated via constant estimator. All the other simulation results are reported in Supplementary Section~\ref{supp_sec:sim}. The boxplot of RMSE shows that both our \texttt{AIPW} proposal and \texttt{OR} achieve low and stable error across replications, while \texttt{IPW}, using a misspecified nuisance function, exhibits bias. Similarly, the naive estimator, which does not adjust for distribution shift, performs poorly. The central and right panels confirm these trends in terms of coverage: our approach and \texttt{OR} attain near-nominal coverage (around 95\%) with ``narrow'' confidence intervals, while \texttt{IPW} and the naive estimator suffer from substantial undercoverage despite the wider confidence intervals. These results highlight the robustness and efficiency of \texttt{AIPW} under the \ds{} regime. Similar conclusions can be drawn by analyzing other simulation scenarios in Supplementary Section~\ref{supp_sec:sim}.

\begin{figure}
    \centering
    \includegraphics[width=\linewidth]{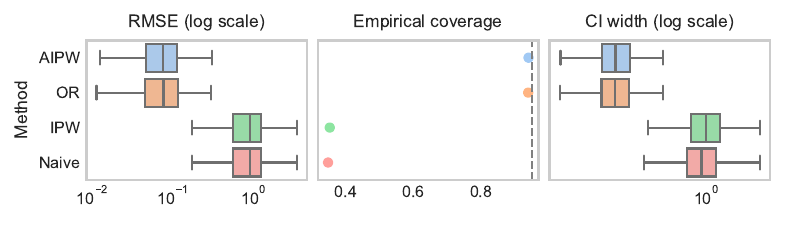}
    \caption{Multivariate outcome mean simulation results. The simulation is run $1000$ times under the decaying logistic scenario with $n = 100$ labeled samples and $N = 1000$ unlabeled samples; the outcome model $\hat\mu$ is estimated using linear regression, and propensity score model $\hat\pi$ is estimated via constant estimator. The left panel displays boxplots summarizing the RMSE distribution over different seeds of the considered estimators. The central panel shows empirical coverage. The right panel displays the distribution of the width of the confidence intervals over different seeds.}
    \label{fig:sim_mean}
\end{figure}

\subsection{Linear regression coefficients}
Given full data $\data^F = (X,Y)$, with $X\in\RR^p$ and $Y\in\RR$, the $p$-dimensional full-data influence function for the $p$-dimensional target parameter $\thetatarget=\argmin_\theta\EE{(Y-X^T\theta)^2}$ is
\begin{equation}
    \finfluence{\data^F;\thetatarget} = \Sigma_X^{-1} X \left(Y - X^T\thetatarget \right)\,,
\end{equation}
where $\Sigma_X = \EE{XX^T} \in \RR^{p\times p}$. The observed data are $\data = (X,R,RY)$, and thus the observed-data influence function takes the form
\begin{equation}
\label{eq:aipw_lm}
    \influence{\data;\thetatarget} = \Sigma_X^{-1}X\mutarget(X) + \frac{R}{\pitarget(X)} \Sigma_X^{-1}X\left(Y - \mutarget(X)\right) -\Sigma_X^{-1}XX^T\thetatarget\,,
\end{equation}
where $\mutarget:\RR^p\to\RR$ is equal to $\EE{Y\mid X}$. 
We estimate the nuisance $\EE{XX^T}$ at a parametric rate, i.e.~$(n+N)^{-1/2}$, using the covariates in both the labeled and unlabeled datasets via the sample covariance $\hat\Sigma = (n+N)^{-1}\sum_{i=1}^{n+N} X_i X_i^T$.
The previous influence function motivates the \textit{augmented inverse probability weighting} (\texttt{AIPW}) estimator
\begin{equation}
\label{eq:ols}
    \hat\theta_{\hat\mu;\hat\pi}^\texttt{AIPW} = \hat\Sigma^{-1} \frac{1}{n+N} \sum_{j=1}^J \sum_{i\in\mathcal{K}^{[j]}} X_i \left[\hat\mu^{[-j]}(X_i) + \frac{R_{i}}{\hat\pi^{[-j]}(X_i)} \left(Y_i - \hat\mu^{[-j]}(X_i)\right) \right]\,.
\end{equation}
Similarly, the naive, \texttt{OR} and \texttt{IPW} estimators for the linear regression coefficients are:
\begin{equation}
    \begin{gathered}
    \hat\theta^{\texttt{naive}} = \left(\frac{1}{\sum_{i=1}^{n+N} R_i} \sum_{i=1}^{n+N} R_i X_i X_i^T \right)^{-1} \frac{1}{\sum_{i=1}^{n+N} R_i} \sum_{i=1}^{n+N}  X_i R_i Y_i \,,\quad \hat\theta_{\hat\mu}^\texttt{OR} = \hat\Sigma^{-1} \frac{1}{n+N} \sum_{j=1}^J \sum_{i\in\mathcal{K}^{[j]}} X_i \hat\mu^{[-j]}(X_i) \,,\\
    \hat\theta_{\hat\pi}^\texttt{IPW} = \hat\Sigma^{-1} \frac{1}{n+N} \sum_{j=1}^J \sum_{i\in\mathcal{K}^{[j]}} X_i \frac{R_{i}Y_i}{\hat\pi^{[-j]}(X_i)}\,.
    \end{gathered}
\end{equation}

Figure~\ref{fig:sim_coef} presents results for the decaying logistic scenario ($n = 100$, $N = 1000$) where the outcome model $\hat\mu$ is estimated via linear regression and the propensity score $\hat\pi$ via a constant estimator. In terms of estimation accuracy, the RMSE boxplots show that all estimators except \texttt{IPW} achieve low and stable error. The \texttt{IPW} estimator, relying on a misspecified propensity model, performs poorly, exhibiting both higher bias and variability. The analysis of inferential properties (central and right panels) reveals a critical limitation of pure outcome regression in this setting. The \texttt{OR} estimator displays extremely low coverage despite its high accuracy; because $\hat\mu$ is well-specified, the residuals vanish, resulting in artificially narrow confidence intervals that fail to account for the intrinsic noise of the outcome. Conversely, the biased \texttt{IPW} estimator approximately attains correct coverage because its substantial variance inflates the confidence interval widths. The proposed \texttt{AIPW} estimator effectively resolves this trade-off, inheriting the high estimation precision of \texttt{OR} while recovering the valid inferential coverage of \texttt{IPW}. We also note that the naive estimator remains valid in this setting, as the conditional distribution of $Y$ given $X$ is invariant under MAR \citep{little1992regression, moon2025augmented}. These findings underscore the robustness of \texttt{AIPW} under decaying overlap; additional scenarios supporting these conclusions are detailed in Supplementary Section~\ref{supp_sec:sim}.

\begin{figure}
    \centering
    \includegraphics[width=\linewidth]{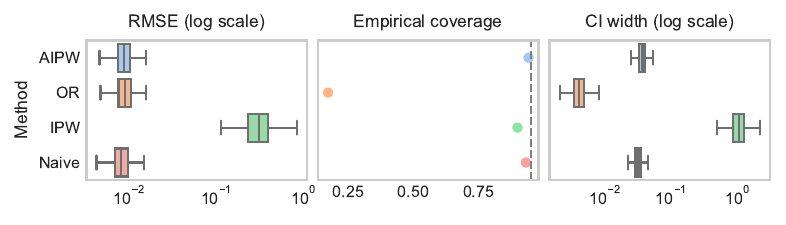}
    \caption{Linear regression coefficients simulation results. The simulation is run $1000$ times under the decaying logistic scenario with $n = 100$ labeled samples and $N = 1000$ unlabeled samples; the outcome model $\hat\mu$ is estimated using linear regression, and propensity score model $\hat\pi$ is estimated via constant estimator. The left panel displays boxplots summarizing the RMSE distribution over different seeds of the considered estimators. The central panel shows empirical coverage. The right panel displays the distribution of the width of the confidence intervals over different seeds.}
    \label{fig:sim_coef}
\end{figure}

\section{Real-data applications}
\label{sec:app}
\subsection{\texttt{BLE-RSSI} study}
To demonstrate the practical utility of our approach, we first apply it to the \texttt{BLE-RSSI} dataset \citep{mohammadi2017semisupervised}. This dataset was collected in a library using Received Signal Strength Indicator (RSSI) measurements obtained from an array of $13$ Bluetooth Low Energy (BLE) iBeacons. Data acquisition was performed using a smartphone during the library’s operational hours. The primary objective motivating this dataset is to enable indoor localization, that is, to accurately estimate the spatial position of a mobile device based on wireless signal strength measurements, a key problem in many real-world applications including navigation, resource management, and smart building systems.

The dataset comprises two components: a labeled subset containing $n = 1420$ observations and an unlabeled subset containing $N = 5191$ observations. For each labeled observation, the recorded features include a location identifier and the RSSI readings from the 13 iBeacons. Location identifiers combine a letter (representing longitude) and a number (representing latitude) within the physical layout of the floor. A schematic diagram illustrating the arrangement of the iBeacons and the labeled locations is provided in Supplementary Figure~\ref{fig:room}. Each location label is then mapped to a two-dimensional real-valued coordinate, representing the longitude and latitude of the device at the time of measurement.

The inferential goal in this application is to estimate the mean spatial position of the device across the observed population. Thus, the problem can be formulated as one of bivariate outcome mean estimation, where the target parameter is the mean of a two-dimensional random vector. Formally, let $Y_i \in \mathbb{R}^2$ denote the bivariate spatial coordinates (longitude and latitude) for observation $i$; $R_{i} \in \{0, 1\}$ indicate whether the location label is observed; and $X_i \in \mathbb{R}^p$ represent the vector of $p=13$ RSSI readings from the iBeacons, used as auxiliary covariates. The observed data take the form $\mathcal{D}_i = (X_i, R_{i}, R_{i} Y_i)$, for $i = 1, \dots, n+N$.

We compare two estimation strategies: the naive estimator, which only uses labeled observations and ignores selection bias, and the \texttt{AIPW} estimator defined in Equation~\ref{eq:aipw}, which corrects for the semi-supervised structure. For the \texttt{AIPW} estimator, the outcome regression function $\hat\mu$ is estimated using random forests \citep{breiman2001random}, with RSSI measurements as predictors. The propensity score $\hat\pi$ is estimated using logistic regression with offset, also based on RSSI predictors. To mitigate overfitting and ensure valid asymptotic properties, both nuisance functions are estimated using cross-fitting with $J = 5$ folds.

\begin{figure}
    \centering
    \includegraphics[width=\linewidth]{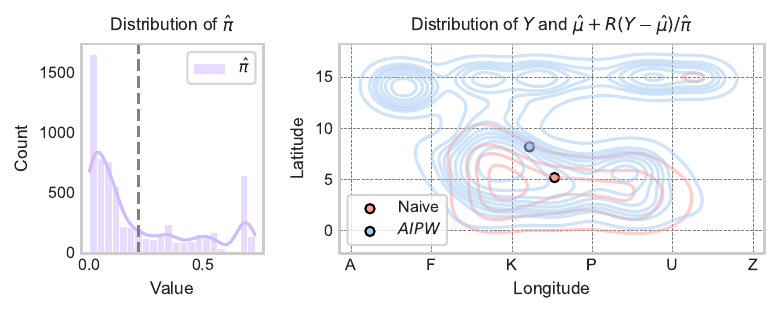}
    \caption{\texttt{BLE-RSSI} application results. Left panel: estimated propensity scores across observations, along with the proportion of labeled data $n / (n + N)$, shown as a vertical dashed line. The heterogeneity in the estimated propensity scores may suggest a missing-at-random (MAR) labeling mechanism. Right panel: joint bivariate distribution over the floor bivariate grid of the observed position of the device $Y$ (red) compared to the distribution of pseudo-outcomes (blue), that is augmented predicted positions from the estimated outcome model using signal recorded by remote sensors. The discrepancy between observed and predicted outcomes may suggest a missing-at-random (MAR) labeling mechanism. Dots represent bivariate means computed using our \texttt{AIPW} approach and the naive approach.}
    \label{fig:app_spatial}
\end{figure}

Results are presented in Figure~\ref{fig:app_spatial}. The heterogeneous distribution of the estimated propensity scores indicates that the probability of label observation varies across the covariate space, supporting the plausibility of a missing-at-random (MAR) assumption conditional on the RSSI features. Furthermore, the distribution of the pseudo-outcomes, i.e.~the spatial coordinates $\hat\mu+R(Y-\hat\mu)/\hat\pi$ shows a marked shift compared to the labeled outcomes $Y$, moving upward and to the left in the two-dimensional grid. This discrepancy between the labeled and imputed outcome distributions highlights the necessity of properly accounting for the semi-supervised structure of the data. Ignoring the unlabeled data or treating the labeled set as representative would likely introduce substantial bias in the estimation of the mean device position. Overall, this example demonstrates how the proposed methodology can successfully leverage partially labeled data to produce robust and unbiased inference in complex real-world settings.

\subsection{\texttt{NHEFS} study}
We further apply our method to the \textit{National Health and Nutrition Examination Survey Data I Epidemiologic Follow-up Study} (\texttt{NHEFS}). The \texttt{NHEFS} was jointly initiated by the National Center for Health Statistics and the National Institute on Aging to investigate relationships between clinical, nutritional, and behavioral factors. 

Following \citet{hernan2010causal, ertefaie2023nonparametric, zhang2023decaying}, we consider a subset of the \texttt{NHEFS} which consists of 1561 smokers aged 25-74, who had a first baseline visit and a second follow-up visit approximately 10 years later. Our goal is to estimate the mean for two outcomes of interest: weight change and smoking intensity change. We thus formulate this problem as bivariate outcome mean estimation. To simulate a challenging semi-supervised setting with distribution shift, we define the labeled set based on alcohol consumption. Specifically, we treat observations as labeled ($R_i=1$) only if the participant abstains from alcohol, while drinkers are treated as unlabeled observations ($R_i=0$). This creates a scenario where the labeled mechanism is sparse (approximately $12.5\%$ of the data are labeled, i.e.~$n=195$ and $N=1366$) and likely biased, as abstainers may differ systematically from the general population in terms of health behaviors. When observed, the response variable $Y_i \in \RR^2$ contains weight change and smoking intensity change; while the always observed $X_i \in \RR^{p}$ contains $p=9$ demographic and health covariates, including baseline weight, age, sex, race, education, exercise and activity habits, years of smoking, and whether the individual has quitted smoking between the two visits. We compare the naive estimator (sample mean of labeled data) against the \texttt{AIPW} estimator defined in Equation~\ref{eq:aipw}. As in the previous application, we estimate the nuisance function $\hat{\mu}$ using random forests \citep{breiman2001random} and the propensity score $\hat{\pi}$ using logistic regression with an offset, employing 5-fold cross-fitting.

\begin{figure}
    \centering
    \includegraphics[width=\linewidth]{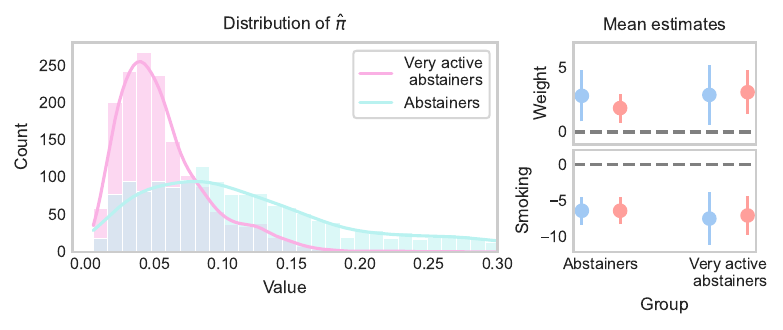}
    \caption{\texttt{NHEFS} application results. Left panel: estimated propensity scores across observations, for various missingness indicators (very active, moderately active, marginal). The heterogeneity in the estimated propensity scores may suggest a missing-at-random (MAR) labeling mechanism. Right panel: estimated means from \texttt{AIPW} (blue) and the naive estimator based solely on labeled data (red). Confidence intervals are at $\alpha=0.05$ significance level.}
    \label{fig:nhefs}
\end{figure}

Results are presented in Figure~\ref{fig:nhefs}. On the labeled dataset, the naive estimator suggests a mean weight gain of $1.88kg$. Although not significantly different, the \texttt{AIPW} estimator appears to correct for the distribution shift, estimating a higher population mean weight gain of $2.34kg$. This suggests that the labeled sub-population (alcohol abstainers) may experience less weight gain than the general population, a bias our method identifies and corrects. The results for the change in smoking intensity seem to agree between the two estimators, with \texttt{AIPW} estimating a reduction of $-6.36$ cigarettes per day.

To empirically validate our theoretical findings regarding decaying overlap, we further stratify the dataset by activity level. In particular, we only consider the cohort of ``very active'' individuals as labeled. This subset comprises only $n=89$ subjects (approximately 5.7\% of the entire dataset). In this regime, the \texttt{AIPW} estimator remains stable with reasonable standard errors. In fact, the estimated weight change is $2.91kg$, compared to the naive estimate of $3.12kg$. Similarly, for the change in smoking intensity, the \texttt{AIPW} estimator reports a reduction of $-7.44$ cigarettes per day, while the naive estimator suggests a reduction of $-7.01$. The left panel of Figure~\ref{fig:nhefs} illustrates the distribution of the estimated propensity scores for this subgroup, which are concentrated near zero (mostly $<0.15$), reflecting the severe sparsity of the labeled data. Despite this challenging regime ($n/N\approx0.057$), the estimator yields informative 95\% confidence intervals -- $[0.58,5.24]$ for weight gain and $[-11.05,-3.82]$ for smoking reduction -- confirming that our method maintains inferential validity even under significant label scarcity. The divergence between the naive and \texttt{AIPW} estimates indicates that the ``very active'' labeled cohort is systematically different from the general population; specifically, the naive estimator appears to overestimate weight gain and underestimate smoking reduction, a bias that our approach detects and corrects by properly weighting the unlabeled data.


\section{Conclusions}
\label{sec:end}

In this work, we have developed a general semiparametric framework for efficient estimation and inference in semi-supervised learning problems where outcome labels are missing at random (MAR) and the overlap decays with sample size. Our approach accommodates arbitrary statistical targets and remains valid under flexible modeling of high-dimensional nuisance functions. Leveraging the tools of semiparametric theory, we have derived efficient influence functions under decaying overlap and shown that \texttt{AIPW} estimators preserve double robustness, asymptotic normality, and semiparametric efficiency. 

A key insight of our analysis is that overlap decay imposes fundamental adaptations of inference, invalidating classical root-$n$ asymptotics and requiring a redefinition of the effective sample size. We formally characterized this rate through a generalization of Lindeberg conditions and provided practical \texttt{AIPW} estimators that remain consistent in this more challenging regime. We further extended our methodology to incorporate black-box prediction models in the spirit of prediction-powered inference (\texttt{PPI}), broadening its applicability to MAR labeling mechanisms.

Our simulation studies confirmed the theoretical properties of \texttt{AIPW} estimators in the \ds{} setting, including double robustness, asymptotic normality, and valid coverage under finite samples and severe overlap decay. These findings suggest that our framework is both theoretically sound and practically useful, especially in settings where labeled data are scarce and distribution shift is non-negligible.

Despite these contributions, several limitations remain. First, while we allow for decaying overlap and general targets, our current framework assumes that the MAR assumption holds, i.e., distribution shift depends only on covariates and not on unobserved variables. Relaxing this assumption toward a missing-not-at-random (MNAR) setting would require substantially different identification strategies. Second, the performance of \texttt{AIPW} estimators under strong model misspecification remains to be fully understood. Lastly, while we allow for high-dimensional covariates and flexible machine learning models, our theoretical guarantees rely on sufficient regularity and convergence rates that may be difficult to verify or achieve in practice.

Future research directions include extending our theory to MNAR mechanisms and exploring adaptive procedures that optimize the bias-variance trade-off in regions of low overlap. Another important direction is the integration of these techniques into large-scale machine learning pipelines for high-throughput scientific applications, such as genomic association studies, fairness-aware recommendation systems, and policy evaluation in digital platforms.

By relaxing traditional assumptions while preserving rigorous inferential guarantees, our work provides a foundation for reliable statistical inference in the increasingly common regime of biased labeling and partial supervision.

\section*{Acknowledgments}
L.T.~wishes to thank Francesca Chiaromonte and Edward H.~Kennedy for very helpful discussions. This project was funded by National Institute of Mental Health (NIMH) grant R01MH123184.

\bibliographystyle{plainnat}
\bibliography{bib}

\clearpage
\setcounter{page}{1}
\appendix
\section*{Supplementary Material of ``Semiparametric semi-supervised learning for general targets under distribution shift and decaying overlap''}

\renewcommand\thefigure{\thesection.\arabic{figure}}
\renewcommand\thetable{\thesection.\arabic{table}}
\setcounter{figure}{0} 
\setcounter{table}{0}

\section{Proofs of main statements}
\label{supp_sec:proofs}

We denote the empirical average operator as $\mathbbmss{P}_{n+N}$. We also write $\shortfif$ to indicate the full-data influence function for notational simplicity.

\subsection{Derivation of Equation~\ref{eq:if}}
Given a full-data influence function $\shortfif$, by Theorem 7.2 in \citet{tsiatis2006semiparametric}, we know that the space of associated observed-data influence functions $\Lambda^\perp$ is given by 
\begin{equation}
\label{eq:class_oif}
    \Lambda^\perp = \left\{\frac{R}{\pitarget(X)} \shortfif \oplus \Lambda_2 \right\}\,,
\end{equation}
where $ \Lambda_2 = \left\{ L_2(\data) \,:\, \EE{L_2(\data)\mid \data^F} = 0 \right\}$. 

By Theorem 10.1 in \citet{tsiatis2006semiparametric}, for a fixed $\shortfif$, the optimal observed-data influence function among the class in Equation~\ref{eq:class_oif} is obtained by choosing
\begin{equation}
    L_2(\data) = - \Pi\left(\frac{R}{\pitarget(X)} \shortfif \mid \Lambda_2 \right)\,,
\end{equation}
where the operator $\Pi$ projects the element $R\shortfif/\pitarget(X)$ onto the space $\Lambda_2$ -- see Theorem 2.1 in \citet{tsiatis2006semiparametric}. 

Finally, by Theorem 10.2 in \citet{tsiatis2006semiparametric}, we know that 
\begin{equation}
    \Pi\left(\frac{R}{\pitarget(X)} \shortfif \mid \Lambda_2 \right) = \left(\frac{R - \pitarget(X)}{\pitarget(X)} \right) h_2(X) \in \Lambda_2\,,
\end{equation}
where $h_2(X) = \EE{\shortfif\mid X}$. This concludes the derivation.

\subsection{Proof of Theorem~\ref{th:ansi_knownpi}}
\begin{proof} 
For simplicity, we show the result for a single cross-fitting fold. In particular, we denote the distribution where $\hat\mu$ is trained as $\hat{\mathbbmss{P}}$ and the distribution where the influence functions are approximated as $\mathbbmss{P}_{n+N}$. Similarly, we denote as $\Bar{\mu}$ the population limit of $\hat\mu$. The propensity score is known, so that $\hat\pi=\pibar=\pitarget$. By \textit{Von Mises} expansion, we have
\begin{equation}
\begin{split}
    \hat\theta_{\hat\mu} - \thetatarget &= \theta\left(\hat{\mathbbmss{P}}\right) + \mathbbmss{P}_{n+N}\left[\influence{\data;\thetatarget;\hat\mu;\pitarget} \right] - \theta(\jPtarget) \\
    &= \left(\mathbbmss{P}_{n+N} - \mathbbmss{E}\right)\left[\influence{\data;\thetatarget;\hat\mu;\pitarget} \right] + \eta(\hat\mu,\mutarget) \\
    &=\left(\mathbbmss{P}_{n+N} - \mathbbmss{E}\right)\left[\influence{\data;\thetatarget;\Bar{\mu};\pitarget}\right] + \left(\mathbbmss{P}_{n+N} - \mathbbmss{E}\right) \left[ \influence{\data;\thetatarget;\hat\mu;\pitarget} - \influence{\data;\thetatarget;\Bar{\mu};\pitarget} \right] + \eta(\hat\mu,\mutarget) \\
    &=\mathbbmss{P}_{n+N}\left[\influence{\data;\thetatarget;\Bar{\mu};\pitarget}\right] + \left(\mathbbmss{P}_{n+N} - \mathbbmss{E}\right) \left[ \influence{\data;\thetatarget;\hat\mu;\pitarget} - \influence{\data;\thetatarget;\Bar{\mu};\pitarget} \right] + \eta(\hat\mu,\mutarget) \,,
\end{split}
\end{equation}
where $\eta(\hat\mu,\mutarget) = \theta\left(\hat{\mathbbmss{P}}\right) - \theta(\jPtarget) + \EE{\influence{\data;\thetatarget;\hat{\mu};\pitarget}}$ by definition. We refer to the three elements on the right-hand side respectively as the \textit{influence function term},  the \textit{empirical process term} and the \textit{remainder term}. 

\textbf{Consistency.} To prove consistency, we need to show that the norm of each term on the right-hand side is $O_\mathbbmss{P}\left((n+N)^{-1/2}\an^{-1/2}\right)$. We do so by analyzing each term separately. 

\textit{First term (influence function).} By definition of influence function, $\EE{\influence{\data;\thetatarget;\Bar{\mu};\pitarget}}=0$, and thus the first term has mean zero. We now analyze the variance. Define $\varbar = \VV{\influence{\data;\thetatarget;\Bar{\mu};\pitarget}} = \EE{\influence{\data;\thetatarget;\Bar{\mu};\pitarget} \influence{\data;\thetatarget;\Bar{\mu};\pitarget}^T}$. By Assumption~\ref{ass:reg}, \textbf{c}, we have $\maxeig{\varbar} \asymp \an^{-1}$. Therefore, we have
\begin{equation}
    \EE{\left((n+N)^{1/2}\an^{1/2} \norm{\mathbbmss{P}_{n+N} \left[\influence{\data;\thetatarget;\Bar{\mu};\pitarget} \right]}{2} \right)^2} = \an \norm{\varbar}{2}^2 \leq \an q^2 \maxeig{\varbar} = O_\mathbbmss{P}(1)\,,
\end{equation}
which in turn implies $\norm{\mathbbmss{P}_{n+N} \left[\influence{\data;\thetatarget;\Bar{\mu};\pitarget} \right]}{2} = O_\mathbbmss{P}\left((n+N)^{-1/2}\an^{-1/2}\right)$.

\textit{Second term (Empirical process).} We need to show that the norm of $\left(\mathbbmss{P}_{n+N} - \mathbbmss{E}\right) \left[ \influence{\data;\thetatarget;\hat\mu;\pitarget} - \influence{\data;\thetatarget;\Bar{\mu};\pitarget} \right]$ is of order $O_{\mathbbmss{P}}\left((n+N)^{-1/2}\an^{-1/2}\right)$. Define $\delta(\data) = \influence{\data;\thetatarget;\hat\mu;\pitarget} - \influence{\data;\thetatarget;\Bar{\mu};\pitarget}\in\RR^q$. Here, we show the stronger result 
\begin{equation}
    (n+N)^{1/2}\an^{1/2} \norm{\left(\mathbbmss{P}_{n+N} - \mathbbmss{E}\right) \left[ \delta(\data) \right]}{2} = o_{\mathbbmss{P}}\left(1\right)\,.
\end{equation}
First, we have:
\begin{equation}
    \delta(\data) = \left(1-\frac{R}{\pitarget(X)}\right) \left(\hat\mu(X) - \Bar\mu(X)\right)\,.
\end{equation}
Conditional on the training fold $\hat{\mathbbmss{P}}$ (so $\hat{\mu}$ is fixed), this term is a mean zero vector. We now calculate the expected squared Euclidean norm of the scaled empirical process vector $(n+N)^{1/2} \an^{1/2} \mathbbmss{P}_{n+N}\left[\delta(\data) \right]$ conditional on $\hat{\mathbbmss{P}}$:
\begin{equation}
\begin{split}
    \EE{\norm{(n+N)^{1/2} \an^{1/2} \mathbbmss{P}_{n+N}\left[\delta(\data) \right]}{2}^2\mid \hat{\mathbbmss{P}}}
    &= \frac{\an}{n+N} \sum_{i=1}^{n+N} \EE{\norm{\delta(\data_i)}{2}^2 \mid \hat{\mathbbmss{P}}} \\
    &= \an \EE{\norm{\delta(\data)}{2}^2 \mid \hat{\mathbbmss{P}}} \\
    &= \an \EE{ \left( 1 - \frac{R}{\pitarget(X)} \right)^2 \norm{\hat{\mu}(X) - \Bar{\mu}(X)}{2}^2 \mid \hat{\mathbbmss{P}}} \\
    &= \an \EE{ \left(\frac{1}{\pitarget(X)} - 1\right) \norm{\hat{\mu}(X) - \Bar{\mu}(X)}{2}^2 \mid \hat{\mathbbmss{P}}} \,,
\end{split}
\end{equation}
where the last equality follows from Lemma~\ref{lem:simple_pi}. Notice that this term goes to 0 with $n$ going to infinity under Assumption~\ref{ass:reg}, \textbf{a}. We now apply conditional Chebyshev inequality:
\begin{equation}
    \PP{\norm{(n+N)^{1/2} \an^{1/2} \mathbbmss{P}_{n+N}\left[\delta(\data) \right]}{2} \geq \varepsilon \mid \hat{\mathbbmss{P}}} \leq \frac{\EE{\norm{(n+N)^{1/2} \an^{1/2} \mathbbmss{P}_{n+N}\left[\delta(\data) \right]}{2}^2\mid \hat{\mathbbmss{P}}}}{\varepsilon^2} \to 0\,.
\end{equation}
By taking marginal expectations over the training set, one obtains the desired result.

\textit{Third term (Remainder).} Again, we need to show that the norm of $\eta(\hat\mu,\mutarget)$ equals $O_{\mathbbmss{P}}\left((n+N)^{-1/2}\an^{-1/2}\right)$. This directly follows from the fact the propensity score is known. In fact, by direct evaluation, in this case the remainder term is exactly zero:
\begin{equation}
\begin{split}
    \eta(\hat\mu,\mutarget) &= \theta\left(\hat{\mathbbmss{P}}\right) - \thetatarget + \EE{\hat\mu(X) + \frac{R}{\pitarget(X)}\left( \finfluence{\data;\theta\left(\hat{\mathbbmss{P}}\right)} - \hat\mu(X)\right)} \\
    &= \theta\left(\hat{\mathbbmss{P}}\right) - \thetatarget +\EE{\finfluence{\data;\theta\left(\hat{\mathbbmss{P}}\right)}} = 0\,.
\end{split}
\end{equation}

The previous results, combined, imply that $\norm{\hat\theta_{\hat\mu} - \thetatarget}{2} = O_{\mathbbmss{P}}\left((n+N)^{-1/2}\an^{-1/2}\right)$.

\textbf{Central Limit Theorem.} We now turn to the asymptotic normality of $\hat\theta_{\hat\mu}$ by showing that the first term on the right-hand side satisfies a form of Central Limit Theorem. This is enough to prove asymptotic normality of the estimator, as the remaining terms are already of order $o_\mathbbmss{P}((n+N)^{-1/2}\an^{-1/2})$.

By Assumption~\ref{ass:reg}, \textbf{c}, we have $\mineig{\varbar} \asymp \an^{-1}$. This implies:
\begin{equation}
\label{eq:eigen}
    \norm{\varbar^{-1/2}}{2}^2 = \left(\maxeig{\varbar^{-1/2}}\right)^2 = \left(\mineig{\varbar}\right)^{-1} \asymp \an\,.
\end{equation}

By the tail condition \textbf{d} in Assumption~\ref{ass:reg}, we have that
\begin{equation}
    \begin{split}
        &\EE{\norm{\varbar^{-1/2}\influence{\data;\thetatarget;\Bar{\mu};\pitarget}}{2}^2 \onea{\norm{\varbar^{-1/2}\influence{\data;\thetatarget;\Bar{\mu};\pitarget}}{2} > \varepsilon \sqrt{n+N}}} \\
        &\leq \EE{\norm{\varbar^{-1/2}}{2}^2\norm{\influence{\data;\thetatarget;\Bar{\mu};\pitarget}}{2}^2 \onea{\norm{\varbar^{-1/2}}{2}\norm{\influence{\data;\thetatarget;\Bar{\mu};\pitarget}}{2} > \varepsilon \sqrt{n+N}}} \\
        &\leq \EE{C\an\norm{\influence{\data;\thetatarget;\Bar{\mu};\pitarget}}{2}^2 \onea{\norm{\sqrt{\an}\influence{\data;\thetatarget;\Bar{\mu};\pitarget}}{2} > \varepsilon \sqrt{n+N}}} \to 0\quad\text{as}\,n,N\to\infty\,,
    \end{split}
\end{equation}
where the first inequality is due to triangle inequality and the second inequality, for some constant $C>0$, is due to the chain of equalities in Equation~\ref{eq:eigen}. Therefore the following Lindeberg condition also holds
\begin{equation}
    (n+N)^{-1} \sum_{i=1}^{n+N}\EE{\norm{\varbar^{-1/2}\influence{\data_i;\thetatarget;\Bar{\mu};\pitarget}}{2}^2 \onea{\norm{\varbar^{-1/2}\influence{\data_i;\thetatarget;\Bar{\mu};\pitarget}}{2} > \varepsilon \sqrt{n+N}}} \to 0\quad\text{as}\,n,N\to\infty\,,
\end{equation}
which implies, by Lindeberg-Feller Central Limit Theorem \citep{van2000asymptotic}, that
\begin{equation}
    (n+N)^{-1/2} \varbar^{-1/2} \sum_{i=1}^{n+N} \influence{\data_i;\thetatarget;\Bar{\mu};\pitarget} \dto \Normal{0}{I_q}\,.
\end{equation}
The last step is to show that:
\begin{equation}
    \left(\VV{\influence{\data;\thetatarget;\hat\mu;\pitarget}}\right)^{-1/2}= \varbar^{-1/2} + o_\mathbbmss{P}(\an^{1/2}) \,.
\end{equation}
This can be shown by bounding the difference between the two variance matrices. By expanding the $\mathbbmss{L}_2$ norm and applying the Cauchy-Schwarz inequality, the difference between the second moments is dominated by the cross-term:
\begin{equation}
    \begin{split}
        \norm{\VV{\influence{\data;\thetatarget;\hat\mu;\pitarget}} - \varbar}{2} &\le \EE{\norm{\influence{\data;\thetatarget;\hat\mu;\pitarget} - \influence{\data;\thetatarget;\bar\mu;\pitarget}}{2}^2} \\
        &+ 2 \sqrt{\EE{\norm{\influence{\data;\thetatarget;\bar\mu;\pitarget}}{2}^2}} \sqrt{\EE{\norm{\influence{\data;\thetatarget;\hat\mu;\pitarget} - \influence{\data;\thetatarget;\bar\mu;\pitarget}}{2}^2}} \\
        &\leq o_\mathbbmss{P}(\an^{-1/2})\,,
    \end{split}
\end{equation}
where the last step follows from Assumptions~\ref{ass:reg}, \textbf{a} and \textbf{c}. This means:
\begin{equation}
    \an\ \VV{\influence{\data;\thetatarget;\hat\mu;\pitarget}} \pto \an \varbar\,.
\end{equation}
Because Assumption~\ref{ass:reg}, \textbf{c}, guarantees that the scaled variance $\an \varbar$ is a positive definite matrix with eigenvalues bounded away from zero, we can apply the continuous mapping theorem to the inverse square root function:
\begin{equation}
    \left(\an \VV{\influence{\data;\thetatarget;\hat\mu;\pitarget}} \right)^{-1/2} = (\an \varbar)^{-1/2} + o_\mathbbmss{P}(1)\,,
\end{equation}
which delivers the result. In summary, we have
\begin{equation}
    (n+N)^{1/2} \VV{\influence{\data;\thetatarget;\hat\mu;\pitarget}}^{-1/2} \left(\hat\theta_{\hat\mu} - \thetatarget\right) =  (n+N)^{-1/2} \varbar^{-1/2} \sum_{i=1}^{n+N} \influence{\data_i;\thetatarget;\Bar{\mu};\pitarget} + o_\mathbbmss{P}(1)\,.
\end{equation}
\end{proof}

\subsection{Proof of Theorem~\ref{th:ansi_unknownpi}}
\begin{proof} For simplicity, we show the result for a single cross-fitting fold. In particular, we denote the distribution where $\hat\mu$ and $\hat\pi$ are trained as $\hat{\mathbbmss{P}}$ and the distribution where the influence functions are approximated as $\mathbbmss{P}_{n+N}$. Similarly, we denote as $\Bar{\mu}$ and $\pibar$ the population limit of $\hat\mu$ and $\hat\pi$, respectively. By \textit{Von Mises} expansion, we have
\begin{equation}
\begin{split}
    \hat\theta_{\hat\mu;\hat\pi} - \thetatarget &= \theta\left(\hat{\mathbbmss{P}}\right) + \mathbbmss{P}_{n+N}\left[\influence{\data;\thetatarget;\hat\mu;\hat\pi} \right] - \theta(\jPtarget) \\
    &= \left(\mathbbmss{P}_{n+N} - \mathbbmss{E}\right)\left[\influence{\data;\thetatarget;\hat\mu;\hat\pi} \right] + \eta\left((\hat\mu,\hat\pi),(\mutarget,\pitarget)\right) \\
    &=\left(\mathbbmss{P}_{n+N} - \mathbbmss{E}\right)\left[\influence{\data;\thetatarget;\Bar{\mu};\pibar}\right] + \left(\mathbbmss{P}_{n+N} - \mathbbmss{E}\right) \left[ \influence{\data;\thetatarget;\hat\mu;\hat\pi} - \influence{\data;\thetatarget;\Bar{\mu};\pibar} \right] + \eta\left((\hat\mu,\hat\pi),(\mutarget,\pitarget)\right) \\
    &=\mathbbmss{P}_{n+N}\left[\influence{\data;\thetatarget;\Bar{\mu};\pibar}\right] + \left(\mathbbmss{P}_{n+N} - \mathbbmss{E}\right) \left[ \influence{\data;\thetatarget;\hat\mu;\hat\pi} - \influence{\data;\thetatarget;\Bar{\mu};\pibar} \right] + \eta\left((\hat\mu,\hat\pi),(\mutarget,\pitarget)\right) \,,
\end{split}
\end{equation}
where $\eta((\hat\mu, \hat\pi),(\mutarget,\pitarget)) = \theta\left(\hat{\mathbbmss{P}}\right) - \theta(\jPtarget) + \EE{\influence{\data;\thetatarget;\hat{\mu};\hat\pi}}$ by definition. We refer to the three elements on the right-hand side respectively as the \textit{influence function term},  the \textit{empirical process term} and the \textit{remainder term}. 

\textbf{Consistency.} To prove consistency, we need to show that the norm of each term on the right-hand side is $O_\mathbbmss{P}\left((n+N)^{-1/2}\an^{-1/2}\right)$. We do so by analyzing each term separately. 

\textit{First term (influence function).} By definition of influence function, $\EE{\influence{\data;\thetatarget;\Bar{\mu};\pibar}}=0$, and thus the first term has mean zero. We now analyze the variance. Define $\varbar = \VV{\influence{\data;\thetatarget;\Bar{\mu};\pibar}} = \EE{\influence{\data;\thetatarget;\Bar{\mu};\pibar} \influence{\data;\thetatarget;\Bar{\mu};\pibar}^T}$. 
By Assumption~\ref{ass:reg}, \textbf{c}, we have $\maxeig{\varbar} \asymp \an^{-1}$. Therefore, we have
\begin{equation}
    \EE{\left((n+N)^{1/2}\an^{1/2} \norm{\mathbbmss{P}_{n+N} \left[\influence{\data;\thetatarget;\Bar{\mu};\Bar\pi} \right]}{2} \right)^2} = \an \norm{\varbar}{2}^2 \leq \an q^2 \maxeig{\varbar} = O_\mathbbmss{P}(1)\,,
\end{equation}
which in turn implies $\norm{\mathbbmss{P}_{n+N} \left[\influence{\data;\thetatarget;\Bar{\mu};\pibar} \right]}{2} = O_\mathbbmss{P}\left((n+N)^{-1/2}\an^{-1/2}\right)$.

\textit{Second term (Empirical process).} We need to show that the norm of $\left(\mathbbmss{P}_{n+N} - \mathbbmss{E}\right) \left[ \influence{\data;\thetatarget;\hat\mu;\hat\pi} - \influence{\data;\thetatarget;\Bar{\mu};\pibar} \right]$ is of order $O_{\mathbbmss{P}}\left((n+N)^{-1/2}\an^{-1/2}\right)$. Define $\delta(\data) = \influence{\data;\thetatarget;\hat\mu;\hat\pi} - \influence{\data;\thetatarget;\Bar{\mu};\pibar}\in\RR^q$. Here, we show the stronger result 
\begin{equation}
    (n+N)^{1/2}\an^{1/2} \norm{\left(\mathbbmss{P}_{n+N} - \mathbbmss{E}\right) \left[ \delta(\data) \right]}{2} = o_{\mathbbmss{P}}\left(1\right)\,.
\end{equation}
By definition of the empirical process operator $\left(\mathbbmss{P}_{n+N} - \mathbbmss{E}\right)$, the term $\left(\mathbbmss{P}_{n+N} - \mathbbmss{E}\right)\left[ \delta(\data) \right]$ is centered. We need to focus on the scaled conditional variance, i.e., the scaled expected squared Euclidean norm conditional on $\hat{\mathbbmss{P}}$:
\begin{equation}
\begin{split}
    \EE{\norm{(n+N)^{1/2} \an^{1/2} \left(\mathbbmss{P}_{n+N} - \mathbbmss{E}\right)\left[\delta(\data) \right]}{2}^2\mid \hat{\mathbbmss{P}}}
    &= \frac{\an}{n+N} \sum_{i=1}^{n+N} \EE{\norm{\delta(\data_i) - \EE{\delta(\data)}}{2}^2 \mid \hat{\mathbbmss{P}}} \\
    &= \an \EE{\norm{\delta(\data) - \EE{\delta(\data)}}{2}^2 \mid \hat{\mathbbmss{P}}} \\
    &\leq \an \EE{\norm{\delta(\data)}{2}^2 \mid \hat{\mathbbmss{P}}} \,.
\end{split}
\end{equation}
Notice that this term goes to 0 with $n$ going to infinity under Assumption~\ref{ass:reg}, \textbf{a}. We now apply conditional Chebyshev inequality:
\begin{equation}
    \PP{\norm{(n+N)^{1/2} \an^{1/2} \left(\mathbbmss{P}_{n+N} - \mathbbmss{E}\right)\left[\delta(\data) \right]}{2} \geq \varepsilon \mid \hat{\mathbbmss{P}}} \leq \frac{\EE{\norm{(n+N)^{1/2} \an^{1/2} \left(\mathbbmss{P}_{n+N} - \mathbbmss{E}\right)\left[\delta(\data) \right]}{2}^2\mid \hat{\mathbbmss{P}}}}{\varepsilon^2} \to 0\,.
\end{equation}
By taking marginal expectations over the training set, one obtains the desired result.

\textit{Third term (Remainder).} Again, we need to show that the norm of $\eta((\hat\mu,\hat\pi),(\mutarget,\pitarget))$ equals $O_{\mathbbmss{P}}\left((n+N)^{-1/2}\an^{-1/2}\right)$. This directly follows from Assumption~\ref{ass:reg}, \textbf{b}. 

The previous results, combined, imply that $\norm{\hat\theta_{\hat\mu;\hat\pi} - \thetatarget}{2} = O_{\mathbbmss{P}}\left((n+N)^{-1/2}\an^{-1/2}\right)$.

\textbf{Central Limit Theorem.} We now turn to the asymptotic normality of $\hat\theta_{\hat\mu;\hat\pi}$ by showing that the first term on the right-hand side satisfies a form of Central Limit Theorem. This is enough to prove asymptotic normality of the estimator, as the remaining terms are already of order $o_\mathbbmss{P}((n+N)^{-1/2}\an^{-1/2})$.

By Assumption~\ref{ass:reg}, \textbf{c}, we have $\mineig{\varbar} \asymp \an^{-1}$. This implies:
\begin{equation}
\label{eq:eigen2}
    \norm{\varbar^{-1/2}}{2}^2 = \left(\maxeig{\varbar^{-1/2}}\right)^2 = \left(\mineig{\varbar}\right)^{-1} \asymp \an\,.
\end{equation}
By the tail condition \textbf{d} in Assumption~\ref{ass:reg}, we have that
\begin{equation}
    \begin{split}
        &\EE{\norm{\varbar^{-1/2}\influence{\data;\thetatarget;\Bar{\mu};\pibar}}{2}^2 \onea{\norm{\varbar^{-1/2}\influence{\data;\thetatarget;\Bar{\mu};\pibar}}{2} > \varepsilon \sqrt{n+N}}} \\
        &\leq \EE{\norm{\varbar^{-1/2}}{2}^2\norm{\influence{\data;\thetatarget;\Bar{\mu};\pibar}}{2}^2 \onea{\norm{\varbar^{-1/2}}{2}\norm{\influence{\data;\thetatarget;\Bar{\mu};\pibar}}{2} > \varepsilon \sqrt{n+N}}} \\
        &\leq \EE{C\an\norm{\influence{\data;\thetatarget;\Bar{\mu};\pibar}}{2}^2 \onea{\norm{\sqrt{\an}\influence{\data;\thetatarget;\Bar{\mu};\pibar}}{2} > \varepsilon \sqrt{n+N}}} \to 0\quad\text{as}\,n,N\to\infty\,,
    \end{split}
\end{equation}
where the first inequality is due to triangle inequality and the second inequality, for some constant $C>0$, is due to the chain of equalities in Equation~\ref{eq:eigen2}. Therefore the following Lindeberg condition also holds
\begin{equation}
    (n+N)^{-1} \sum_{i=1}^{n+N}\EE{\norm{\varbar^{-1/2}\influence{\data_i;\thetatarget;\Bar{\mu};\pibar}}{2}^2 \onea{\norm{\varbar^{-1/2}\influence{\data_i;\thetatarget;\Bar{\mu};\bar\pi}}{2} > \varepsilon \sqrt{n+N}}} \to 0\quad\text{as}\,n,N\to\infty\,,
\end{equation}
which implies, by Lindeberg-Feller Central Limit Theorem \citep{van2000asymptotic}, that
\begin{equation}
    (n+N)^{-1/2} \varbar^{-1/2} \sum_{i=1}^{n+N} \influence{\data_i;\thetatarget;\Bar{\mu};\Bar\pi} \dto \Normal{0}{I_q}\,.
\end{equation}

The last step is to show that:
\begin{equation}
    \left(\VV{\influence{\data;\thetatarget;\hat\mu;\hat\pi}}\right)^{-1/2}= \varbar^{-1/2} + o_\mathbbmss{P}(\an^{1/2}) \,.
\end{equation}
This can be shown by bounding the difference between the two variance matrices. By expanding the $\mathbbmss{L}_2$ norm and applying the Cauchy-Schwarz inequality, the difference between the second moments is dominated by the cross-term:
\begin{equation}
    \begin{split}
        \norm{\VV{\influence{\data;\thetatarget;\hat\mu;\hat\pi}} - \varbar}{2} &\le \EE{\norm{\influence{\data;\thetatarget;\hat\mu;\hat\pi} - \influence{\data;\thetatarget;\bar\mu;\pibar}}{2}^2} \\
        &+ 2 \sqrt{\EE{\norm{\influence{\data;\thetatarget;\bar\mu;\pibar}}{2}^2}} \sqrt{\EE{\norm{\influence{\data;\thetatarget;\hat\mu;\hat\pi} - \influence{\data;\thetatarget;\bar\mu;\pibar}}{2}^2}} \\
        &\leq o_\mathbbmss{P}(\an^{-1/2})\,,
    \end{split}
\end{equation}
where the last step follows from Assumptions~\ref{ass:reg}, \textbf{a} and \textbf{c}. This means:
\begin{equation}
    \an\ \VV{\influence{\data;\thetatarget;\hat\mu;\hat\pi}} \pto \an \varbar\,.
\end{equation}
Because Assumption~\ref{ass:reg}, \textbf{c}, guarantees that the scaled variance $\an \varbar$ is a positive definite matrix with eigenvalues bounded away from zero, we can apply the continuous mapping theorem to the inverse square root function:
\begin{equation}
    \left(\an \VV{\influence{\data;\thetatarget;\hat\mu;\hat\pi}} \right)^{-1/2} = (\an \varbar)^{-1/2} + o_\mathbbmss{P}(1)\,,
\end{equation}
which delivers the result. In summary, we have
\begin{equation}
    (n+N)^{1/2} \VV{\influence{\data;\thetatarget;\hat\mu;\hat\pi}}^{-1/2} \left(\hat\theta_{\hat\mu;\hat\pi} - \thetatarget\right) =  (n+N)^{-1/2} \varbar^{-1/2} \sum_{i=1}^{n+N} \influence{\data_i;\thetatarget;\Bar{\mu};\pibar} + o_\mathbbmss{P}(1)\,.
\end{equation}
\end{proof}

\subsection{Proof of Corollary~\ref{cor:an_ppi}}
\begin{proof}
    This is a straightforward application of Theorem~\ref{th:ansi_unknownpi}, where $f(X)$ is employed as an additional covariate. 
\end{proof}

\subsection{Proof of Theorem~\ref{th:convolution}}
\begin{proof}
    If the sequence of local experiments $\{\jPtarget_{n,N,h}:\, h\in\RR^q \}$ converges to a Gaussian shift experiment with covariance $\Sigma^\star$, then the regularity condition in Assumption~\ref{ass:eff}, \textbf{c}, implies the multivariate Hájek-Le Cam convolution theorem \citep[Theorem 8.8]{van2000asymptotic}. Therefore, for any estimator $\hat\theta_{n,N}$ regular at rate $\sqrt{(n+N)\an}$, there exists a probability measure $M$ such that
    \begin{equation}
       (n+N)^{1/2}\an^{1/2} \left(\hat\theta_{n,N} - \thetatarget \right) \dto \Normal{0}{\Sigma^\star} * M
    \end{equation}
    under the central law. The burden of the proof is thus to show \textit{local asymptotic normality} (LAN) of the sequence of local experiments. 

    Define $\delta_{n,N,h}(\data) = \left(\sqrt{p^\star_{n,N,h}(\data)} - \sqrt{p^\star_{n,N}(\data)}\right) / \sqrt{p^\star_{n,N}(\data)}$. Then:
    \begin{equation}
        \sqrt{p^\star_{n,N,h}(\data)} = \sqrt{p^\star_{n,N}(\data)} \left( 1 + \delta_{n,N,h}(\data) \right)\,,\quad \frac{p^\star_{n,N,h}(\data)}{p^\star_{n,N}(\data)} =  \left( 1 + \delta_{n,N,h}(\data) \right)^2\,.
    \end{equation}
    Therefore we can write the log-likelihood ratio as:
    \begin{equation}
        \Lambda_{n,N}(h) = \sum_{i=1}^{n+N} \log\frac{p^\star_{n,N,h}}{p^\star_{n,N}}(\data_i) = 2 \sum_{i=1}^{n+N} \log(1 + \delta_{n,N,h} (\data_i))\,.
    \end{equation}
    The goal is to show that the expansion
    \begin{equation}
        \Lambda_{n,N}(h) = \sum_{i=1}^{n+N} \ell_{n,N,h}(\data_i) - \frac{1}{2} (n+N) \EE{\ell_{n,N,h}^2(\data)} + o_{\jPtarget_{n,N}}(1)
    \end{equation}
    holds. By Taylor expansion, we have
     \begin{equation}
        \Lambda_{n,N}(h) = 2 \sum_{i=1}^{n+N} \log(1 + \delta_{n,N,h} (\data_i)) = 2 \sum_{i=1}^{n+N} \delta_{n,N,h} (\data_i) - \sum_{i=1}^{n+N} \delta_{n,N,h}^2 (\data_i) + \sum_{i=1}^{n+N} r(\delta_{n,N,h}) \,,
    \end{equation}
    where for every fixed $\varepsilon\in(0,1/2)$ there is a constant $C_\varepsilon$ such that $|\delta_{n,N,h}| \leq \varepsilon \implies |r(\delta_{n,N,h})| \leq C_\varepsilon |\delta_{n,N,h}|^3$.

    \textbf{The centered linear term equals the score sum.} Let $\zeta_{n,N,h}(\data) = \delta_{n,N,h}(\data) - \ell_{n,N,h}(\data)/2$. Then:
    \begin{equation}
        2 \sum_{i=1}^{n+N} \left( \delta_{n,N,h} (\data_i) - \EE{\delta_{n,N,h} (\data)}\right) = \sum_{i=1}^{n+N} \ell_{n,N,h} (\data_i) + 2 \sum_{i=1}^{n+N} \left(\zeta_{n,N,h}(\data_i) - \EE{\zeta_{n,N,h}(\data)} \right)\,.
    \end{equation}
    The second term is negligible because
    \begin{equation}
    \begin{split}
        \VV{2 \sum_{i=1}^{n+N} \left(\zeta_{n,N,h}(\data_i) - \EE{\zeta_{n,N,h}(\data)} \right)} &= 4(n+N)\VV{\zeta_{n,N,h}(\data)} \\
        &\leq 4(n+N)\EE{\zeta_{n,N,h}^2(\data)} \\
        &\to 0
    \end{split}
    \end{equation}
    by Lemma~\ref{lem:dqm_cons}. Therefore:
    \begin{equation}
        2 \sum_{i=1}^{n+N} \left( \delta_{n,N,h} (\data_i) - \EE{\delta_{n,N,h} (\data)}\right) = \sum_{i=1}^{n+N} \ell_{n,N,h} (\data_i) + o_{\jPtarget_{n,N}}(1)\,.
    \end{equation}
    
    \textbf{The square term is asymptotically deterministic.} Expanding, one gets:
    \begin{equation}
        \delta_{n,N,h}^2 (\data) = \left(\zeta_{n,N,h}(\data) + \frac{\ell_{n,N,h}(\data)}{2}\right)^2 = \zeta_{n,N,h}^2(\data) + \zeta_{n,N,h}(\data) \ell_{n,N,h}(\data) + \frac{\ell_{n,N,h}^2(\data)}{4}\,,
    \end{equation}
    and by summing over $i$:
    \begin{equation}
        \sum_{i=1}^{n+N} \delta_{n,N,h}^2 (\data_i) =  \sum_{i=1}^{n+N} \zeta_{n,N,h}^2(\data_i) +  \sum_{i=1}^{n+N} \zeta_{n,N,h}(\data_i) \ell_{n,N,h}(\data_i)  + \sum_{i=1}^{n+N} \frac{\ell_{n,N,h}^2(\data_i)}{4}\,.
    \end{equation}    
    By Lemma~\ref{lem:dqm_cons}, the first term is $o_{\jPtarget_{n,N}}(1)$ because its expectation is $(n+N)\EE{\zeta_{n,N,h}^2(\data)}\to0$. The middle term is again $o_{\jPtarget_{n,N}}(1)$. In fact, by Cauchy-Schwarz:
    \begin{equation}
        \EE{\left|\sum_{i=1}^{n+N} \zeta_{n,N,h}(\data_i) \ell_{n,N,h}(\data_i) \right|} \leq  \left((n+N)\EE{\zeta_{n,N,h}^2(\data)}\right)^{1/2} \left((n+N)\EE{\ell_{n,N,h}^2(\data)} \right)^{1/2} \to 0\,,
    \end{equation}
    where the convergence holds by Lemma~\ref{lem:dqm_cons} and Lemma~\ref{lem:l2_good}. Hence:
    \begin{equation}
        \sum_{i=1}^{n+N} \delta_{n,N,h}^2 (\data_i) = \sum_{i=1}^{n+N} \frac{\ell_{n,N,h}^2(\data_i)}{4} + o_{\jPtarget_{n,N}}(1)\,.
    \end{equation}
    By Lemma~\ref{lem:conv_to_constant}, $\sum_{i=1}^{n+N} \ell_{n,N,h}^2(\data_i) = (n+N) \EE{\ell_{n,N,h}^2(\data)} + o_{\jPtarget_{n,N}}(1)$. Therefore:
    \begin{equation}
        \sum_{i=1}^{n+N} \delta_{n,N,h}^2 (\data_i) = \frac{(n+N)\EE{\ell_{n,N,h}^2(\data)}}{4} + o_{\jPtarget_{n,N}}(1)\,.
    \end{equation}

    \textbf{The Taylor reminder is negligible.} Fix $\varepsilon\in(0,1/2)$, and define the event
    \begin{equation}
        E_{n,N,\varepsilon} = \left\{ \max_{1\leq i \leq n+N} | \delta_{n,N,h}(\data_i) | \leq \varepsilon \right\}\,.
    \end{equation}
    By Lemma~\ref{lem:delta_control}, $\PP{E_{n,N,\varepsilon}}\to1$. On this event, it holds that
    \begin{equation}
        \sum_{i=1}^{n+N} |r(\delta_{n,N,h})| \leq C_\varepsilon \sum_{i=1}^{n+N} |\delta_{n,N,h}|^3 \leq C_\varepsilon \left(\max_{1\leq i \leq n+N} | \delta_{n,N,h}(\data_i) |\right) \sum_{i=1}^{n+N} \delta_{n,N,h}^2 = o_{\jPtarget_{n,N}}(1)\,.
    \end{equation}

    \textbf{Final CLT.} By combining the previous results, one gets
    \begin{equation}
        \Lambda_{n,N}(h) = \sum_{i=1}^{n+N} \ell_{n,N,h}(\data_i) - \frac{1}{2} (n+N) \EE{\ell_{n,N,h}^2(\data)} + o_{\jPtarget_{n,N}}(1)\,.
    \end{equation}
    Finally, by Theorem~\ref{th:ansi_unknownpi}, we have
    \begin{equation}
        \sum_{i=1}^{n+N} \ell_{n,N,h}(\data_i) \dto \Normal{0}{h^T \left(\Sigma^\star\right)^{-1} h}\,, \quad (n+N) \EE{\ell_{n,N,h}^2(\data)} \to h^T \left(\Sigma^\star\right)^{-1} h\,,
    \end{equation}
    which in turn imply
    \begin{equation}
        \Lambda_{n,N}(h) \dto \Normal{-\frac{1}{2}h^T \left(\Sigma^\star\right)^{-1} h}{h^T \left(\Sigma^\star\right)^{-1} h}\,.
    \end{equation}
\end{proof}

\subsection{Proof of Corollary~\ref{cor:aipw}}
\begin{proof}
    When $\hat\mu$ and $\hat\pi$ converge at rates satisfying Assumption~\ref{ass:reg}, \textbf{b}, to $\bar\mu=\mutarget$ and $\pibar=\pitarget$, respectively, the asymptotic variance of $\ninfluence{\data;\thetatarget;\mutarget;\pitarget}$ is $\Sigma^\star$, and thus the estimator $\hat\theta_{\hat\mu;\hat\pi}$ attains the convolution bound. 
\end{proof}

\section{Additional theoretical results}
\label{supp_sec:highass}
Here, we provide some high-level conditions that imply Assumption~\ref{ass:reg}, \textbf{c}, for the multivariate outcome mean target. We assume that the propensity score $\pitarget$ is known, so that $\hat\pi=\pibar=\pitarget$. We first provide a simple yet useful Lemma.
\begin{lemma}
\label{lem:simple_pi}
    Assume $\hat\pi=\pibar=\pitarget$. Then
    \begin{equation}
        \EE{\left(\frac{R}{\hat\pi(X)} - 1\right)^2\mid X} = \frac{1}{\hat\pi(X)} - 1\,.
    \end{equation}
\end{lemma}
\begin{proof}
The following chain holds:
\begin{equation}
    \begin{split}
        \EE{\left(\frac{R}{\hat\pi(X)} - 1\right)^2\mid X} &= \EE{\frac{R^2 + \hat\pi^2(X) - 2R\hat\pi(X)}{\hat\pi^2(X)} \mid X} \\
        &=\EE{\frac{R}{\hat\pi^2(X)} \mid X} + 1 - 2\EE{\frac{R}{\hat\pi(X)} \mid X} \\
        &=\frac{\EE{R\mid X}}{\hat\pi^2(X)} + 1 - 2\frac{\EE{R \mid X}}{\hat\pi(X)}  \\
        &=\frac{1}{\hat\pi(X)} - 1 \,,
    \end{split}
\end{equation}
where we are using the facts $R^2=R$ and $\hat\pi(X)=\pitarget(X)$.
\end{proof}

\begin{proposition}[High-level conditions for multivariate outcome mean target]
\label{prop:high_level}
    Let $\thetatarget=\EE{Y}$ and $\hat\pi=\pibar=\pitarget$. Assume that, for all $\ell=1,\dots,k$, $\EE{\left(Y_\ell-\mutarget_\ell(X)\right)^2 \mid X} \geq \sigma^2_\ell$, $\EE{\left(Y_\ell-\Bar{\mu}_\ell(X)\right)^2\mid X} \leq \Bar{\sigma}^2_\ell$ and $\VV{Y_\ell}\leq\Bar{\sigma}^2_\ell$. Then $\mineig{\VV{\influence{\data;\thetatarget;\Bar{\mu};\pitarget}}} \asymp \an^{-1}$. 
\end{proposition}

\begin{proof}
    We first derive additional uniform lower bounds for $\EE{\left(Y_\ell-\Bar\mu_\ell(X)\right)^2 \mid X}$ and $\VV{Y_\ell}$. By iterated expectation, we have:
    \begin{equation}
        \begin{split}
            \EE{\left(Y_\ell-\Bar\mu_\ell(X)\right)^2 \mid X} &= \EE{\left(Y_\ell-\mutarget_\ell(X) + \mutarget_\ell(X) - \Bar\mu_\ell(X)\right)^2 \mid X} \\
            &= \EE{\left(Y_\ell-\mutarget_\ell(X)\right)^2 \mid X} + \EE{ \left(\mutarget_\ell(X) - \Bar\mu_\ell(X)\right)^2 \mid X} \\
            &+ 2 \EE{\left(Y_\ell-\mutarget_\ell(X)\right) \left(\mutarget_\ell(X) - \Bar\mu_\ell(X)\right) \mid X} \\
            &= \EE{\left(Y_\ell-\mutarget_\ell(X)\right)^2 \mid X} + \EE{ \left(\mutarget_\ell(X) - \Bar\mu_\ell(X)\right)^2 \mid X} \\
            &\geq \sigma^2_\ell\,.
        \end{split}
    \end{equation}
    Similarly, we can show that:
    \begin{equation}
        \begin{split}
            \VV{Y_\ell} &= \EE{\left(Y_\ell - \EE{Y_\ell}\right)^2} \\
            &= \EE{\EE{\left(Y_\ell-\mutarget_\ell(X)\right)^2 \mid X} + \EE{\left(\mutarget_\ell(X) - \EE{Y_\ell}\right)^2 \mid X} } \\
            &\geq \sigma^2_\ell\,.
        \end{split}
    \end{equation}
    We now show that each diagonal element of $\varbar$ is of the same rate as $\an^{-1}$. In particular, for each $\ell=1,\dots,k$ we have:
    \begin{equation}
    \begin{split}
          \varbar[\ell,\ell] &= \EE{\influence{\data;\thetatarget;\Bar{\mu};\pitarget}_{\ell}^2} \\
          &=\EE{\left(\Bar\mu_\ell(X) + \frac{R\left(Y_\ell - \Bar\mu_\ell(X) \right)}{\pitarget(X)} - \EE{Y_\ell} \right)^2} \\
          &=\EE{\left(\frac{\left(R - \pitarget(X)\right) \left(Y_\ell - \Bar\mu_\ell(X)\right)}{\pitarget(X)} + Y_\ell - \EE{Y_\ell} \right)^2} \\
          &= \EE{\frac{\left(R - \pitarget(X)\right)^2 }{\left(\pitarget(X)\right)^2} \left(Y_\ell - \Bar\mu_\ell(X)\right)^2} + \VV{Y_\ell} \\
          &= \EE{\frac{\left(1 - \pitarget(X)\right)}{\pitarget(X)} \left(Y_\ell - \Bar\mu_\ell(X)\right)^2} + \VV{Y_\ell}\,,
    \end{split}
    \end{equation}
    where we are using Lemma~\ref{lem:simple_pi}. Therefore we have that
    \begin{equation}
        \an\varbar[\ell,\ell] \geq \an \left(\left(\frac{1}{\pitarget(X)} - 1\right)  \sigma^2_\ell +  \sigma^2_\ell \right) = \sigma^2_\ell > 0\,,
    \end{equation}
    and similarly we can show that
    \begin{equation}
        \an\varbar[\ell,\ell] \leq \an \left(\left(\frac{1}{\pitarget(X)} - 1\right)  \Bar{\sigma}^2_\ell +  \Bar{\sigma}^2_\ell \right) = \Bar{\sigma}^2_\ell < \infty\,.
    \end{equation}
    The last two inequalities imply that $\varbar[\ell,\ell] \asymp \an^{-1}$. We can now exploit Cauchy-Schwarz to expand the previous results to off-diagonal elements of $\varbar$. In fact:
    \begin{equation}
        |\varbar[\ell_1,\ell_2]| \leq \ \sqrt{\varbar[\ell_1,\ell_1] \varbar[\ell_2, \ell_2]} \asymp \an^{-1} \,.
    \end{equation}
    Therefore, for all $\ell_1=1,\dots,k$ and $\ell_2=1,\dots,k$, it holds that
    \begin{equation}
        \varbar[\ell_1,\ell_2] \asymp \an^{-1}\,,
    \end{equation}
    which concludes the proof.
\end{proof}

\begin{lemma}[Direct consequence of QMD]
\label{lem:dqm_cons}
Under Assumption~\ref{ass:eff}, \textbf{b}, it holds that 
\begin{equation}
    (n+N)\EE{\left( \delta_{n,N,h}(\data) - \frac{1}{2}\ell_{n,N,h}(\data) \right)^2} \to 0\,.
\end{equation}
\end{lemma}
\begin{proof}
    The claim is just a manipulation of Assumption~\ref{ass:eff} \textbf{b}:
    \begin{equation}
    \begin{split}
        (n+N)\EE{\left( \delta_{n,N,h}(\data) - \frac{1}{2}\ell_{n,N,h}(\data) \right)^2} &= (n+N) \int \left( \frac{\sqrt{p^\star_{n,N,h}(\data)} - \sqrt{p^\star_{n,N}(\data)}}{\sqrt{p^\star_{n,N}(\data)}} - \frac{1}{2}\ell_{n,N,h}(\data) \right)^2 p^\star_{n,N}(\data) \, d\nu_{n,N} \\
        &= (n+N) \int \left( \sqrt{p^\star_{n,N,h}(\data)} - \sqrt{p^\star_{n,N}(\data)} - \frac{1}{2}\ell_{n,N,h}(\data)\sqrt{p^\star_{n,N}(\data)} \right)^2 \, d\nu_{n,N} \\
        &\to 0\,.
    \end{split}
    \end{equation}
\end{proof}

\begin{lemma}
\label{lem:l2_good}
    Under Assumption~\ref{ass:eff}, \textbf{a}, it holds that
    \begin{equation}
        (n+N) \EE{\ell_{n,N,h}^2(\data)} \to h^T \left(\Sigma^\star\right)^{-1} h\,.
    \end{equation}
\end{lemma}
\begin{proof}
Direct evaluation implies:
    \begin{equation}
    \begin{split}
        (n+N)\EE{\ell_{n,N,h}^2(\data)} &= (n+N) \EE{\left((n+N)^{-1/2} \an^{-1/2} h^T \Sigma_{n,N}^{-1} \ninfluence{\data;\thetatarget}\right)^2} \\
        &= \frac{1}{\an} h^T \Sigma_{n,N}^{-1} \VV{\ninfluence{\data;\thetatarget}} \Sigma_{n,N}^{-1} h \\
        &= h^T \left(\an\Sigma_{n,N}\right)^{-1} h \\
        &\to h^T \left(\Sigma^\star\right)^{-1} h \,,
    \end{split}
    \end{equation}
    where the convergence holds under Assumption~\ref{ass:eff}, \textbf{a}.
\end{proof}

\begin{lemma}
    \label{lem:conv_to_constant}
    Under Assumption~\ref{ass:reg}, \textbf{d}, and Assumption~\ref{ass:eff}, \textbf{a}, it holds that
    \begin{equation}
        \sum_{i=1}^{n+N} \ell_{n,N,h}^2(\data_i) - (n+N) \EE{\ell_{n,N,h}^2(\data)} \pto 0\,.
    \end{equation}
\end{lemma}
\begin{proof}
    Let $S_{n,N} =  \sum_{i=1}^{n+N} \left(\ell_{n,N,h}^2(\data_i) -  \EE{\ell_{n,N,h}^2(\data)}\right)$. Set $\varepsilon>0$ and decompose $\ell_{n,N,h}^2(\data_i) = U^{(\varepsilon)}_{n,N} + V^{(\varepsilon)}_{n,N}$ where $ U^{(\varepsilon)}_{n,N} = \ell_{n,N,h}^2 \onea{|\ell_{n,N,h}|\leq \varepsilon}$ and $ V^{(\varepsilon)}_{n,N}=  \ell_{n,N,h}^2 \onea{|\ell_{n,N,h}|> \varepsilon}$. Then 
    \begin{equation}
        S_{n,N} = \sum_{i=1}^{n+N} \left( U^{(\varepsilon)}_{n,N}(\data_i) - \EE{U^{(\varepsilon)}_{n,N}(\data)} \right) + \sum_{i=1}^{n+N} \left( V^{(\varepsilon)}_{n,N}(\data_i) - \EE{V^{(\varepsilon)}_{n,N}(\data)} \right)\,.
    \end{equation}
    For the truncated part, by independence we have
    \begin{equation}
    \begin{split}
         \VV{\sum_{i=1}^{n+N} \left( U^{(\varepsilon)}_{n,N}(\data_i) - \EE{U^{(\varepsilon)}_{n,N}(\data)} \right)} &= (n+N) \VV{U^{(\varepsilon)}_{n,N}(\data)} \\
         &\leq (n+N) \EE{\left(U^{(\varepsilon)}_{n,N}(\data)\right)^2} \\
         &\leq (n+N) \varepsilon^2 \EE{\ell_{n,N,h}^2(\data)}\,.
    \end{split}
    \end{equation}
    By Lemma~\ref{lem:l2_good}, the right-hand side converges to the constant $\varepsilon^2 h^T \left(\Sigma^\star\right)^{-1} h$. Hence, by Chebyshev,
    \begin{equation}
        \sum_{i=1}^{n+N} U^{(\varepsilon)}_{n,N}(\data_i) - \EE{U^{(\varepsilon)}_{n,N}(\data)} = O_{\jPtarget_{n,N}}(\varepsilon)\,.
    \end{equation}
    For the tail part, given that $V^{(\varepsilon)}_{n,N}$ is non-negative, we have
    \begin{equation}
    \begin{split}
        \EE{ \left| \sum_{i=1}^{n+N} \left( V^{(\varepsilon)}_{n,N}(\data_i) - \EE{V^{(\varepsilon)}_{n,N}(\data)} \right) \right|} &\leq 2(n+N) \EE{V^{(\varepsilon)}_{n,N}(\data)} \\
        &\leq 2(n+N) \EE{\ell_{n,N,h}^2 \onea{|\ell_{n,N,h}|> \varepsilon}} \\
        &\to 0\,,
    \end{split}
    \end{equation}
    by Assumption~\ref{ass:reg}, \textbf{d}. Summarizing, $S_{n,N} = O_{\jPtarget_{n,N}}(\varepsilon) + o_{\jPtarget_{n,N}}(1)$. By letting $\varepsilon\to0$, we get $S_{n,N} = o_{\jPtarget_{n,N}}(1)$.
\end{proof}

\begin{lemma}[Control of $\delta_{n,N,h}$]
\label{lem:delta_control}
Under Assumption~\ref{ass:reg}, \textbf{d}, and Assumption~\ref{ass:eff}, \textbf{a} and \textbf{b}, the following hold:
\begin{enumerate}
    \item $\max_{1\leq i \leq n+N} | \delta_{n,N,h}(\data_i) | \pto 0\,;$
    \item $\sum_{i=1}^{n+N} \delta_{n,N,h}^2(\data_i) = O_{\jPtarget_{n,N}}(1)\,.$
\end{enumerate}
\end{lemma}
\begin{proof}
    Define $\zeta_{n,N,h}(\data) = \delta_{n,N,h}(\data) - \ell_{n,N,h}(\data)/2$. Then, for any $\varepsilon>0$, we have
    \begin{equation}
        \PP{\max_{1\leq i \leq n+N} |\zeta_{n,N,h}(\data_i)| > \varepsilon} \leq \frac{(n+N) \EE{\zeta_{n,N,h}^2(\data)}}{\varepsilon^2} \to 0
    \end{equation}
    by Lemma~\ref{lem:dqm_cons}. Moreover, fixed a $\varepsilon>0$, by union bound and Markov, we have
    \begin{equation}
         \PP{\max_{1\leq i \leq n+N} |\ell_{n,N,h}(\data_i)| > \varepsilon} \leq (n+N) \PP{|\ell_{n,N,h}(\data_i)| > \varepsilon} \leq  \frac{(n+N)\EE{\ell_{n,N,h}^2 \onea{|\ell_{n,N,h}|> \varepsilon}}}{\varepsilon^2} \to 0
    \end{equation}
    by Assumption~\ref{ass:reg}, \textbf{d}. These two facts together imply
    \begin{equation}
        \max_{1\leq i \leq n+N} |\delta_{n,N,h}(\data_i)| \leq \frac{1}{2} \max_{1\leq i \leq n+N} |\ell_{n,N,h}(\data_i)| + \max_{1\leq i \leq n+N} |\zeta_{n,N,h}(\data_i)| \pto 0\,.
    \end{equation}
    For the second claim, notice that
    \begin{equation}
    \begin{split}
        \sum_{i=1}^{n+N} \delta_{n,N,h}^2(\data_i) &= (n+N) \EE{\delta_{n,N,h}^2(\data)} \\
        &= (n+N) \EE{\zeta_{n,N,h}^2(\data) + \zeta_{n,N,h}(\data) \ell_{n,N,h}(\data) + \frac{\ell_{n,N,h}^2(\data)}{4}}\,.
    \end{split}
    \end{equation}
    By Lemma~\ref{lem:dqm_cons}, the first term in the expectation is negligible. By Cauchy-Schwarz, the second term is again negligible. In fact:
    \begin{equation}
        |(n+N)\zeta_{n,N,h}(\data) \ell_{n,N,h}(\data)| \leq \left((n+N) \EE{\zeta_{n,N,h}^2(\data)} \right)^{1/2}  \left((n+N) \EE{\ell_{n,N,h}^2(\data)} \right)^{1/2} \to 0
    \end{equation}
    by the previous point and Lemma~\ref{lem:l2_good}. Therefore, by Lemma~\ref{lem:l2_good} again, we get
    \begin{equation}
        \sum_{i=1}^{n+N} \delta_{n,N,h}^2(\data_i) \to \frac{1}{4} h^T \left(\Sigma^\star\right)^{-1} h\,.
    \end{equation}
    Hence $\sum_{i=1}^{n+N} \delta_{n,N,h}^2(\data_i) = O_{\jPtarget_{n,N}}(1)$ by Markov.
\end{proof}

\section{Practical considerations}

\subsection{Examples of decaying models}
\label{supp_sec:dec_models}
In the semi‐supervised setting, decaying overlap means the probability of
seeing a label ($R=1$) tends to zero across the covariate space as $n,N\to\infty$. Here, we illustrate explicit models that generate this behavior. 

\subsubsection{Logistic regression with decaying offset}
A simple example is a logistic selection model with an intercept that grows large and
negative. This problem has received considerable attention: \citet{zhang2023decaying} focuses on logistic regression with decaying offset; \citet{owen2007infinitely} and \citet{wang2020logistic} study logistic regression with diverging intercept. The general model reads:
\begin{equation}
    \pitarget(X) = \text{expit}\left(\log\left(\frac{n}{n+N}\right) + X^T \gamma \right)\,,
\end{equation}
where $\text{expit}(x) = \exp(x) / (1 + \exp(x))$.

\subsubsection{Decaying MCAR}
\citet{zhang2023decaying} present a second model for the decaying propensity score, which consists in a simple decaying sequence independent from the covariates:
\begin{equation}
    \pitarget(X) = \frac{n}{n+N}\,.
\end{equation}
Equivalently, this model is a decaying \textit{missing completely at random} propensity score model.

\subsubsection{Nonparametric model with uniform decay}
Trying to move away from strict parametric models, here we present a third model for the decaying propensity score that decouples the decaying component and the dependence on the covariates. By doing so, the functional form of the propensity score as a function of $X$ is free to be arbitrarily complex, allowing for non-parametric modeling. The model reads as:
\begin{equation}
\label{eq:pinonparam}
    \pitarget(X) = \an \pi^\star_0(X)\,,
\end{equation}
where, by Assumption~\ref{ass:decMAR}, $\EE{\left(\pi^\star_0(X)\right)^{-1}}=1$. This model is fairly general (it is a generalization of the previous two), and captures most of the nuisances of the \ds{} setting.

\subsection{Sufficient and verifiable conditions for Assumption~\ref{ass:eff}}
\label{supp_sec:eff_examples}

Here, we provide concrete, verifiable sufficient conditions for Assumption~\ref{ass:eff} under the canonical example of the outcome mean target $\thetatarget = \EE{Y}$, with
$Y \in \RR^k$.

Throughout this Section we work under Assumption~\ref{ass:decMAR} and we
let $\pitarget(X) = \an \pi^\star_0(X)$, where $\pi^\star_0(X)$ satisfies $\EE{\left(\pi^\star_0(X)\right)^{-1}} = 1$. This is the nonparametric uniform-decay model of Appendix~\ref{supp_sec:dec_models}, which subsumes both the decaying MCAR model ($\pi^\star_0 = 1$) and the logistic model with decaying offset.
We write $\sigma^2_\ell(X) = \EE{\left(Y_\ell-\mutarget_\ell(X)\right)^2 \mid X}$ for the conditional variance of the $\ell$-th outcome coordinate.

Condition~\ref{ass:eff}, \textbf{a}, regards the existence of the limit $\Sigma^\star$. This requires that $\an\Sigma_{n,N}$ converges. From Proposition~\ref{prop:high_level} and Assumption~\ref{ass:reg}, \textbf{c}, we already know that $\Sigma_{n,N} \asymp \an^{-1}$ for each coordinate $\ell$. So $\an\Sigma_{n,N}$ converges to a finite positive limit whenever $\sigma^2_\ell(X) > 0$.

Condition~\ref{ass:eff}, \textbf{b}, is QMD. For the outcome mean, the local perturbation reduces to tilting the outcome distribution by a score proportional to $Y-\thetatarget$. A sufficient verifiable condition is that the outcome $Y$ has a sub-Gaussian or bounded distribution conditionally on $X$, which ensures the QMD remainder in Assumption~\ref{ass:eff}, \textbf{b}, vanishes at the required rate. More explicitly, if $\EE{Y^4\mid X} \leq C < \infty$ (finite fourth moments) then the QMD condition holds. The local shift condition then follows from a standard Taylor expansion of the functional $\theta(\jPtarget_{n,N,h})$ around $\jPtarget_{n,N}$, using the fact that the influence function for the mean is simply $Y - \thetatarget$.

Condition~\ref{ass:eff}, \textbf{c}, is regularity. This is the weakest of the three; it essentially requires that estimators do not have asymptotic distributions that depend on the local drift $h$, which holds for any RAL estimator.

\section{Further simulation results}
\label{supp_sec:sim}

\begin{figure}[h]
    \centering
    \includegraphics[width=\linewidth]{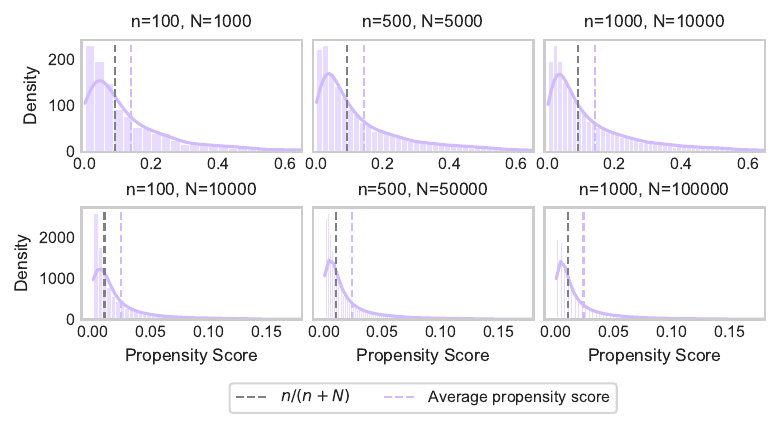}
    \caption{True propensity score distribution across simulation parameters when the missingness mechanism is governed by a logistic model with offset at $n/(n+N)$. Vertical dashed lines represent empirical average (purple) and theoretical one (grey).}
    \label{fig:logistic_ps}
\end{figure}

To compute joint coverage, we evaluate the empirical fraction of times $\Delta$ in which $\thetatarget\in\RR^q$ falls into the asympotic $1-\alpha$ confidence region of $\hat\theta$. In particular, we employ Hotelling one sample $T^2$ test, which rejects the null hypothesis $\hat\theta = \thetatarget$ at $1-\alpha$ level if 
\begin{equation}
    \left(\hat\theta - \thetatarget\right)^T\hat{V}_{n,N}^{-1}  \left(\hat\theta - \thetatarget\right) \geq \frac{q}{n+N-q} F_{q,n+N-q,1-\alpha}\,,
\end{equation}
where $\hat{V}_{n,N}$ is an estimate of the covariance matrix and $F_{q,n+N-q,1-\alpha}$ is the $1-\alpha$ quantile of a $F$-distribution with $q$ and $n+N-q$ degrees of freedom. 

\begin{table}[ht]
\centering
\begin{tabular}{rrlll|cccc}
\toprule
\textbf{n} & \textbf{N} & \textbf{Setting} & $\mu$ \textbf{model} & $\pi$ \textbf{model} & \texttt{AIPW} & \texttt{OR} & \texttt{IPW} & \texttt{Naive} \\
\midrule
100    & 1000   & Decaying MCAR & Constant   & Constant   & 0.282  & 0.282  & 0.282  & 0.282  \\
100    & 1000   & Decaying MCAR & Constant   & Logistic   & 0.233  & 0.282  & 0.216  & 0.282  \\
100    & 1000   & Decaying MCAR & Linear     & Constant   & 0.081  & 0.081  & 0.282  & 0.282  \\
100    & 1000   & Decaying MCAR & Linear     & Logistic   & 0.082  & 0.081  & 0.216  & 0.282  \\
100    & 1000   & Decaying logistic & Constant   & Constant   & 0.910  & 0.910  & 0.910  & 0.910  \\
100    & 1000   & Decaying logistic & Constant   & Logistic   & 0.683  & 0.910  & 0.582  & 0.910  \\
100    & 1000   & Decaying logistic & Linear     & Constant   & 0.081  & 0.081  & 0.910  & 0.910  \\
100    & 1000   & Decaying logistic & Linear     & Logistic   & 0.086  & 0.081  & 0.582  & 0.910  \\
100    & 10000  & Decaying MCAR & Constant   & Constant   & 0.267  & 0.267  & 0.267  & 0.267  \\
100    & 10000  & Decaying MCAR & Constant   & Logistic   & 0.200  & 0.267  & 0.179  & 0.267  \\
100    & 10000  & Decaying MCAR & Linear     & Constant   & 0.029  & 0.029  & 0.267  & 0.267  \\
100    & 10000  & Decaying MCAR & Linear     & Logistic   & 0.032  & 0.029  & 0.179  & 0.267  \\
100    & 10000  & Decaying logistic & Constant   & Constant   & 1.158  & 1.158  & 1.158  & 1.158  \\
100    & 10000  & Decaying logistic & Constant   & Logistic   & 0.585  & 1.158  & 0.492  & 1.158  \\
100    & 10000  & Decaying logistic & Linear     & Constant   & 0.028  & 0.028  & 1.158  & 1.158  \\
100    & 10000  & Decaying logistic & Linear     & Logistic   & 0.033  & 0.028  & 0.492  & 1.158  \\
500    & 5000   & Decaying MCAR & Constant   & Constant   & 0.123  & 0.123  & 0.123  & 0.123  \\
500    & 5000   & Decaying MCAR & Constant   & Logistic   & 0.045  & 0.123  & 0.046  & 0.123  \\
500    & 5000   & Decaying MCAR & Linear     & Constant   & 0.036  & 0.036  & 0.123  & 0.123  \\
500    & 5000   & Decaying MCAR & Linear     & Logistic   & 0.036  & 0.036  & 0.046  & 0.123  \\
500    & 5000   & Decaying logistic & Constant   & Constant   & 0.868  & 0.868  & 0.868  & 0.868  \\
500    & 5000   & Decaying logistic & Constant   & Logistic   & 0.242  & 0.868  & 0.217  & 0.868  \\
500    & 5000   & Decaying logistic & Linear     & Constant   & 0.036  & 0.036  & 0.868  & 0.868  \\
500    & 5000   & Decaying logistic & Linear     & Logistic   & 0.036  & 0.036  & 0.217  & 0.868  \\
500    & 50000  & Decaying MCAR & Constant   & Constant   & 0.120  & 0.120  & 0.120  & 0.120  \\
500    & 50000  & Decaying MCAR & Constant   & Logistic   & 0.032  & 0.120  & 0.031  & 0.120  \\
500    & 50000  & Decaying MCAR & Linear     & Constant   & 0.013  & 0.013  & 0.120  & 0.120  \\
500    & 50000  & Decaying MCAR & Linear     & Logistic   & 0.014  & 0.013  & 0.031  & 0.120  \\
500    & 50000  & Decaying logistic & Constant   & Constant   & 1.131  & 1.131  & 1.131  & 1.131  \\
500    & 50000  & Decaying logistic & Constant   & Logistic   & 0.228  & 1.131  & 0.198  & 1.131  \\
500    & 50000  & Decaying logistic & Linear     & Constant   & 0.013  & 0.014  & 1.131  & 1.131  \\
500    & 50000  & Decaying logistic & Linear     & Logistic   & 0.014  & 0.014  & 0.198  & 1.131  \\
1000   & 10000  & Decaying MCAR & Constant   & Constant   & 0.087  & 0.087  & 0.087  & 0.087  \\
1000   & 10000  & Decaying MCAR & Constant   & Logistic   & 0.029  & 0.087  & 0.030  & 0.087  \\
1000   & 10000  & Decaying MCAR & Linear     & Constant   & 0.026  & 0.026  & 0.087  & 0.087  \\
1000   & 10000  & Decaying MCAR & Linear     & Logistic   & 0.026  & 0.026  & 0.030  & 0.087  \\
1000   & 10000  & Decaying logistic & Constant   & Constant   & 0.866  & 0.866  & 0.866  & 0.866  \\
1000   & 10000  & Decaying logistic & Constant   & Logistic   & 0.157  & 0.866  & 0.142  & 0.866  \\
1000   & 10000  & Decaying logistic & Linear     & Constant   & 0.026  & 0.026  & 0.866  & 0.866  \\
1000   & 10000  & Decaying logistic & Linear     & Logistic   & 0.026  & 0.026  & 0.142  & 0.866  \\
1000   & 100000 & Decaying MCAR & Constant   & Constant   & 0.085  & 0.085  & 0.085  & 0.085  \\
1000   & 100000 & Decaying MCAR & Constant   & Logistic   & 0.017  & 0.085  & 0.016  & 0.085  \\
1000   & 100000 & Decaying MCAR & Linear     & Constant   & 0.009  & 0.009  & 0.085  & 0.085  \\
1000   & 100000 & Decaying MCAR & Linear     & Logistic   & 0.009  & 0.009  & 0.016  & 0.085  \\
1000   & 100000 & Decaying logistic & Constant   & Constant   & 1.133  & 1.133  & 1.133  & 1.133  \\
1000   & 100000 & Decaying logistic & Constant   & Logistic   & 0.158  & 1.133  & 0.137  & 1.133  \\
1000   & 100000 & Decaying logistic & Linear     & Constant   & 0.009  & 0.009  & 1.133  & 1.133  \\
1000   & 100000 & Decaying logistic & Linear     & Logistic   & 0.010  & 0.009  & 0.137  & 1.133  \\
\bottomrule
\end{tabular}
\caption{Multivariate outcome mean simulation results. RMSE is reported for four different estimators. Each row corresponds to a specific configuration of labeled and unlabeled sample sizes, with $n \in \{100, 500, 1000\}$ and $N \in \{10n, 100n\}$, setting, and models for the nuisance functions.} 
\label{tab:sim_results_multivariate_rmse}
\end{table}

\begin{table}[ht]
\centering
\begin{tabular}{rrlll|cccc}
\toprule
\textbf{n} & \textbf{N} & \textbf{Setting} & $\mu$ \textbf{model} & $\pi$ \textbf{model} & \texttt{AIPW} & \texttt{OR} & \texttt{IPW} & \texttt{Naive} \\
\midrule
100    & 1000   & Decaying MCAR & Constant   & Constant   & 0.941  & 0.037  & 0.942  & 0.938  \\
100    & 1000   & Decaying MCAR & Constant   & Logistic   & 1.000  & 0.037  & 0.999  & 0.938  \\
100    & 1000   & Decaying MCAR & Linear     & Constant   & 0.945  & 0.946  & 0.942  & 0.938  \\
100    & 1000   & Decaying MCAR & Linear     & Logistic   & 0.948  & 0.946  & 0.999  & 0.938  \\
100    & 1000   & Decaying logistic & Constant   & Constant   & 0.349  & 0.011  & 0.352  & 0.347  \\
100    & 1000   & Decaying logistic & Constant   & Logistic   & 0.982  & 0.011  & 0.974  & 0.347  \\
100    & 1000   & Decaying logistic & Linear     & Constant   & 0.940  & 0.939  & 0.352  & 0.347  \\
100    & 1000   & Decaying logistic & Linear     & Logistic   & 0.940  & 0.939  & 0.974  & 0.347  \\
100    & 10000  & Decaying MCAR & Constant   & Constant   & 0.953  & 0.012  & 0.952  & 0.948  \\
100    & 10000  & Decaying MCAR & Constant   & Logistic   & 1.000  & 0.012  & 1.000  & 0.948  \\
100    & 10000  & Decaying MCAR & Linear     & Constant   & 0.947  & 0.933  & 0.952  & 0.948  \\
100    & 10000  & Decaying MCAR & Linear     & Logistic   & 0.943  & 0.933  & 1.000  & 0.948  \\
100    & 10000  & Decaying logistic & Constant   & Constant   & 0.221  & 0.004  & 0.222  & 0.220  \\
100    & 10000  & Decaying logistic & Constant   & Logistic   & 0.965  & 0.004  & 0.961  & 0.220  \\
100    & 10000  & Decaying logistic & Linear     & Constant   & 0.940  & 0.931  & 0.222  & 0.220  \\
100    & 10000  & Decaying logistic & Linear     & Logistic   & 0.956  & 0.931  & 0.961  & 0.220  \\
500    & 5000   & Decaying MCAR & Constant   & Constant   & 0.945  & 0.015  & 0.945  & 0.945  \\
500    & 5000   & Decaying MCAR & Constant   & Logistic   & 1.000  & 0.015  & 1.000  & 0.945  \\
500    & 5000   & Decaying MCAR & Linear     & Constant   & 0.957  & 0.957  & 0.945  & 0.945  \\
500    & 5000   & Decaying MCAR & Linear     & Logistic   & 0.957  & 0.957  & 1.000  & 0.945  \\
500    & 5000   & Decaying logistic & Constant   & Constant   & 0.165  & 0.002  & 0.165  & 0.164  \\
500    & 5000   & Decaying logistic & Constant   & Logistic   & 0.976  & 0.002  & 0.970  & 0.164  \\
500    & 5000   & Decaying logistic & Linear     & Constant   & 0.958  & 0.956  & 0.165  & 0.164  \\
500    & 5000   & Decaying logistic & Linear     & Logistic   & 0.959  & 0.956  & 0.970  & 0.164  \\
500    & 50000  & Decaying MCAR & Constant   & Constant   & 0.953  & 0.005  & 0.953  & 0.952  \\
500    & 50000  & Decaying MCAR & Constant   & Logistic   & 1.000  & 0.005  & 1.000  & 0.952  \\
500    & 50000  & Decaying MCAR & Linear     & Constant   & 0.935  & 0.921  & 0.953  & 0.952  \\
500    & 50000  & Decaying MCAR & Linear     & Logistic   & 0.933  & 0.921  & 1.000  & 0.952  \\
500    & 50000  & Decaying logistic & Constant   & Constant   & 0.093  & 0.000  & 0.093  & 0.093  \\
500    & 50000  & Decaying logistic & Constant   & Logistic   & 0.966  & 0.000  & 0.959  & 0.093  \\
500    & 50000  & Decaying logistic & Linear     & Constant   & 0.932  & 0.927  & 0.093  & 0.093  \\
500    & 50000  & Decaying logistic & Linear     & Logistic   & 0.941  & 0.927  & 0.959  & 0.093  \\
1000   & 10000  & Decaying MCAR & Constant   & Constant   & 0.942  & 0.010  & 0.941  & 0.941  \\
1000   & 10000  & Decaying MCAR & Constant   & Logistic   & 1.000  & 0.010  & 1.000  & 0.941  \\
1000   & 10000  & Decaying MCAR & Linear     & Constant   & 0.952  & 0.951  & 0.941  & 0.941  \\
1000   & 10000  & Decaying MCAR & Linear     & Logistic   & 0.954  & 0.951  & 1.000  & 0.941  \\
1000   & 10000  & Decaying logistic & Constant   & Constant   & 0.110  & 0.001  & 0.110  & 0.109  \\
1000   & 10000  & Decaying logistic & Constant   & Logistic   & 0.975  & 0.001  & 0.971  & 0.109  \\
1000   & 10000  & Decaying logistic & Linear     & Constant   & 0.953  & 0.952  & 0.110  & 0.109  \\
1000   & 10000  & Decaying logistic & Linear     & Logistic   & 0.954  & 0.952  & 0.971  & 0.109  \\
1000   & 100000 & Decaying MCAR & Constant   & Constant   & 0.959  & 0.003  & 0.960  & 0.959  \\
1000   & 100000 & Decaying MCAR & Constant   & Logistic   & 1.000  & 0.003  & 1.000  & 0.959  \\
1000   & 100000 & Decaying MCAR & Linear     & Constant   & 0.949  & 0.933  & 0.960  & 0.959  \\
1000   & 100000 & Decaying MCAR & Linear     & Logistic   & 0.950  & 0.933  & 1.000  & 0.959  \\
1000   & 100000 & Decaying logistic & Constant   & Constant   & 0.076  & 0.000  & 0.076  & 0.076  \\
1000   & 100000 & Decaying logistic & Constant   & Logistic   & 0.963  & 0.000  & 0.963  & 0.076  \\
1000   & 100000 & Decaying logistic & Linear     & Constant   & 0.940  & 0.934  & 0.076  & 0.076  \\
1000   & 100000 & Decaying logistic & Linear     & Logistic   & 0.954  & 0.934  & 0.963  & 0.076  \\
\bottomrule
\end{tabular}
\caption{Multivariate outcome mean simulation results. Empirical coverage is reported for four different estimators. Each row corresponds to a specific configuration of labeled and unlabeled sample sizes, with $n \in \{100, 500, 1000\}$ and $N \in \{10n, 100n\}$, setting, and models for the nuisance functions.} 
\label{tab:sim_results_multivariate_coverage}
\end{table}

\begin{table}[ht]
\centering
\begin{tabular}{rrlll|cccc}
\toprule
\textbf{n} & \textbf{N} & \textbf{Setting} & $\mu$ \textbf{model} & $\pi$ \textbf{model} & \texttt{AIPW} & \texttt{OR} & \texttt{IPW} & \texttt{Naive} \\
\midrule
100    & 1000   & Decaying MCAR & Constant   & Constant   & 1.224  & 0.029  & 1.213  & 1.207  \\
100    & 1000   & Decaying MCAR & Constant   & Logistic   & 1.800  & 0.029  & 1.723  & 1.207  \\
100    & 1000   & Decaying MCAR & Linear     & Constant   & 0.366  & 0.364  & 1.213  & 1.207  \\
100    & 1000   & Decaying MCAR & Linear     & Logistic   & 0.369  & 0.364  & 1.723  & 1.207  \\
100    & 1000   & Decaying logistic & Constant   & Constant   & 0.957  & 0.023  & 0.993  & 0.949  \\
100    & 1000   & Decaying logistic & Constant   & Logistic   & 3.135  & 0.023  & 2.727  & 0.949  \\
100    & 1000   & Decaying logistic & Linear     & Constant   & 0.365  & 0.364  & 0.993  & 0.949  \\
100    & 1000   & Decaying logistic & Linear     & Logistic   & 0.382  & 0.364  & 2.727  & 0.949  \\
100    & 10000  & Decaying MCAR & Constant   & Constant   & 1.225  & 0.010  & 1.213  & 1.207  \\
100    & 10000  & Decaying MCAR & Constant   & Logistic   & 1.795  & 0.010  & 1.717  & 1.207  \\
100    & 10000  & Decaying MCAR & Linear     & Constant   & 0.129  & 0.121  & 1.213  & 1.207  \\
100    & 10000  & Decaying MCAR & Linear     & Logistic   & 0.137  & 0.121  & 1.717  & 1.207  \\
100    & 10000  & Decaying logistic & Constant   & Constant   & 0.782  & 0.006  & 0.837  & 0.777  \\
100    & 10000  & Decaying logistic & Constant   & Logistic   & 2.636  & 0.006  & 2.201  & 0.777  \\
100    & 10000  & Decaying logistic & Linear     & Constant   & 0.124  & 0.121  & 0.837  & 0.777  \\
100    & 10000  & Decaying logistic & Linear     & Logistic   & 0.147  & 0.121  & 2.201  & 0.777  \\
500    & 5000   & Decaying MCAR & Constant   & Constant   & 0.537  & 0.006  & 0.536  & 0.536  \\
500    & 5000   & Decaying MCAR & Constant   & Logistic   & 0.576  & 0.006  & 0.571  & 0.536  \\
500    & 5000   & Decaying MCAR & Linear     & Constant   & 0.163  & 0.162  & 0.536  & 0.536  \\
500    & 5000   & Decaying MCAR & Linear     & Logistic   & 0.163  & 0.162  & 0.571  & 0.536  \\
500    & 5000   & Decaying logistic & Constant   & Constant   & 0.426  & 0.005  & 0.443  & 0.425  \\
500    & 5000   & Decaying logistic & Constant   & Logistic   & 1.160  & 0.005  & 1.024  & 0.425  \\
500    & 5000   & Decaying logistic & Linear     & Constant   & 0.162  & 0.162  & 0.443  & 0.425  \\
500    & 5000   & Decaying logistic & Linear     & Logistic   & 0.166  & 0.162  & 1.024  & 0.425  \\
500    & 50000  & Decaying MCAR & Constant   & Constant   & 0.542  & 0.002  & 0.541  & 0.541  \\
500    & 50000  & Decaying MCAR & Constant   & Logistic   & 0.585  & 0.002  & 0.579  & 0.541  \\
500    & 50000  & Decaying MCAR & Linear     & Constant   & 0.057  & 0.054  & 0.541  & 0.541  \\
500    & 50000  & Decaying MCAR & Linear     & Logistic   & 0.057  & 0.054  & 0.579  & 0.541  \\
500    & 50000  & Decaying logistic & Constant   & Constant   & 0.348  & 0.001  & 0.373  & 0.348  \\
500    & 50000  & Decaying logistic & Constant   & Logistic   & 1.073  & 0.001  & 0.912  & 0.348  \\
500    & 50000  & Decaying logistic & Linear     & Constant   & 0.055  & 0.054  & 0.373  & 0.348  \\
500    & 50000  & Decaying logistic & Linear     & Logistic   & 0.062  & 0.054  & 0.912  & 0.348  \\
1000   & 10000  & Decaying MCAR & Constant   & Constant   & 0.379  & 0.003  & 0.379  & 0.379  \\
1000   & 10000  & Decaying MCAR & Constant   & Logistic   & 0.393  & 0.003  & 0.391  & 0.379  \\
1000   & 10000  & Decaying MCAR & Linear     & Constant   & 0.115  & 0.114  & 0.379  & 0.379  \\
1000   & 10000  & Decaying MCAR & Linear     & Logistic   & 0.115  & 0.114  & 0.391  & 0.379  \\
1000   & 10000  & Decaying logistic & Constant   & Constant   & 0.300  & 0.002  & 0.313  & 0.300  \\
1000   & 10000  & Decaying logistic & Constant   & Logistic   & 0.772  & 0.002  & 0.684  & 0.300  \\
1000   & 10000  & Decaying logistic & Linear     & Constant   & 0.115  & 0.114  & 0.313  & 0.300  \\
1000   & 10000  & Decaying logistic & Linear     & Logistic   & 0.116  & 0.114  & 0.684  & 0.300  \\
1000   & 100000 & Decaying MCAR & Constant   & Constant   & 0.383  & 0.001  & 0.382  & 0.382  \\
1000   & 100000 & Decaying MCAR & Constant   & Logistic   & 0.398  & 0.001  & 0.396  & 0.382  \\
1000   & 100000 & Decaying MCAR & Linear     & Constant   & 0.040  & 0.038  & 0.382  & 0.382  \\
1000   & 100000 & Decaying MCAR & Linear     & Logistic   & 0.040  & 0.038  & 0.396  & 0.382  \\
1000   & 100000 & Decaying logistic & Constant   & Constant   & 0.246  & 0.001  & 0.264  & 0.246  \\
1000   & 100000 & Decaying logistic & Constant   & Logistic   & 0.756  & 0.001  & 0.643  & 0.246  \\
1000   & 100000 & Decaying logistic & Linear     & Constant   & 0.039  & 0.038  & 0.264  & 0.246  \\
1000   & 100000 & Decaying logistic & Linear     & Logistic   & 0.043  & 0.038  & 0.643  & 0.246  \\
\bottomrule
\end{tabular}
\caption{Multivariate outcome mean simulation results. Average width of component-wise confidence intervals is reported for four different estimators. Each row corresponds to a specific configuration of labeled and unlabeled sample sizes, with $n \in \{100, 500, 1000\}$ and $N \in \{10n, 100n\}$, setting, and models for the nuisance functions.} 
\label{tab:sim_results_multivariate_width}
\end{table}

\begin{table}[ht]
\centering
\begin{tabular}{rrlll|cccc}
\toprule
\textbf{n} & \textbf{N} & \textbf{Setting} & $\mu$ \textbf{model} & $\pi$ \textbf{model} & \texttt{AIPW} & \texttt{OR} & \texttt{IPW} & \texttt{Naive} \\
\midrule
100    & 1000   & Decaying MCAR & Constant   & Constant   & 0.932  & 0.495  & 0.940  & 0.934  \\
100    & 1000   & Decaying MCAR & Constant   & Logistic   & 1.000  & 0.495  & 1.000  & 0.934  \\
100    & 1000   & Decaying MCAR & Linear     & Constant   & 0.956  & 0.953  & 0.940  & 0.934  \\
100    & 1000   & Decaying MCAR & Linear     & Logistic   & 0.950  & 0.953  & 1.000  & 0.934  \\
100    & 1000   & Decaying logistic & Constant   & Constant   & 0.139  & 0.486  & 0.141  & 0.211  \\
100    & 1000   & Decaying logistic & Constant   & Logistic   & 0.970  & 0.486  & 0.963  & 0.211  \\
100    & 1000   & Decaying logistic & Linear     & Constant   & 0.947  & 0.950  & 0.141  & 0.211  \\
100    & 1000   & Decaying logistic & Linear     & Logistic   & 0.945  & 0.950  & 0.963  & 0.211  \\
100    & 10000  & Decaying MCAR & Constant   & Constant   & 0.945  & 0.486  & 0.949  & 0.948  \\
100    & 10000  & Decaying MCAR & Constant   & Logistic   & 1.000  & 0.486  & 1.000  & 0.948  \\
100    & 10000  & Decaying MCAR & Linear     & Constant   & 0.956  & 0.924  & 0.949  & 0.948  \\
100    & 10000  & Decaying MCAR & Linear     & Logistic   & 0.942  & 0.924  & 1.000  & 0.948  \\
100    & 10000  & Decaying logistic & Constant   & Constant   & 0.065  & 0.478  & 0.066  & 0.134  \\
100    & 10000  & Decaying logistic & Constant   & Logistic   & 0.951  & 0.478  & 0.947  & 0.134  \\
100    & 10000  & Decaying logistic & Linear     & Constant   & 0.940  & 0.928  & 0.066  & 0.134  \\
100    & 10000  & Decaying logistic & Linear     & Logistic   & 0.950  & 0.928  & 0.947  & 0.134  \\
500    & 5000   & Decaying MCAR & Constant   & Constant   & 0.940  & 0.459  & 0.940  & 0.941  \\
500    & 5000   & Decaying MCAR & Constant   & Logistic   & 1.000  & 0.459  & 1.000  & 0.941  \\
500    & 5000   & Decaying MCAR & Linear     & Constant   & 0.957  & 0.955  & 0.940  & 0.941  \\
500    & 5000   & Decaying MCAR & Linear     & Logistic   & 0.958  & 0.955  & 1.000  & 0.941  \\
500    & 5000   & Decaying logistic & Constant   & Constant   & 0.030  & 0.515  & 0.030  & 0.047  \\
500    & 5000   & Decaying logistic & Constant   & Logistic   & 0.977  & 0.515  & 0.970  & 0.047  \\
500    & 5000   & Decaying logistic & Linear     & Constant   & 0.958  & 0.956  & 0.030  & 0.047  \\
500    & 5000   & Decaying logistic & Linear     & Logistic   & 0.961  & 0.956  & 0.970  & 0.047  \\
500    & 50000  & Decaying MCAR & Constant   & Constant   & 0.955  & 0.471  & 0.956  & 0.948  \\
500    & 50000  & Decaying MCAR & Constant   & Logistic   & 1.000  & 0.471  & 1.000  & 0.948  \\
500    & 50000  & Decaying MCAR & Linear     & Constant   & 0.933  & 0.895  & 0.956  & 0.948  \\
500    & 50000  & Decaying MCAR & Linear     & Logistic   & 0.937  & 0.895  & 1.000  & 0.948  \\
500    & 50000  & Decaying logistic & Constant   & Constant   & 0.008  & 0.511  & 0.008  & 0.025  \\
500    & 50000  & Decaying logistic & Constant   & Logistic   & 0.956  & 0.511  & 0.950  & 0.025  \\
500    & 50000  & Decaying logistic & Linear     & Constant   & 0.923  & 0.912  & 0.008  & 0.025  \\
500    & 50000  & Decaying logistic & Linear     & Logistic   & 0.938  & 0.912  & 0.950  & 0.025  \\
1000   & 10000  & Decaying MCAR & Constant   & Constant   & 0.943  & 0.505  & 0.945  & 0.945  \\
1000   & 10000  & Decaying MCAR & Constant   & Logistic   & 1.000  & 0.505  & 1.000  & 0.945  \\
1000   & 10000  & Decaying MCAR & Linear     & Constant   & 0.959  & 0.959  & 0.945  & 0.945  \\
1000   & 10000  & Decaying MCAR & Linear     & Logistic   & 0.961  & 0.959  & 1.000  & 0.945  \\
1000   & 10000  & Decaying logistic & Constant   & Constant   & 0.013  & 0.498  & 0.013  & 0.021  \\
1000   & 10000  & Decaying logistic & Constant   & Logistic   & 0.974  & 0.498  & 0.970  & 0.021  \\
1000   & 10000  & Decaying logistic & Linear     & Constant   & 0.959  & 0.958  & 0.013  & 0.021  \\
1000   & 10000  & Decaying logistic & Linear     & Logistic   & 0.959  & 0.958  & 0.970  & 0.021  \\
1000   & 100000 & Decaying MCAR & Constant   & Constant   & 0.965  & 0.503  & 0.965  & 0.960  \\
1000   & 100000 & Decaying MCAR & Constant   & Logistic   & 1.000  & 0.503  & 1.000  & 0.960  \\
1000   & 100000 & Decaying MCAR & Linear     & Constant   & 0.956  & 0.924  & 0.965  & 0.960  \\
1000   & 100000 & Decaying MCAR & Linear     & Logistic   & 0.958  & 0.924  & 1.000  & 0.960  \\
1000   & 100000 & Decaying logistic & Constant   & Constant   & 0.007  & 0.490  & 0.007  & 0.014  \\
1000   & 100000 & Decaying logistic & Constant   & Logistic   & 0.957  & 0.490  & 0.963  & 0.014  \\
1000   & 100000 & Decaying logistic & Linear     & Constant   & 0.934  & 0.921  & 0.007  & 0.014  \\
1000   & 100000 & Decaying logistic & Linear     & Logistic   & 0.952  & 0.921  & 0.963  & 0.014  \\
\bottomrule
\end{tabular}
\caption{Multivariate outcome mean simulation results. Joint empirical coverage is reported for four different estimators. Each row corresponds to a specific configuration of labeled and unlabeled sample sizes, with $n \in \{100, 500, 1000\}$ and $N \in \{10n, 100n\}$, setting, and models for the nuisance functions.}
\label{tab:sim_results_multivariate_joint_coverage}
\end{table}

\begin{table}[ht]
\centering
\begin{tabular}{rrlll|cccc}
\toprule
\textbf{n} & \textbf{N} & \textbf{Setting} & $\mu$ \textbf{model} & $\pi$ \textbf{model} & \texttt{AIPW} & \texttt{OR} & \texttt{IPW} & \texttt{Naive} \\
\midrule
100    & 1000   & Decaying MCAR & Constant   & Constant   & 0.300  & 0.974  & 0.297  & 0.010  \\
100    & 1000   & Decaying MCAR & Constant   & Logistic   & 0.549  & 0.974  & 0.509  & 0.010  \\
100    & 1000   & Decaying MCAR & Linear     & Constant   & 0.011  & 0.012  & 0.297  & 0.010  \\
100    & 1000   & Decaying MCAR & Linear     & Logistic   & 0.015  & 0.012  & 0.509  & 0.010  \\
100    & 1000   & Decaying logistic & Constant   & Constant   & 0.244  & 0.974  & 0.315  & 0.008  \\
100    & 1000   & Decaying logistic & Constant   & Logistic   & 0.975  & 0.974  & 0.850  & 0.008  \\
100    & 1000   & Decaying logistic & Linear     & Constant   & 0.009  & 0.009  & 0.315  & 0.008  \\
100    & 1000   & Decaying logistic & Linear     & Logistic   & 0.026  & 0.009  & 0.850  & 0.008  \\
100    & 10000  & Decaying MCAR & Constant   & Constant   & 0.313  & 0.976  & 0.310  & 0.010  \\
100    & 10000  & Decaying MCAR & Constant   & Logistic   & 0.554  & 0.976  & 0.515  & 0.010  \\
100    & 10000  & Decaying MCAR & Linear     & Constant   & 0.011  & 0.011  & 0.310  & 0.010  \\
100    & 10000  & Decaying MCAR & Linear     & Logistic   & 0.015  & 0.011  & 0.515  & 0.010  \\
100    & 10000  & Decaying logistic & Constant   & Constant   & 0.204  & 0.976  & 0.462  & 0.006  \\
100    & 10000  & Decaying logistic & Constant   & Logistic   & 0.720  & 0.976  & 0.611  & 0.006  \\
100    & 10000  & Decaying logistic & Linear     & Constant   & 0.007  & 0.007  & 0.462  & 0.006  \\
100    & 10000  & Decaying logistic & Linear     & Logistic   & 0.019  & 0.007  & 0.611  & 0.006  \\
500    & 5000   & Decaying MCAR & Constant   & Constant   & 0.134  & 0.969  & 0.134  & 0.004  \\
500    & 5000   & Decaying MCAR & Constant   & Logistic   & 0.154  & 0.969  & 0.150  & 0.004  \\
500    & 5000   & Decaying MCAR & Linear     & Constant   & 0.004  & 0.005  & 0.134  & 0.004  \\
500    & 5000   & Decaying MCAR & Linear     & Logistic   & 0.005  & 0.005  & 0.150  & 0.004  \\
500    & 5000   & Decaying logistic & Constant   & Constant   & 0.127  & 0.969  & 0.219  & 0.003  \\
500    & 5000   & Decaying logistic & Constant   & Logistic   & 0.305  & 0.969  & 0.274  & 0.003  \\
500    & 5000   & Decaying logistic & Linear     & Constant   & 0.004  & 0.004  & 0.219  & 0.003  \\
500    & 5000   & Decaying logistic & Linear     & Logistic   & 0.008  & 0.004  & 0.274  & 0.003  \\
500    & 50000  & Decaying MCAR & Constant   & Constant   & 0.139  & 0.972  & 0.139  & 0.004  \\
500    & 50000  & Decaying MCAR & Constant   & Logistic   & 0.160  & 0.972  & 0.156  & 0.004  \\
500    & 50000  & Decaying MCAR & Linear     & Constant   & 0.004  & 0.004  & 0.139  & 0.004  \\
500    & 50000  & Decaying MCAR & Linear     & Logistic   & 0.005  & 0.004  & 0.156  & 0.004  \\
500    & 50000  & Decaying logistic & Constant   & Constant   & 0.096  & 0.972  & 0.410  & 0.003  \\
500    & 50000  & Decaying logistic & Constant   & Logistic   & 0.300  & 0.972  & 0.257  & 0.003  \\
500    & 50000  & Decaying logistic & Linear     & Constant   & 0.003  & 0.003  & 0.410  & 0.003  \\
500    & 50000  & Decaying logistic & Linear     & Logistic   & 0.008  & 0.003  & 0.257  & 0.003  \\
1000   & 10000  & Decaying MCAR & Constant   & Constant   & 0.094  & 0.963  & 0.094  & 0.003  \\
1000   & 10000  & Decaying MCAR & Constant   & Logistic   & 0.100  & 0.963  & 0.099  & 0.003  \\
1000   & 10000  & Decaying MCAR & Linear     & Constant   & 0.003  & 0.003  & 0.094  & 0.003  \\
1000   & 10000  & Decaying MCAR & Linear     & Logistic   & 0.003  & 0.003  & 0.099  & 0.003  \\
1000   & 10000  & Decaying logistic & Constant   & Constant   & 0.100  & 0.963  & 0.198  & 0.002  \\
1000   & 10000  & Decaying logistic & Constant   & Logistic   & 0.206  & 0.963  & 0.184  & 0.002  \\
1000   & 10000  & Decaying logistic & Linear     & Constant   & 0.003  & 0.003  & 0.198  & 0.002  \\
1000   & 10000  & Decaying logistic & Linear     & Logistic   & 0.006  & 0.003  & 0.184  & 0.002  \\
1000   & 100000 & Decaying MCAR & Constant   & Constant   & 0.098  & 0.969  & 0.098  & 0.003  \\
1000   & 100000 & Decaying MCAR & Constant   & Logistic   & 0.105  & 0.969  & 0.104  & 0.003  \\
1000   & 100000 & Decaying MCAR & Linear     & Constant   & 0.003  & 0.003  & 0.098  & 0.003  \\
1000   & 100000 & Decaying MCAR & Linear     & Logistic   & 0.003  & 0.003  & 0.104  & 0.003  \\
1000   & 100000 & Decaying logistic & Constant   & Constant   & 0.072  & 0.969  & 0.379  & 0.002  \\
1000   & 100000 & Decaying logistic & Constant   & Logistic   & 0.222  & 0.969  & 0.192  & 0.002  \\
1000   & 100000 & Decaying logistic & Linear     & Constant   & 0.002  & 0.002  & 0.379  & 0.002  \\
1000   & 100000 & Decaying logistic & Linear     & Logistic   & 0.006  & 0.002  & 0.192  & 0.002  \\
\bottomrule
\end{tabular}
\caption{Linear regression coefficients simulation results. RMSE is reported for four different estimators. Each row corresponds to a specific configuration of labeled and unlabeled sample sizes, with $n \in \{100, 500, 1000\}$ and $N \in \{10n, 100n\}$, setting, and models for the nuisance functions.} 
\label{tab:sim_results_lm_rmse}
\end{table} 
    
\begin{table}[ht]
\centering
\begin{tabular}{rrlll|cccc}
\toprule
\textbf{n} & \textbf{N} & \textbf{Setting} & $\mu$ \textbf{model} & $\pi$ \textbf{model} & \texttt{AIPW} & \texttt{OR} & \texttt{IPW} & \texttt{Naive} \\
\midrule
100    & 1000   & Decaying MCAR & Constant   & Constant   & 0.954  & 0.017  & 0.954  & 0.920  \\
100    & 1000   & Decaying MCAR & Constant   & Logistic   & 0.942  & 0.017  & 0.950  & 0.920  \\
100    & 1000   & Decaying MCAR & Linear     & Constant   & 0.934  & 0.162  & 0.954  & 0.920  \\
100    & 1000   & Decaying MCAR & Linear     & Logistic   & 0.959  & 0.162  & 0.950  & 0.920  \\
100    & 1000   & Decaying logistic & Constant   & Constant   & 0.952  & 0.039  & 0.899  & 0.932  \\
100    & 1000   & Decaying logistic & Constant   & Logistic   & 0.950  & 0.039  & 0.956  & 0.932  \\
100    & 1000   & Decaying logistic & Linear     & Constant   & 0.942  & 0.172  & 0.899  & 0.932  \\
100    & 1000   & Decaying logistic & Linear     & Logistic   & 0.948  & 0.172  & 0.956  & 0.932  \\
100    & 10000  & Decaying MCAR & Constant   & Constant   & 0.951  & 0.006  & 0.952  & 0.928  \\
100    & 10000  & Decaying MCAR & Constant   & Logistic   & 0.947  & 0.006  & 0.954  & 0.928  \\
100    & 10000  & Decaying MCAR & Linear     & Constant   & 0.941  & 0.057  & 0.952  & 0.928  \\
100    & 10000  & Decaying MCAR & Linear     & Logistic   & 0.961  & 0.057  & 0.954  & 0.928  \\
100    & 10000  & Decaying logistic & Constant   & Constant   & 0.961  & 0.017  & 0.835  & 0.944  \\
100    & 10000  & Decaying logistic & Constant   & Logistic   & 0.939  & 0.017  & 0.945  & 0.944  \\
100    & 10000  & Decaying logistic & Linear     & Constant   & 0.955  & 0.062  & 0.835  & 0.944  \\
100    & 10000  & Decaying logistic & Linear     & Logistic   & 0.950  & 0.062  & 0.945  & 0.944  \\
500    & 5000   & Decaying MCAR & Constant   & Constant   & 0.956  & 0.004  & 0.956  & 0.946  \\
500    & 5000   & Decaying MCAR & Constant   & Logistic   & 0.950  & 0.004  & 0.953  & 0.946  \\
500    & 5000   & Decaying MCAR & Linear     & Constant   & 0.950  & 0.079  & 0.956  & 0.946  \\
500    & 5000   & Decaying MCAR & Linear     & Logistic   & 0.956  & 0.079  & 0.953  & 0.946  \\
500    & 5000   & Decaying logistic & Constant   & Constant   & 0.917  & 0.017  & 0.833  & 0.944  \\
500    & 5000   & Decaying logistic & Constant   & Logistic   & 0.951  & 0.017  & 0.950  & 0.944  \\
500    & 5000   & Decaying logistic & Linear     & Constant   & 0.948  & 0.089  & 0.833  & 0.944  \\
500    & 5000   & Decaying logistic & Linear     & Logistic   & 0.949  & 0.089  & 0.950  & 0.944  \\
500    & 50000  & Decaying MCAR & Constant   & Constant   & 0.955  & 0.002  & 0.955  & 0.946  \\
500    & 50000  & Decaying MCAR & Constant   & Logistic   & 0.951  & 0.002  & 0.953  & 0.946  \\
500    & 50000  & Decaying MCAR & Linear     & Constant   & 0.948  & 0.024  & 0.955  & 0.946  \\
500    & 50000  & Decaying MCAR & Linear     & Logistic   & 0.955  & 0.024  & 0.953  & 0.946  \\
500    & 50000  & Decaying logistic & Constant   & Constant   & 0.956  & 0.007  & 0.798  & 0.951  \\
500    & 50000  & Decaying logistic & Constant   & Logistic   & 0.943  & 0.007  & 0.944  & 0.951  \\
500    & 50000  & Decaying logistic & Linear     & Constant   & 0.959  & 0.028  & 0.798  & 0.951  \\
500    & 50000  & Decaying logistic & Linear     & Logistic   & 0.950  & 0.028  & 0.944  & 0.951  \\
1000   & 10000  & Decaying MCAR & Constant   & Constant   & 0.957  & 0.002  & 0.957  & 0.946  \\
1000   & 10000  & Decaying MCAR & Constant   & Logistic   & 0.955  & 0.002  & 0.956  & 0.946  \\
1000   & 10000  & Decaying MCAR & Linear     & Constant   & 0.947  & 0.054  & 0.957  & 0.946  \\
1000   & 10000  & Decaying MCAR & Linear     & Logistic   & 0.952  & 0.054  & 0.956  & 0.946  \\
1000   & 10000  & Decaying logistic & Constant   & Constant   & 0.889  & 0.011  & 0.817  & 0.947  \\
1000   & 10000  & Decaying logistic & Constant   & Logistic   & 0.947  & 0.011  & 0.946  & 0.947  \\
1000   & 10000  & Decaying logistic & Linear     & Constant   & 0.950  & 0.056  & 0.817  & 0.947  \\
1000   & 10000  & Decaying logistic & Linear     & Logistic   & 0.953  & 0.056  & 0.946  & 0.947  \\
1000   & 100000 & Decaying MCAR & Constant   & Constant   & 0.957  & 0.001  & 0.958  & 0.953  \\
1000   & 100000 & Decaying MCAR & Constant   & Logistic   & 0.956  & 0.001  & 0.957  & 0.953  \\
1000   & 100000 & Decaying MCAR & Linear     & Constant   & 0.952  & 0.018  & 0.958  & 0.953  \\
1000   & 100000 & Decaying MCAR & Linear     & Logistic   & 0.954  & 0.018  & 0.957  & 0.953  \\
1000   & 100000 & Decaying logistic & Constant   & Constant   & 0.944  & 0.004  & 0.790  & 0.950  \\
1000   & 100000 & Decaying logistic & Constant   & Logistic   & 0.946  & 0.004  & 0.947  & 0.950  \\
1000   & 100000 & Decaying logistic & Linear     & Constant   & 0.958  & 0.021  & 0.790  & 0.950  \\
1000   & 100000 & Decaying logistic & Linear     & Logistic   & 0.955  & 0.021  & 0.947  & 0.950  \\
\bottomrule
\end{tabular}
\caption{Linear regression coefficients simulation results. Component-wise empirical coverage is reported for four different estimators. Each row corresponds to a specific configuration of labeled and unlabeled sample sizes, with $n \in \{100, 500, 1000\}$ and $N \in \{10n, 100n\}$, setting, and models for the nuisance functions.}
\label{tab:sim_results_lm_coverage}
\end{table}

\begin{table}[ht]
\centering
\begin{tabular}{rrlll|cccc}
\toprule
\textbf{n} & \textbf{N} & \textbf{Setting} & $\mu$ \textbf{model} & $\pi$ \textbf{model} & \texttt{AIPW} & \texttt{OR} & \texttt{IPW} & \texttt{Naive} \\
\midrule
100    & 1000   & Decaying MCAR & Constant   & Constant   & 1.247  & 0.046  & 1.237  & 0.039  \\
100    & 1000   & Decaying MCAR & Constant   & Logistic   & 1.879  & 0.046  & 1.803  & 0.039  \\
100    & 1000   & Decaying MCAR & Linear     & Constant   & 0.045  & 0.005  & 1.237  & 0.039  \\
100    & 1000   & Decaying MCAR & Linear     & Logistic   & 0.063  & 0.005  & 1.803  & 0.039  \\
100    & 1000   & Decaying logistic & Constant   & Constant   & 0.991  & 0.101  & 1.014  & 0.030  \\
100    & 1000   & Decaying logistic & Constant   & Logistic   & 3.450  & 0.101  & 3.016  & 0.030  \\
100    & 1000   & Decaying logistic & Linear     & Constant   & 0.035  & 0.004  & 1.014  & 0.030  \\
100    & 1000   & Decaying logistic & Linear     & Logistic   & 0.095  & 0.004  & 3.016  & 0.030  \\
100    & 10000  & Decaying MCAR & Constant   & Constant   & 1.305  & 0.015  & 1.296  & 0.039  \\
100    & 10000  & Decaying MCAR & Constant   & Logistic   & 1.950  & 0.015  & 1.872  & 0.039  \\
100    & 10000  & Decaying MCAR & Linear     & Constant   & 0.045  & 0.002  & 1.296  & 0.039  \\
100    & 10000  & Decaying MCAR & Linear     & Logistic   & 0.065  & 0.002  & 1.872  & 0.039  \\
100    & 10000  & Decaying logistic & Constant   & Constant   & 0.884  & 0.042  & 0.946  & 0.024  \\
100    & 10000  & Decaying logistic & Constant   & Logistic   & 2.653  & 0.042  & 2.275  & 0.024  \\
100    & 10000  & Decaying logistic & Linear     & Constant   & 0.028  & 0.001  & 0.946  & 0.024  \\
100    & 10000  & Decaying logistic & Linear     & Logistic   & 0.069  & 0.001  & 2.275  & 0.024  \\
500    & 5000   & Decaying MCAR & Constant   & Constant   & 0.558  & 0.009  & 0.558  & 0.017  \\
500    & 5000   & Decaying MCAR & Constant   & Logistic   & 0.604  & 0.009  & 0.599  & 0.017  \\
500    & 5000   & Decaying MCAR & Linear     & Constant   & 0.018  & 0.001  & 0.558  & 0.017  \\
500    & 5000   & Decaying MCAR & Linear     & Logistic   & 0.019  & 0.001  & 0.599  & 0.017  \\
500    & 5000   & Decaying logistic & Constant   & Constant   & 0.444  & 0.042  & 0.457  & 0.014  \\
500    & 5000   & Decaying logistic & Constant   & Logistic   & 1.168  & 0.042  & 1.043  & 0.014  \\
500    & 5000   & Decaying logistic & Linear     & Constant   & 0.015  & 0.001  & 0.457  & 0.014  \\
500    & 5000   & Decaying logistic & Linear     & Logistic   & 0.032  & 0.001  & 1.043  & 0.014  \\
500    & 50000  & Decaying MCAR & Constant   & Constant   & 0.584  & 0.003  & 0.583  & 0.018  \\
500    & 50000  & Decaying MCAR & Constant   & Logistic   & 0.634  & 0.003  & 0.628  & 0.018  \\
500    & 50000  & Decaying MCAR & Linear     & Constant   & 0.018  & 0.000  & 0.583  & 0.018  \\
500    & 50000  & Decaying MCAR & Linear     & Logistic   & 0.019  & 0.000  & 0.628  & 0.018  \\
500    & 50000  & Decaying logistic & Constant   & Constant   & 0.394  & 0.018  & 0.424  & 0.011  \\
500    & 50000  & Decaying logistic & Constant   & Logistic   & 1.156  & 0.018  & 0.989  & 0.011  \\
500    & 50000  & Decaying logistic & Linear     & Constant   & 0.012  & 0.000  & 0.424  & 0.011  \\
500    & 50000  & Decaying logistic & Linear     & Logistic   & 0.029  & 0.000  & 0.989  & 0.011  \\
1000   & 10000  & Decaying MCAR & Constant   & Constant   & 0.393  & 0.005  & 0.393  & 0.012  \\
1000   & 10000  & Decaying MCAR & Constant   & Logistic   & 0.409  & 0.005  & 0.407  & 0.012  \\
1000   & 10000  & Decaying MCAR & Linear     & Constant   & 0.013  & 0.000  & 0.393  & 0.012  \\
1000   & 10000  & Decaying MCAR & Linear     & Logistic   & 0.013  & 0.000  & 0.407  & 0.012  \\
1000   & 10000  & Decaying logistic & Constant   & Constant   & 0.311  & 0.029  & 0.321  & 0.010  \\
1000   & 10000  & Decaying logistic & Constant   & Logistic   & 0.800  & 0.029  & 0.712  & 0.010  \\
1000   & 10000  & Decaying logistic & Linear     & Constant   & 0.010  & 0.000  & 0.321  & 0.010  \\
1000   & 10000  & Decaying logistic & Linear     & Logistic   & 0.021  & 0.000  & 0.712  & 0.010  \\
1000   & 100000 & Decaying MCAR & Constant   & Constant   & 0.412  & 0.001  & 0.411  & 0.012  \\
1000   & 100000 & Decaying MCAR & Constant   & Logistic   & 0.429  & 0.001  & 0.427  & 0.012  \\
1000   & 100000 & Decaying MCAR & Linear     & Constant   & 0.013  & 0.000  & 0.411  & 0.012  \\
1000   & 100000 & Decaying MCAR & Linear     & Logistic   & 0.013  & 0.000  & 0.427  & 0.012  \\
1000   & 100000 & Decaying logistic & Constant   & Constant   & 0.278  & 0.012  & 0.297  & 0.008  \\
1000   & 100000 & Decaying logistic & Constant   & Logistic   & 0.864  & 0.012  & 0.745  & 0.008  \\
1000   & 100000 & Decaying logistic & Linear     & Constant   & 0.009  & 0.000  & 0.297  & 0.008  \\
1000   & 100000 & Decaying logistic & Linear     & Logistic   & 0.021  & 0.000  & 0.745  & 0.008  \\
\bottomrule
\end{tabular}
\caption{Linear regression coefficients simulation results. Average width of component-wise confidence intervals is reported for four different estimators. Each row corresponds to a specific configuration of labeled and unlabeled sample sizes, with $n \in \{100, 500, 1000\}$ and $N \in \{10n, 100n\}$, setting, and models for the nuisance functions.}
\label{tab:sim_results_lm_width}
\end{table}

\begin{table}[ht]
\centering
\begin{tabular}{rrlll|cccc}
\toprule
\textbf{n} & \textbf{N} & \textbf{Setting} & $\mu$ \textbf{model} & $\pi$ \textbf{model} & \texttt{AIPW} & \texttt{OR} & \texttt{IPW} & \texttt{Naive} \\
\midrule
100    & 1000   & Decaying MCAR & Constant   & Constant   & 0.828  & 0.000  & 0.827  & 0.853  \\
100    & 1000   & Decaying MCAR & Constant   & Logistic   & 0.910  & 0.000  & 0.910  & 0.853  \\
100    & 1000   & Decaying MCAR & Linear     & Constant   & 0.736  & 0.000  & 0.827  & 0.853  \\
100    & 1000   & Decaying MCAR & Linear     & Logistic   & 0.835  & 0.000  & 0.910  & 0.853  \\
100    & 1000   & Decaying logistic & Constant   & Constant   & 0.864  & 0.000  & 0.680  & 0.993  \\
100    & 1000   & Decaying logistic & Constant   & Logistic   & 0.521  & 0.000  & 0.600  & 0.993  \\
100    & 1000   & Decaying logistic & Linear     & Constant   & 0.817  & 0.000  & 0.680  & 0.993  \\
100    & 1000   & Decaying logistic & Linear     & Logistic   & 0.722  & 0.000  & 0.600  & 0.993  \\
100    & 10000  & Decaying MCAR & Constant   & Constant   & 0.762  & 0.000  & 0.758  & 0.849  \\
100    & 10000  & Decaying MCAR & Constant   & Logistic   & 0.864  & 0.000  & 0.876  & 0.849  \\
100    & 10000  & Decaying MCAR & Linear     & Constant   & 0.743  & 0.000  & 0.758  & 0.849  \\
100    & 10000  & Decaying MCAR & Linear     & Logistic   & 0.847  & 0.000  & 0.876  & 0.849  \\
100    & 10000  & Decaying logistic & Constant   & Constant   & 0.902  & 0.000  & 0.404  & 1.000  \\
100    & 10000  & Decaying logistic & Constant   & Logistic   & 0.350  & 0.000  & 0.450  & 1.000  \\
100    & 10000  & Decaying logistic & Linear     & Constant   & 0.880  & 0.000  & 0.404  & 1.000  \\
100    & 10000  & Decaying logistic & Linear     & Logistic   & 0.717  & 0.000  & 0.450  & 1.000  \\
500    & 5000   & Decaying MCAR & Constant   & Constant   & 0.934  & 0.000  & 0.934  & 0.936  \\
500    & 5000   & Decaying MCAR & Constant   & Logistic   & 0.942  & 0.000  & 0.946  & 0.936  \\
500    & 5000   & Decaying MCAR & Linear     & Constant   & 0.931  & 0.000  & 0.934  & 0.936  \\
500    & 5000   & Decaying MCAR & Linear     & Logistic   & 0.941  & 0.000  & 0.946  & 0.936  \\
500    & 5000   & Decaying logistic & Constant   & Constant   & 0.797  & 0.000  & 0.372  & 0.997  \\
500    & 5000   & Decaying logistic & Constant   & Logistic   & 0.612  & 0.000  & 0.664  & 0.997  \\
500    & 5000   & Decaying logistic & Linear     & Constant   & 0.922  & 0.000  & 0.372  & 0.997  \\
500    & 5000   & Decaying logistic & Linear     & Logistic   & 0.866  & 0.000  & 0.664  & 0.997  \\
500    & 50000  & Decaying MCAR & Constant   & Constant   & 0.907  & 0.000  & 0.901  & 0.926  \\
500    & 50000  & Decaying MCAR & Constant   & Logistic   & 0.939  & 0.000  & 0.939  & 0.926  \\
500    & 50000  & Decaying MCAR & Linear     & Constant   & 0.912  & 0.000  & 0.901  & 0.926  \\
500    & 50000  & Decaying MCAR & Linear     & Logistic   & 0.926  & 0.000  & 0.939  & 0.926  \\
500    & 50000  & Decaying logistic & Constant   & Constant   & 0.938  & 0.000  & 0.194  & 1.000  \\
500    & 50000  & Decaying logistic & Constant   & Logistic   & 0.580  & 0.000  & 0.620  & 1.000  \\
500    & 50000  & Decaying logistic & Linear     & Constant   & 0.949  & 0.000  & 0.194  & 1.000  \\
500    & 50000  & Decaying logistic & Linear     & Logistic   & 0.894  & 0.000  & 0.620  & 1.000  \\
1000   & 10000  & Decaying MCAR & Constant   & Constant   & 0.941  & 0.000  & 0.938  & 0.922  \\
1000   & 10000  & Decaying MCAR & Constant   & Logistic   & 0.946  & 0.000  & 0.948  & 0.922  \\
1000   & 10000  & Decaying MCAR & Linear     & Constant   & 0.920  & 0.000  & 0.938  & 0.922  \\
1000   & 10000  & Decaying MCAR & Linear     & Logistic   & 0.934  & 0.000  & 0.948  & 0.922  \\
1000   & 10000  & Decaying logistic & Constant   & Constant   & 0.670  & 0.000  & 0.294  & 0.997  \\
1000   & 10000  & Decaying logistic & Constant   & Logistic   & 0.703  & 0.000  & 0.735  & 0.997  \\
1000   & 10000  & Decaying logistic & Linear     & Constant   & 0.939  & 0.000  & 0.294  & 0.997  \\
1000   & 10000  & Decaying logistic & Linear     & Logistic   & 0.931  & 0.000  & 0.735  & 0.997  \\
1000   & 100000 & Decaying MCAR & Constant   & Constant   & 0.943  & 0.000  & 0.943  & 0.949  \\
1000   & 100000 & Decaying MCAR & Constant   & Logistic   & 0.950  & 0.000  & 0.953  & 0.949  \\
1000   & 100000 & Decaying MCAR & Linear     & Constant   & 0.932  & 0.000  & 0.943  & 0.949  \\
1000   & 100000 & Decaying MCAR & Linear     & Logistic   & 0.948  & 0.000  & 0.953  & 0.949  \\
1000   & 100000 & Decaying logistic & Constant   & Constant   & 0.918  & 0.000  & 0.139  & 1.000  \\
1000   & 100000 & Decaying logistic & Constant   & Logistic   & 0.675  & 0.000  & 0.694  & 1.000  \\
1000   & 100000 & Decaying logistic & Linear     & Constant   & 0.946  & 0.000  & 0.139  & 1.000  \\
1000   & 100000 & Decaying logistic & Linear     & Logistic   & 0.948  & 0.000  & 0.694  & 1.000  \\
\bottomrule
\end{tabular}
\caption{Linear regression coefficients simulation results. Joint empirical coverage is reported for four different estimators. Each row corresponds to a specific configuration of labeled and unlabeled sample sizes, with $n \in \{100, 500, 1000\}$ and $N \in \{10n, 100n\}$, setting, and models for the nuisance functions. The significant undercoverage in small samples is due to the fact that the parameter we are targeting here is 10-dimensional.}
\label{tab:sim_results_lm_joint_coverage}
\end{table}

\clearpage

\section{Further application details}

\begin{figure}[ht!]
    \centering
    \includegraphics[width=0.7\linewidth]{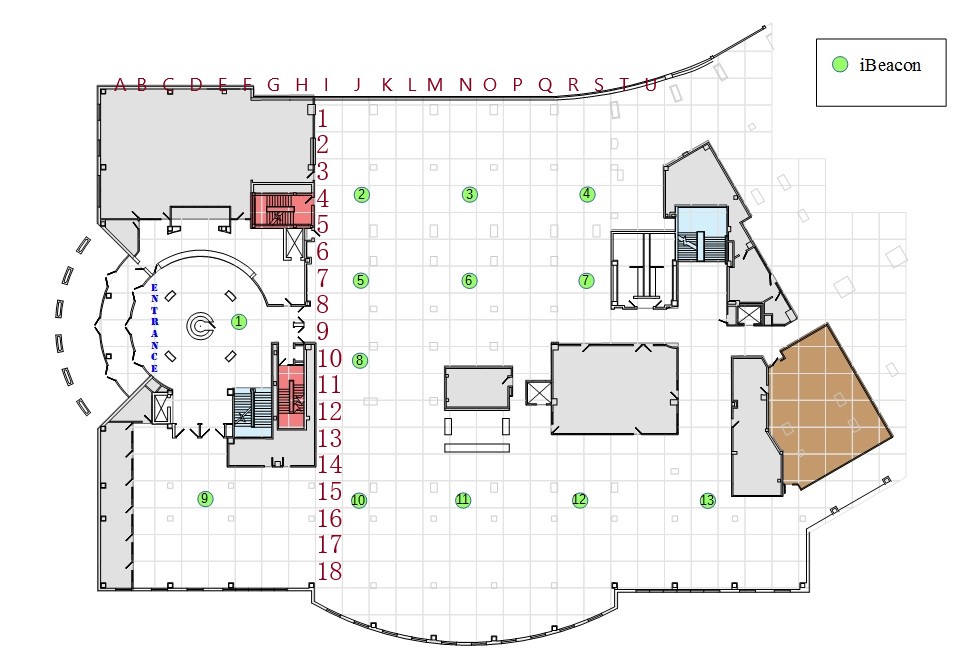}
    \caption{Layout of library room in \texttt{BLE-RSSI} study. Retrieved from \citet{mohammadi2017semisupervised}.}
    \label{fig:room}
\end{figure}


\subsection{\texttt{METABRIC} study}
We further apply our approach to the \texttt{METABRIC} dataset \citep{curtis2012genomic, pereira2016somatic}, one of the largest open-access breast cancer cohorts with gene expression profiles, clinical covariates, and long-term survival follow-up. \citet{cheng2021integrating} and \citet{hsu2023learning} identified 20 biomarkers (reported in Supplementary Table~\ref{tab:app}) as potentially predictive of 5-year disease-specific survival (DSS) status, and defined two prognosis classes accordingly: the poor prognosis class includes patients who died of breast cancer within five years of diagnosis, while the good prognosis class includes those who died of breast cancer after five years.

\begin{table}[ht!]
    \centering
    \begin{tabularx}{\linewidth}{lX}
    \toprule
    \textbf{Type} & \textbf{Covariates} \\
    \midrule
    Biomarkers & ESR1, PGR, ERBB2, MKI67, PLAU, ELAVL1, EGFR, BTRC, FBXO6, SHMT2, KRAS, SRPK2, YWHAQ, PDHA1, EWSR1, ZDHHC17, ENO1, DBN1, PLK1, GSK3B \\
    Clinical features & Age, menopausal state, tumor size, radiotherapy, chemotherapy, hormone therapy, neoplasm histologic grade, cellularity, surgery-breast conserving, surgery-mastectomy \\
    \bottomrule
    \end{tabularx}
    \caption{Variables employed in the \texttt{METABRIC} application as described by \citet{hsu2023learning}.}
    \label{tab:app}
\end{table}

Here, we ask a related but complementary question. Instead of discretizing survival outcomes into prognosis classes, we treat survival time (in months) as a continuous response variable. Then, we aim to estimate the effect of the selected biomarkers on survival among patients who died of breast cancer. Following the labeling conventions of \citet{cheng2021integrating, hsu2023learning}, we treat censored patients -- either still alive or deceased from other causes -- as unlabeled. The resulting dataset includes $n = 611$ labeled observations and $N = 1231$ unlabeled observations.

We frame this task as a coefficient estimation problem in linear regression. Let $Y_i$ denote the survival time (in months) since initial diagnosis for individual $i$; $R_{i} \in \{0, 1\}$ indicate whether the individual died of breast cancer ($R_{i} = 1$) or not; $X_i\in\RR^p$ represent the vector of expression values for $p = 20$ biomarkers of interest; and $W_i\in\RR^s$, with $s = 10$, denote additional clinical features (also described in Supplementary Table~\ref{tab:app}) used to aid estimation of nuisance components. The observed data are $\data_i = (X_i, W_i, R_{i}, R_{i} Y_i)$, for $i=1,\dots,n+N$. We employ both the naive estimator, which only analyzes labeled data, and the \texttt{AIPW} estimator in Equation~\ref{eq:ols}. For the latter, the outcome regression model $\hat\mu$ is estimated using random forests \citep{breiman2001random}, with both biomarker and clinical covariates included as predictors. The propensity score $\hat\pi$ is estimated via logistic regression, again leveraging both sets of covariates. All nuisance functions are estimated using cross-fitting with $J = 5$ folds to ensure valid inference. 

\begin{figure}[t]
    \centering
    \includegraphics[width=\linewidth]{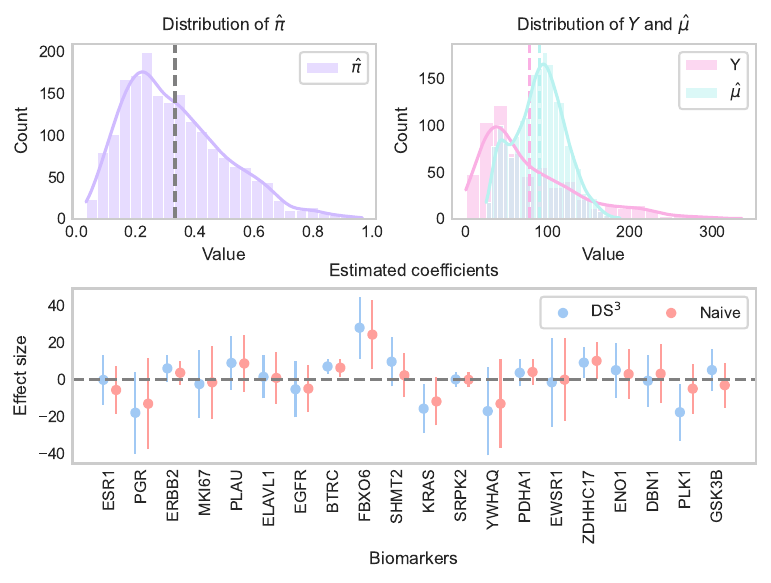}
    \caption{\texttt{METABRIC} application results. Left panel: estimated propensity scores across observations, along with the proportion of labeled data $n / (n + N)$, shown as a vertical dashed line. The heterogeneity in the estimated propensity scores may suggest a missing-at-random (MAR) labeling mechanism. Right panel: marginal distribution of the observed survival times $Y$ (in months) compared to the distribution of predicted survival times from the estimated outcome model using both biomarkers and clinical covariates. Vertical lines indicate the means of each distribution. The discrepancy between observed and predicted outcomes further suggests a MAR mechanism. Bottom panel: estimated linear coefficients from our semi-supervised approach and the naive estimator based solely on labeled data. Confidence intervals are at $\alpha=0.1$ significance level.}
    \label{fig:app}
\end{figure}

While our approach casts the problem as one of linear regression under distribution shift, we acknowledge that alternative methodological frameworks could also be applied. In particular, since the outcome of interest is survival time, this setting naturally lends itself to tools from survival analysis, such as Cox proportional hazards models \citep{cox1972regression}, which are designed to explicitly account for censoring and time-to-event structure. Moreover, our focus on a fixed, literature-informed set of biomarkers connects to the field of inference after feature selection \citep{berk2013valid,kuchibhotla2022post}, which provides tools for valid inference after variable selection. Although we do not pursue these directions here, our method complements these perspectives by offering a flexible and semiparametrically grounded approach to semi-supervised inference under distribution shift -- especially in regimes where labeling is both biased and sparse.

Results are shown in Supplementary Figure~\ref{fig:app}. At the significance level $\alpha = 0.1$, the naive estimator identifies three biomarkers -- \textit{BTRC}, \textit{FBXO6}, and \textit{ZDHHC17} -- as significantly associated with survival outcomes. In contrast, our semi-supervised approach identifies two additional significant biomarkers: \textit{KRAS} and \textit{PLK1}. Prior studies provide biological support for these findings. Specifically, \citet{kim2015activation, tokumaru2020kras} report that elevated \textit{KRAS} expression is linked to poorer survival outcomes, while \citet{garlapati2023plk1, ueda2019therapeutic} associate higher \textit{PLK1} levels with shorter survival in breast cancer patients. 

We note that, because survival times are subject to right-censoring, the missing-at-random (MAR) assumption does not hold: patients with longer survival times are more likely to be censored. However, many patients in \texttt{METABRIC} entered the study at different times, leading to censoring across a range of survival times, not only after five years. Consequently, although the MAR assumption is imperfect, its violation may be moderate in practice. 
Therefore, to investigate this issue, we repeat our analysis using only those censored observations corresponding to deaths from competing causes, and report results in Supplementary Figure~\ref{fig:app_competing}. The results using only competing-risk censored data are qualitatively similar to those obtained using all censored observations, supporting the robustness of our approach despite the potential violation of MAR.

\begin{figure}
    \centering
    \includegraphics[width=\linewidth]{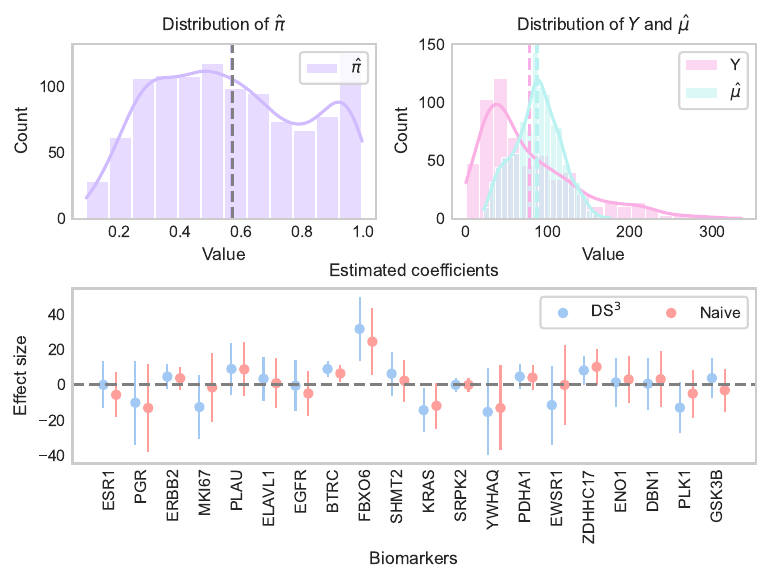}
    \caption{\texttt{METABRIC} application using only competing risk censored data as unlabeled data.}
    \label{fig:app_competing}
\end{figure}

Summarizing, our analysis supports adverse roles for \textit{KRAS} and \textit{PLK1} in breast cancer. These findings underscore the importance of adjusting for distribution shift in semi-supervised survival analyses and demonstrate the practical utility of our proposed approach in real-world biomedical applications.

\end{document}

%% file: preamble.tex
\usepackage[T1]{fontenc}

\usepackage[bitstream-charter]{mathdesign}
\usepackage{bbm}
\usepackage{amsmath}
\usepackage[scaled=0.92]{PTSans}
\usepackage{csquotes}

\usepackage[
  paper  = letterpaper,
  left   = 1.3in,
  right  = 1.3in,
  top    = 1.0in,
  bottom = 1.0in,
  ]{geometry}

\usepackage[dvipsnames]{xcolor}
\definecolor{shadecolor}{gray}{0.9}

\usepackage[final,expansion=alltext]{microtype}
\usepackage[parfill]{parskip}
\usepackage[english]{babel}

\usepackage{graphicx}
\usepackage[labelfont=bf]{caption}
\usepackage[format=hang]{subcaption}

\usepackage{booktabs}
\usepackage{tabularx}

\usepackage{authblk}

\usepackage[round, sort]{natbib}

\usepackage[colorlinks,linktoc=all]{hyperref}
\usepackage[all]{hypcap}
\hypersetup{citecolor=Violet}
\hypersetup{linkcolor=black}
\hypersetup{urlcolor=MidnightBlue}

\usepackage{tcolorbox}

\usepackage{listings}

\usepackage{amsthm}
\newtheorem{theorem}{Theorem}[section]
\newtheorem{lemma}[theorem]{Lemma}
\newtheorem{proposition}[theorem]{Proposition}
\newtheorem{corollary}[theorem]{Corollary}

\theoremstyle{definition}

\newtheorem{remark}[theorem]{Remark}
\newtheorem{assumption}[theorem]{Assumption}

\usepackage{enumitem}

\usepackage{tikz} 

\usepackage{pifont}
\newcommand{\cmark}{\ding{51}}
\newcommand{\xmark}{\ding{55}}

\newcommand{\RR}{\mathbbmss{R}} 
\newcommand{\EE}[1]{\mathbbmss{E}\left[#1\right]} 
\newcommand{\PP}[1]{\mathbbmss{P}\left[#1\right]} 
\newcommand{\VV}[1]{\mathbbmss{V}\left[#1\right]} 
\newcommand{\Normal}[2]{\mathcal{N}\left(#1,#2\right)} 
\newcommand{\pto}{\overset{p}{\rightarrow}} 
\newcommand{\dto}{\leadsto} 
\newcommand{\Lto}[1]{\overset{\mathbbmss{L}^{#1}}{\rightarrow}} 
\newcommand{\norm}[2]{\left\lVert#1\right\rVert_{#2}} 

\newcommand{\finfluence}[1]{\varphi^{F}\left(#1\right)}
\newcommand{\orinfluence}[1]{\varphi^\texttt{OR}\left(#1\right)}
\newcommand{\ipwinfluence}[1]{\varphi^\texttt{IPW}\left(#1\right)}

\newcommand{\influence}[1]{\varphi\left(#1\right)}
\newcommand{\ninfluence}[1]{\varphi_{n,N}\left(#1\right)}
\newcommand{\ppiinfluence}[1]{\varphi^\texttt{PPI}\left(#1\right)}

\newcommand{\data}{\mathcal{D}}

\newcommand{\onea}[1]{\mathbbmss1\left\{#1\right\}} 

\newcommand{\thetatarget}{\theta^\star}

\newcommand{\pitarget}{\pi^\star_{n,N}}
\newcommand{\pibar}{\Bar{\pi}_{n,N}}
\newcommand{\mutarget}{\mu^\star}

\newcommand{\thetaf}{\hat{\theta}^F}

\newcommand{\MSE}[1]{\normalfont \text{RMSE}\left[#1\right]}

\newcommand{\jPtarget}{\mathbbmss{P}^\star}
\newcommand{\cPtarget}{\mathbbmss{P}^\star_{Y|X}}
\newcommand{\mPtarget}{\mathbbmss{P}^\star_{X}}
\newcommand{\cPRtarget}{\mathbbmss{P}^\star_{R|X, n,N}}

\newcommand{\varbar}{\Bar{\mathbbmss{V}}_{\thetatarget}}
\newcommand{\varstar}{\Sigma^\star}

\newcommand{\shortfif}{\varphi^F}

\newcommand{\an}{a_{n,N}}

\newcommand{\argmin}{\text{argmin}}

\newcommand{\mineig}[1]{\lambda_\text{min}\left(#1\right)}
\newcommand{\maxeig}[1]{\lambda_\text{max}\left(#1\right)}

\newcommand{\indep}{\perp \!\!\!\! \perp}

\newcommand{\ds}{\texttt{D\textsuperscript{2}S\textsuperscript{3}}}